\documentclass[11pt]{amsart}
\usepackage{amsmath}
\usepackage{tikz}

\theoremstyle{plain}
\newtheorem{theorem}{Theorem}[section]
\newtheorem{lemma}[theorem]{Lemma}

\theoremstyle{definition}
\newtheorem{remark}[theorem]{Remark}
\newtheorem{definition}[theorem]{Definition}

\numberwithin{equation}{section}

\def\be{\begin{equation}}
\def\ee{\end{equation}}

\begin{document}

\title[Asymptotics and Convergence]
{Asymptotics and Convergence\\ for the Complex Monge-Amp\`ere Equation}

\author[Qing Han]{Qing Han}
\address{Department of Mathematics\\
University of Notre Dame\\
Notre Dame, IN 46556, USA} \email{qhan@nd.edu}

\author[Xumin Jiang]{Xumin Jiang}
\address{Department of Mathematics\\
Fordham University\\
Bronx, NY 10458, USA}  \email{xjiang77@fordham.edu}

\begin{abstract}
We study the asymptotics of complete K\"ahler-Einstein metrics on strictly pseudoconvex domains in $\mathbb{C}^n$  and derive a convergence theorem for solutions to the corresponding Monge-Amp\`ere equation. 
If only a portion of the boundary is  analytic, the solutions
satisfy Gevrey type estimates for tangential derivatives. A counterexample for the model linearized equation suggests that there is no local convergence theorem for the complex Monge-Amp\`ere equation. 
\end{abstract}

\maketitle

\section{Introduction}\label{sec-Intro}
Let $\Omega \subseteq \mathbb{C}^n$ ($n\geq 2$) be a $C^7$, bounded, 
and strictly pseudoconvex domain. 
Consider 
\begin{align}\begin{split}\label{eq-KE}
\det(w_{i \bar{j}})&= e^{(n+1)w} \quad \text{ in } \Omega,\\
w_{i \bar{j}}&>0 \quad\text{ in }  \Omega,\\
w&=\infty \quad\text{ on }  \partial \Omega.
\end{split}
\end{align}
Then, $w_{i\bar{j}}d z^i d z^{\bar{j}}$ defines a K\"ahler-Einstein metric on $\Omega$.  The function $v=e^{-w}$ is zero on $\partial \Omega$ and satisfies the following equation of Fefferman \cite{Fefferman1976}:
\begin{equation*}
J(v)=(-1)^n \det \left(
\begin{array}{cc}
v & v_{\bar k}\\
v_j & v_{j \bar k}
\end{array}
\right)=1.
\end{equation*}
Fefferman  computed the formal power series expansion of $v$ on $\partial \Omega$.
Let $\rho$ be a  strictly plurisubharmonic defining function  of $\Omega$ and  consider a 
reference metric given by 
\begin{align}\label{eq-rho}
-(\log (-\rho))_{i\bar{j}}d z^i d z^{\bar{j}}.
\end{align}
Cheng and Yau \cite{ChengYau1980CPAM} proved that   \eqref{eq-KE}
admits a  solution $w$ of the form 
\begin{align}\label{eq-u}
w=-\log(-\rho)+u,
\end{align}
for some $u\in C^{n, \frac{1}{2}-\delta}(\bar\Omega)$,  for any $\delta>0$. In addition, they derived the uniqueness of solutions to \eqref{eq-KE}.
Lee and Melrose \cite{LeeMelrose} proved that $u$ has a logarithmic  expansion near $\partial\Omega$. More precisely, assuming that $\Omega$ is smooth, there are functions $a_j \in C^\infty (\bar \Omega)$ such that, for any $N\ge 0$, 
\begin{align}\label{eq-expan}
u - \sum_{i=0}^N a_j \rho^{(n+1)j}(\log (-\rho))^j \in C^{(n+1)N-1}(\bar \Omega)
\end{align}
vanishes to order $(n+1)N-1$ at the boundary.
In general, $u$ is not $C^{n+1}(\bar\Omega)$, due to the presence of the 
logarithmic factors. 

Our aim in this paper is to study the convergence of series in \eqref{eq-expan} 
or, more precisely, 
the analyticity of $u$ with respect to $z, \rho^{n+1} \log (-\rho)$ near $\partial \Omega$.
We first state a local regularity result.

\begin{theorem}\label{thm-Main1}
Assume that  $\Omega \subseteq \mathbb{C}^n$ is a $C^7$ bounded strictly pseudoconvex domain, with a smooth portion $\Gamma$ of $\partial \Omega$, and the defining function $\rho$ is smooth near  $\Gamma$. 
Then for the solution $w$ to \eqref{eq-KE},
$w + \log(-\rho) $ is a smooth function in
$z, \rho^{n+1} \log (- \rho)$ in a neighborhood of $\Gamma$ in $\bar \Omega$.
\end{theorem}

We point out that Theorem \ref{thm-Main1} is not one of the main results in this paper. In fact, we will provide an alternative proof of \eqref{eq-expan}
using techniques developed in \cite{HanJiang} 
and then provide a short proof of Theorem \ref{thm-Main1} based on \eqref{eq-expan}. 
It is obvious that the smoothness of $w + \log(-\rho) $ in
$z, \rho^{n+1} \log (- \rho)$ as in Theorem \ref{thm-Main1} implies  \eqref{eq-expan}, 
with the help of 
the formal computation in  \cite{Fefferman1976}. 

By refining the arguments leading to  Theorem \ref{thm-Main1} significantly, 
we prove the next theorem, considered as the main result in this paper.

\begin{theorem}\label{thm-GlobalConvergence}
Assume that  $\Omega \subseteq \mathbb{C}^n$ is a  bounded  strictly pseudoconvex domain and $ \rho$  is analytic near $\partial \Omega$.
Then, $u$ in \eqref{eq-u}
is analytic in $z,  \rho^{n+1}\log (-\rho)$ in a neighborhood of $\partial \Omega$ in $\bar \Omega.$
\end{theorem}

We also call Theorem \ref{thm-GlobalConvergence} a convergence theorem, 
as it implies the convergence of the series $\sum_{i=0}^\infty a_j \rho^{(n+1)j}(\log (-\rho))^j$ 
derived from \eqref{eq-expan}.

To prove Theorem \ref{thm-GlobalConvergence}, we construct a class of frames near the boundary 
explicitly and derive analyticity-type estimates. 
We now briefly discuss difficulties in the proof of  Theorem \ref{thm-GlobalConvergence}. 

The function $u$ as in Theorem \ref{thm-GlobalConvergence} satisfies an elliptic equation which becomes 
degenerate along the boundary. However, the degeneracy rates are different along different directions
due to the complex structure. 
This is sharply different from degenerate elliptic equations associated with many geometric problems 
in real Euclidean spaces,  notably the Loewner-Nirenberg problem 
and the study of conformally compact Einstein metrics. In bounded domains in real Euclidean spaces, 
these equations have the following form: 
\begin{equation}\label{eq-UDEE}\rho^2a_{ij}\partial_{ij}w+\rho b_i\partial_iw+cw=f,\end{equation}
where $\rho$ is a defining function of the domain, and the symmetric matrix $(a_{ij})$ 
has uniform positive upper and lower bounds in the domain. 
The factor $\rho^2$ in the leading terms in \eqref{eq-UDEE} causes degeneracy, 
with a degeneracy rate given by 2. 
The equation \eqref{eq-UDEE} is called a uniformly degenerate elliptic equation. 
However, the equation satisfied by $u$ as in Theorem \ref{thm-GlobalConvergence} has a different form. 

Let $\Omega\subset \mathbb C^n$ be the bounded smooth domain in Theorem \ref{thm-GlobalConvergence}. 
Locally, near each boundary point, there is a special tangent field of the boundary. 
This is the tangent field forming the complex structure with the normal field. 
For the equation satisfied by $u$, the degeneracy rate is still 2 in the leading terms 
along the normal direction and this special tangential direction, 
but is 1 in other $2n-2$ tangential directions. 
The degeneracy rate is also 2 in mixed directions. 
(Refer to Lemma \ref{lem-Lu} for details.) 
Such nonuniform degeneracy contributes all difficulties we encounter in this paper. 

To prove the analyticity as asserted in Theorem \ref{thm-GlobalConvergence}, 
we adopt the Fuschsian equation techniques in Kichenassamy and Littman \cite{KL:1}, \cite{KL:2}. 
A fundamental step is to derive analyticity-type estimates for the coefficient $a_0$ in \eqref{eq-expan}.
For uniformly degenerate elliptic equations, we are able to prove a local convergence theorem. 
Refer to \cite{HanJiang2} for the study of the minimal graphs in the hyperbolic space. 
In this paper, we prove the convergence in the global setting. 

We point out one subtlety in the derivation of the analyticity-type estimates for $u$. 
These estimates are to be established for weighted derivatives of $u$ and the weights are
determined by the degeneracy rate in the equation. 
Hence, we use weights of order 2 for the second derivatives in the normal and the special tangential directions, 
and weights of order 1 for the second derivatives in other $2n-2$ tangential directions. 
This suggests that we need to consider weights of order 3/2
for these directions, although fractional decay rates never appear in the equation. 
Refer to Lemma \ref{lem-MatrixEstm} for details. 

It is believed that the local convergence theorem  for  \eqref{eq-KE} may not hold in general. 
A counterexample for a  model linear equation with a similar structure as that of the linearized equation of $u$  
is given  in Section \ref{Sec-CntrEx}. 
Due to the different degeneracy rates along different tangential directions, 
we construct a solution that is not analytic  with respect to the tangential coordinates 
along the boundary.

In the local setting, we have the following result.

\begin{theorem}\label{thm-LocalGevery}
Assume that  $\Omega \subseteq \mathbb{C}^n$ is a $C^7$ bounded strictly pseudoconvex domain and  $\Gamma$ is an open portion of $\partial \Omega$ such that $\rho$ is analytic near $\bar \Gamma$. 
Let $u$ be as in \eqref{eq-u}.
Then, there are  constants $D, B>0$ such that, for any  $p \in \mathbb{N}$,
\begin{align}\label{eq-Dpu-Gev}
|D^p_{y'}u|\leq DB^p(p!)^2
\end{align}
in a neighborhood of $\Gamma$ in $\bar \Omega$,
where $D^p_{y'} u$ denotes any $p$-th order  derivative of $u$   with respect to tangential coordinates $y'$ defined on $\Gamma$. 
\end{theorem}

Here, we extend $y'$ on $\Gamma$ to its neighborhood along the level set of $\rho.$ 
 We point out that the example constructed in Section \ref{Sec-CntrEx} satisfies \eqref{eq-Dpu-Gev}
and the power 2 cannot be improved. 
The collection of functions satisfying \eqref{eq-Dpu-Gev} is the 
tangential Gevrey-2 space.

This paper is organized as follows. 
In Section \ref{sec-Preliminary}, we introduce some basic notations and recall the work of Cheng and Yau \cite{ChengYau1980CPAM}.  In Section \ref{sec-Frame}, we build a set of frames, mainly developed for analyticity estimates of $u$.
In Section \ref{sec-FormalComputation}, we derive the linearized equation under local frames. 
In Section \ref{sec-BasicEstimates}, we derive a weighted $C^{2,\alpha}$ estimates of $u$. 
In Section \ref{sec-TE}, we  establish estimates 
involving tangential derivatives. 
In Section \ref{sec-thm-1}, we prove Theorem \ref{thm-Main1}.
In Section \ref{Sec-CntrEx}, we construct a counterexample for the analyticity of the solutions to
a model linear equation, and show that the example is in Gevrey-2 space. 
In Section \ref{Sec-Converg}, we prove Theorem \ref{thm-GlobalConvergence}. 
In Section \ref{Sec-Local}, we prove Theorem \ref{thm-LocalGevery}. 

We would like to thank Hans-Joachim Hein, Yalong Shi, and  
Xiaodong Wang for  helpful discussions.

 \section{Preliminary}\label{sec-Preliminary} 

In this paper, we use the  index sets $I= \{1,\cdots, 2n\}, I_1 =  \{i \in I: i\neq 2n\}$ and $I_2 = \{i\in I: i\neq n, 2n\}$.
In addition, we always denote by $d(\cdot)$ the Euclidean distance function to the boundary of the domain $\Omega$.

Assume that $\rho$ is a strictly plurisubharmonic defining function of $\Omega$, 
i.e., $\rho=0$ on $\partial \Omega$, $d\rho \neq 0$ on $\partial \Omega$, 
$(\rho_{i\bar{j}})_{i, j \in I}>0$ in $\bar{\Omega}$ and $\Omega= \{ \rho < 0\}$. 
Specifically, we take
\begin{align}\label{def-lambda}
\rho=e^{-\lambda d}-1\quad\text{near }\partial \Omega, 
\end{align}
for some $\lambda>0$ large. We only need to prove Theorems \ref{thm-Main1}-\ref{thm-LocalGevery} when $\rho$ is given by \eqref{def-lambda}.

Let $z^1,\cdots, z^n$ be complex coordinates of $\mathbb{C}^n$ and $x^1, \cdots, x^{2n}$ be the corresponding real coordinates such that $z^i = x^{i+n} +\sqrt{-1}x^i$ for $i=1,\cdots, n$.
We also use the principal coordinates $y^1, \cdots, y^{2n}$ near $\partial \Omega$, which satisfy that
 $y^{2n}= d$, and for $y'=(y^1, \cdots, y^{2n-1})$, 
$x = (y', \varphi(y'))+\nu (y') d$,  where the graph of $\varphi$ is a portion of $\partial \Omega$ and $\nu$ is the unit inner normal vector of $\partial \Omega$. 
See Section 14.6 of \cite{GT}.
 In addition, 
\begin{equation}\label{eq-DxDy}
\left(\frac{\partial x}{\partial y}\right)_{2n\times 2n}=\left(
\begin{array}{ccccc}
1+\frac{\partial \nu_1}{\partial y^1}d & \frac{\partial \nu_1}{\partial y^2}d 
& \ldots & \frac{\partial \nu_1}{\partial y^{2n-1}}d & \nu_1 \\
\frac{\partial \nu_2}{\partial y^1}d & 1+\frac{\partial \nu_2}{\partial y^2}d 
& \ldots & \frac{\partial \nu_2}{\partial y^{2n-1}}d & \nu_2 \\
\vdots&\vdots& &\vdots&\vdots\\
\frac{\partial \nu_{2n-1}}{\partial y^1}d & \frac{\partial \nu_{2n-1}}{\partial y^2}d 
& \ldots & 1+\frac{\partial \nu_{2n-1}}{\partial y^{2n-1}}d & \nu_{2n-1} \\
\frac{\partial \varphi}{\partial y^1}+\frac{\partial \nu_{2n}}{\partial y^1}d 
& \frac{\partial \varphi}{\partial y^2}+\frac{\partial \nu_{2n}}{\partial y^2}d 
& \ldots & \frac{\partial \varphi}{\partial y^{2n-1}}+\frac{\partial \nu_{2n}}{\partial y^{2n-1}}d & \nu_{2n}
\end{array}
\right),
\end{equation}
Denote this matrix by $A$ , and  its inverse  by $B$.  Then
\begin{align}\label{eq-B}
B = I - (\nu_1, \cdots, \nu_{2n-1}, -1)^T \cdot (\nu_1, \cdots, \nu_{2n}) + \mathbf{O}(d),
\end{align}
where $\mathbf{O}(d)$ denotes a matrix with smooth and  $O(d)$ entries. Set $\frac{\partial}{\partial d} = \frac{\partial}{\partial y^{2n}}$. Notice that
\begin{align}\label{eq-Yn-Y2n}
Y_{2n} = \frac{\partial}{\partial d} , \,\, Y_n =J  \left(\frac{\partial}{\partial d}\right)
\end{align}
are defined in a neighborhood of $\partial \Omega.$

Let $g_{i \bar{j}}=(-\log  (-\rho))_{i \bar{j}}$. 
For any  interior $P$ near $\partial \Omega$, we can find a $Q\in \partial \Omega$ such that $ d(P) =dist(P, Q)$. Use the complex coordinates $(z^i_Q, \cdots, z^n_Q)$ defined as \eqref{eq-zQ}.
Cheng and Yau \cite{ChengYau1980CPAM}  proved that
$g_{i\bar{j}}$ satisfies the condition of bounded geometry in the coordinates chart
\begin{align}\label{eq-Theta}
&\qquad (\vartheta^1, \cdots, \vartheta^n)=
\frac{(2d(P)-d(P)^2)^\frac{1}{2}}{z^n_Q+ d(P)-z^n_Q d(P)}\left( z^1_Q, \cdots,  z^{n-1}_Q, \frac{z^n_Q-d(P)}{(2d(P)-d(P)^2)^\frac{1}{2}} \right).
\end{align}
When $d(P)<1$, 
we define the $C^{k, \alpha}_g$-norm on $\{|\vartheta|<\eta\}$ to be $C^{k, \alpha}$ norm
with respect to $\{\vartheta^1, \cdots, \vartheta^n\}$. For the global norm, we use the $\mathbb{C}^n$ coordinate charts for $\{d(P)>1\}$ and a family of holomorphic coordinate charts $(U, (\vartheta^1, \cdots, \vartheta^n))$ covering $\{0<d(P)\leq 1\}$. Define the global norm to be the supremum of the $C^{k, \alpha}_g$-norms on the local domain $U$'s and the regular $C^{k, \alpha}$-norm on $\{d(P)>1\}$. 

Denote
\begin{align}\label{eq-BPe}
B(P, \eta) = \{z\in \mathbb{C}^n :|\vartheta(z)|<\eta\}.
\end{align}
 Throughout the paper, we assume that $\eta>0$ is a fixed small number such that Lemma 
\ref{lem-smooth-dP} holds. We say that the $C^{k, \alpha}_g(B(P,\eta))$ norm of a function is uniformly bounded if the bound  is independent of $P$.

Cheng and Yau \cite{ChengYau1980CPAM} proved the following theorem on the existence and interior estimates of solutions.

\begin{theorem}\label{thm-IntEstm}
Suppose $\Omega$ is a $C^{k+2}$ strictly pseudoconvec domain 
and $\rho$ is a $C^{k+2}$ defining function of $\Omega$, for some $k\geq 5$. 
Then for any $F\in C_g^{k-2, \alpha}(\Omega)$, $\alpha\in (0,1)$, there exists a unique 
$u \in C_g^{k, \alpha}(\Omega)$, such that
\begin{align}\label{eq-MainEq}
\det( g_{i \bar{j}}+ u_{i \bar{j}}) &= e^{(n+1) u}e^F\det(g_{i \bar{j}}) \quad \text{in } \Omega, \\
\frac{1}{c} g_{i\bar{j}} & \leq g_{i\bar{j}} + u_{i\bar{j}} \leq c g_{i\bar{j}}\label{eq-equ-metric},
\end{align}
for some constant $c.$
\end{theorem}

In this paper, we set
\begin{align}\label{eq-F}
F= - \log (\det(\rho_{i \bar{j}} )(-\rho+ \rho^{ \bar{j} i} \rho _i  \rho_{\bar{j}})).
\end{align}
Then $w= -\log (- \rho) + u$  is a solution to \eqref{eq-KE}. By Theorem \ref{thm-IntEstm}, if $\Omega \in C^7$, then
\begin{align}\label{eq-C5a}
u \in C_g^{5, \alpha}(\Omega).
\end{align}

\begin{theorem}
Assume that $\Omega$ is a $C^7$ strictly pseudoconvex domain. Then
the solution to \eqref{eq-KE} is unique.
\end{theorem}
\begin{proof}
It
follows essentially from Lemma 5.3 in \cite{ChengYau1980CPAM}. Let $w$ be any solution to \eqref{eq-KE}.
We set a $C^7$ family  of strictly pseudoconvex domains $\Omega^s$, for $-\delta\leq s\leq  \delta$, to be $$\Omega^s = \{z\in \mathbb{C}^n : \rho(z) < -s\}.$$
Here $\delta$ is sufficiently small.
On each $\Omega^s$, there exists at least a $C^2(\Omega^s)$ solution $w_s$ to \eqref{eq-KE}.
Then $e^{-(n+1) w_s} \det((w_s)_{i \bar j})=1$. Lemma 5.3 in \cite{ChengYau1980CPAM} implies that $w_s \geq w_t$ in $\Omega^s$ if $-\delta\leq s< t\leq \delta$. So we have 
\begin{align}\label{eq-ws}
w_{-t} \leq w \leq w_{s}
\end{align}
 for $0<s, t\leq \delta$.
Now we fix $\{w_s\}_{-\delta\leq s \leq \delta}$ to be the family to solutions which are of form 
$$
w_s = -\log (-\rho_s ) + u_s,$$ where $\rho_s = \rho +s$ is the defining function of $\Omega_s$ and $u_s$ is given by Theorem \ref{thm-IntEstm}. By \cite{ChengYau1980CPAM}, 
\begin{align*}
\sup_{\Omega_s} |u_s| \leq (n+1)^{-1} \sup_{\Omega_s} |F_s|,
\end{align*}
where $F_s$ is given by \eqref{eq-F}, replacing $\rho, g_{i \bar j}$ by $\rho_s, (g_s)_{i \bar j} = - (\log (-\rho_s) )_{i \bar j} $. $F_s$ is uniformly bounded independent of $s$.  Proposition 4.2, (4.9) and (4.10) in \cite{ChengYau1980CPAM} supply an upper bound of $c$, which is independent of $s$, so that
\begin{align*}
\frac{1}{c} (g_s)_{i\bar{j}} & \leq (g_s)_{i\bar{j}} + (u_s)_{i\bar{j}} \leq c (g_s)_{i\bar{j}}.
\end{align*}
Then
 we apply the $C^3$ estimate in \cite{ChengYau1980CPAM} to derive that
 \begin{align}\label{eq-C^3-us}
 \|u_s\|_{C^3_{g_s}}( \Omega_s) \leq C,
 \end{align}
where $C$ is independent of $s$. By passing to a subsequence if necessary, on any compact domain $K \subset \Omega_0$, 
\begin{align}
u_s \rightarrow \tilde u \text{ in } C^3(K),
\end{align}
for some $\tilde u \in C^3(\Omega_0)$.
Then the theorem is concluded if we can show that $\tilde u = u_0$. As then there are two subsquences $\{w_s\}, \{w_{-t}\}$ such that $w_s \rightarrow w_0, w_{-t} \rightarrow w_0$ in $C^3(K)$ as $s, t\rightarrow 0^+.$ \eqref{eq-ws} shows that $w$ must be $w_0$. 

Now we show that $\tilde u = u_0$.
Fix an arbitrary $P\in \Omega$ near $\partial \Omega$.
Set $K = \overline {B(P, \eta)}$ defined as in \eqref{eq-BPe}. Then $\Omega^s$ contains $K$ if $s$ is sufficiently close to 0.
Let $\vartheta_s$ to be the coordinates defined in the form of \eqref{eq-Theta} with the same $P$ but we require that $Q\in \partial \Omega^s.$  $\vartheta^s$ converges to $\vartheta$ in $C^\infty(K)$ as $s \rightarrow 0.$ Recall that the $\|\cdot\|_{g_s}^3$ norm is defined as $C^3$ norm with respect to coordinates $\vartheta_s$.
By \eqref{eq-C^3-us}, $
 \|u_s\|_{C^3_{g_s}}(K) \leq C.
$ We take limit $s\rightarrow 0$ to derive that $
 \|\tilde u\|_{C^3_{g}}(K) \leq C.
$ This shows that $\tilde u$ should be the unique solution $u_0$ described in Theorem \ref{thm-IntEstm}.
\end{proof}

\section{Frames near $\partial \Omega$}\label{sec-Frame}
In this section, we build
  a set of frames $(U_Q, \{Y_i^Q\}_{i \in I})$, where $Q \in \partial \Omega$, $U_Q$ is a neighborhood of $Q$ in $\mathbb{C}^n$, and all vector fields $Y_i^Q$'s are  defined in $\mathbb{C}^n$, while only in the domain $U_Q \subseteq \mathbb{C}^n$, they form a frame. Similar frames are also used in \cite{LeeMelrose}. We require that when $\partial \Omega$ is analytic, all $Y_i^Q$'s are analytic in $\overline \Omega$. In proving Theorem \ref{thm-Main1}, we do not need such specific frames. 
But for the analyticity estimates, it plays an important roll.

Given complex coordinates $\{z^1, \cdots, z^n\}$ for $\mathbb{C}^n$, the corresponding real coordinates are $\{x^1,\cdots, x^{2n}\}$ such that $z^j = x^{j+n}+ \sqrt{-1} x^j $ for $j=1,\cdots, n$. Let 
\begin{align*}
m = \min\{ x_{2n}(Q): Q\in \partial \Omega\}
\end{align*}
Without lost of generally, we assume that $m=0$ and the origin $O \in \partial \Omega$. Then at $O$, the inner normal vector to $\partial \Omega$ is $\frac{\partial}{\partial x^{2n}}$. 

At any point $Q \in \partial \Omega$, denote by $\mathbf{\nu}(Q)\in \mathbb{C}^n$ its unit inner normal vector. 
Note that $\mathbf{\nu}(O)=(0,\cdots, 0, 1)^T.$ We  choose a  unitary transformation $T_1(Q): \mathbb{C}^n \rightarrow \mathbb{C}^n$ such that
\begin{align*}
T_1(Q): \nu(O)\rightarrow \nu(Q).
\end{align*}
$T_1$ can be derived
 from the Gram-Schmidt process. In fact, if the $n$-th column of $\nu(Q)$ satisfies $|\nu_n(Q)| \geq \frac{1}{\sqrt{n}}$, then we take $T_1$ as the Gram-Schmidt process of the matrix
\begin{equation}\label{eq-GS}
\left(
\begin{array}{ccccc}
1 & 0 & \ldots & 0 & \nu_1(Q) \\
0 & 1 & \ldots & 0 & \nu_2(Q) \\
\vdots&\vdots& &\vdots&\vdots\\
0 & 0& \ldots & 1 & \nu_{n-1} (Q)\\
0 & 0 & \ldots & 0 & \nu_{n}
(Q)\end{array}
\right).
\end{equation}
During the Gram-Schmidt process, we keep the $n$-th column. 
 The choice of $T_1$ is not unique and we pick one. If  $|\nu_n(Q)| \geq \frac{1}{\sqrt{n}}$ is not true, then there is some $i \neq n$ such that  $|\nu_i(Q)| \geq \frac{1}{\sqrt{n}}$. In this case, we  replace the $i$-th column of \eqref{eq-GS} by $(0, \cdots, 0,1)^T$ and apply  the Gram-Schmidt process to derive $T_1$.

Define $T_2(Q)$ as the translation from $O$ to $Q$.  We consider the following complex coordinate chart centered at $Q$,
\begin{align}\label{eq-zQ}
 (z^i_Q, \cdots, z^n_Q)^T= T_1^{-1}(Q)T_2^{-1}(Q) (z^1, \cdots, z^n)^T.
\end{align}
Consider the corresponding real coordinates as $(x_1^Q, \cdots, x_{2n}^Q)$.   Denote $Y_{2n}^Q = Y_{2n}$,  $Y_n^Q = Y_n$ as defined in \eqref{eq-Yn-Y2n} and for $i \in I_2,$
\begin{align}\label{eq-YiQ}
Y_i^Q = \frac{\partial}{\partial x_Q^i}-\left(\frac{\partial}{\partial x_Q^i}, Y_n\right) Y_n - \left(\frac{\partial}{\partial x_Q^i}, Y_{2n}\right) Y_{2n}.
\end{align}
 Then $\{Y_i^Q\}_{i\in I}$ is globally defined and it is a frame  in some neighborhood $U_Q$ of $Q$. 
Obviously, 
\begin{align}\label{eq-Yi-Y2n-Q}
\left(Y^Q_i, Y_{n}\right)=\delta_{i, n},\quad \left(Y^Q_i, Y_{2n}\right)=\delta_{i, 2n}.
\end{align}
Thus for $\alpha \in I_1$,
\begin{align*}
Y^Q_\alpha d = \left(Y^Q_\alpha, \frac{\partial}{\partial d}\right) = 0.
\end{align*}

We also have the following lemma,
\begin{lemma}\label{lem-Lie-Y}
For $\alpha, \beta \in I_1$, 
\begin{align*}
([Y^Q_\alpha, Y^Q_\beta], Y_{2n})=0, \quad ([Y^Q_\alpha, Y_{2n}], Y_{2n})=0.
\end{align*}
\end{lemma}
\begin{proof}
Both $Y^Q_\alpha, Y^Q_\beta$ are orthogonal to $Y_{2n}$.
\eqref{eq-Yi-Y2n-Q} implies that
\begin{align*}
Y_\alpha^Q = \sum_{i\in I_1}f_i(z) \frac{\partial}{\partial y^i_Q}, \quad Y_\beta^Q = \sum_{j\in I_1}g_j(z) \frac{\partial}{\partial y^j_Q}.
\end{align*} 
Here $i,j \in I_1$ means that  $\frac{\partial}{\partial y^{2n}_Q} =\frac{\partial}{\partial d}$ does not appear in  the above equations.
Thus
\begin{align*}
[Y^Q_\alpha, Y^Q_\beta] =\sum_{i, j\in I_1}\left( f_i \frac{\partial g_j}{\partial y_Q^i} \frac{\partial}{\partial y^j_Q} -g_j\frac{\partial f_i}{\partial y_Q^j} \frac{\partial}{\partial y^i_Q} \right),
\end{align*}
is orthogonal to $Y_{2n}.$ Similarly, 
$([Y^Q_\alpha, Y_{2n}], Y_{2n})=0.$
\end{proof}

In  \cite{ChengYau1980CPAM}, it is shown that in $\Omega_1=\{z\in \Omega : d(z)<1\},$ if $|\vartheta| \leq \eta$, then there is a $c$ depending only on $\eta$ such that
\begin{align}\label{eq-c}
|z^n_Q -d(P)| \leq cd(P), |z_Q^\alpha|\leq c\sqrt{d(P)}.
\end{align}
 In addition, 
\begin{align}\begin{split}\label{eq-z-theta}
d(P)\frac{\partial}{\partial z_Q^n} &=\frac{2 -d(P)}{(z^n_Q/d(P)+1-z^n_Q)^2}\frac{\partial}{\partial \vartheta^n}\\&\qquad - \frac{(1-d(P))(2-d(P))^\frac{1}{2}z_Q^\alpha /\sqrt{d(P)}}{(z^n_Q/d(P)+1-z_Q^n)^2 } \frac{\partial}{\partial \vartheta^\alpha}\\ 
\sqrt{d(P)}\frac{\partial}{\partial z_Q^\alpha} &=\frac{\sqrt{2-d(P)}}{z_Q^n/d(P)+1-z^n_Q} \frac{\partial}{\partial \vartheta^\alpha}.
\end{split}
\end{align}
 $z^n_Q/d(P), z^\alpha_Q/\sqrt{d(P)} $ are smooth with respect to $\vartheta$ in $B(P, \eta) = \{|\vartheta|<\eta\}$ according to  the following lemma.
 \begin{lemma}\label{lem-smooth-dP}
 Let $P \in \Omega_1$, $Q \in \partial \Omega$ and $d(P) =dist(P,Q)$. 
 Then there is a constant  $\eta>0$ independent of $P$, such that for $\beta =1,\cdots, n-1,$
 \begin{align}\label{eq-dP}
 \frac{z^n_Q}{d(P)}, \,\, \frac{z^\beta_Q}{\sqrt{d(P)}}, \,\,\frac{d}{d(P)} \in C^{\infty}_g(B(P, \eta)),
 \end{align}
where the $C^{k,\alpha}_g$ norms of terms in \eqref{eq-dP} are independent of the choice of  $P$. In addition,
\begin{align}\label{eq-d-dP}
\frac{d}{d(P)}  \in \left(\frac{1}{2}, 2\right)
\end{align} 
   in $B(P, \eta).$
 \end{lemma}
 \begin{proof}
 According to \eqref{eq-c},  $z^n_Q/d(P), z^\alpha_Q/\sqrt{d(P)} $ are continuous in $B(P, \eta)$, and their $L^\infty$ norms are independent of $P$.
According to   \cite{ChengYau1980CPAM},  for $\alpha,\beta \in \{1,\cdots, n-1\},$
\begin{align*}
\frac{\partial}{\partial \vartheta^n}\left(z^n_Q/d(P) \right) &=\frac{((1-d(P))z_Q^n/d (P)+1)^2}{2-d(P)},\\
\frac{\partial}{\partial \vartheta^n}\left(z^\alpha_Q/\sqrt{d(P)} \right) &=\frac{(1-d(P))((1-d(P))z_Q^n/d(P)+1)z_Q^\alpha/\sqrt{d(P)}}{2-d(P)},\\
\frac{\partial}{\partial \vartheta_\beta}\left(z^n_Q/d(P) \right) &=0,\\
\frac{\partial}{\partial \vartheta_\beta}\left(z^\alpha_Q/\sqrt{d(P)} \right) &=\delta_{\alpha\beta}\frac{(1-d(P))z_Q^n/d(P)+1}{(2-d(P))^\frac{1}{2}}.
\end{align*}
The right-hand side of these equations are smooth in $z^n_Q/d(P)$ and $z^\alpha_Q/\sqrt{d(P)}$ in $\Omega_1$.
Keep differentiating these equations with respect to $\vartheta_n$ and $\vartheta_\beta$, we derive 
 that $z^n_Q/d(P)$, $ z^\alpha_Q/\sqrt{d(P)}$ are smooth with respect to $\vartheta$ in $B(P, \eta)$.

Next,  for points $z\in B(P, \eta)$, applying the mean value theorem and \eqref{eq-c}, 
 \begin{align}\label{eq-d-mvt}
|d(z)-d(P) | &= \left|\frac{\partial d}{\partial x_Q^i}(\xi) (x^i-x^i(P))\right|\\
&\leq c d(P)\max_{|\vartheta|<\eta} B_{2n, 2n}  +c\sqrt{d(P)}\max_{\alpha\in I_1, |\vartheta|<\eta} B_{2n, \alpha},
\end{align}
where $|B_{2n, \alpha}| = |\nu_\alpha+O(d)|\leq C(\sqrt{d(P)}+d)$ according to \eqref{eq-c}, and $B_{2n, 2n} = \nu_n+O(d)$ is uniformly bounded. Requiring $\eta, c$ sufficiently small, we derive that
\begin{align*}
|d(z)-d(P) | \leq \frac{1}{2}d(P),
\end{align*} 
which implies that $\frac{d}{d(P)}\in (\frac{1}{2}, \frac{3}{2}).$ 
Applying \eqref{eq-DxDy} and
\begin{align*}
\frac{\partial}{\partial \vartheta^n} &=\frac{(1-d(P))z_Q^n/d(P)+1}{2-d(P)} d(P)\frac{\partial}{\partial z_Q^n} \\
&\qquad- \sum_{\alpha=1}^{n-1} \frac{(1-d(P))((1-d(P))z^n_Q/d(P)+1)z_Q^\alpha/\sqrt{d(P)}}{2-d(P)}  \sqrt{d(P)} \frac{\partial}{\partial z_Q^\alpha},\\
\frac{\partial}{\partial \vartheta^\alpha} &=\frac{(1-d(P))z_Q^n/d(P)+1}{\sqrt{2-d(P)}} \sqrt{d(P)}\frac{\partial}{\partial z_Q^\alpha}, 
\end{align*}
for $\alpha=1,\cdots, n-1,$ we derive that 
 $\frac{d}{d(P)}$  is smooth with respect to $\vartheta$ in $B(P, \eta)$ with uniform $C^k_g$ bounds independent of $P$.
 \end{proof}
 
 We fix a small $\eta$ such that Lemma \ref{lem-smooth-dP} holds.

\begin{lemma}\label{lem-Yi-xj}
  Let $P \in \Omega_1$, $Q \in \partial \Omega$ and $d(P) =dist(P,Q)$. Then  in $B(P, \eta)$,
\begin{align}\label{eq-Yi-xi}
Y_i^Q = \frac{\partial}{\partial x_Q^i} + \sum_{j=1}^{2n} \mathbf{O_\vartheta}(1)\sqrt{d(P)} \frac{\partial}{\partial x_Q^j},
\end{align}
where $\mathbf{O_\vartheta}(1)$ denotes some functions smooth in $\vartheta$ and whose $C^{k,\alpha}_g$ norms are independent of $P, Q$.
\end{lemma}
\begin{proof}
First, applying \eqref{eq-DxDy} with $d = y_{2n}^Q$,
\begin{align*}
Y_{2n}^Q = \frac{\partial}{\partial d} = \sum_{i=1}^{2n} \frac{\partial x_Q^j}{\partial d}  \frac{\partial}{\partial x_Q^j} =\frac{\partial}{\partial x^Q_{2n}}+ \sum_{j=1}^{2n}( A_{j, 2n}-\delta_{j, 2n})\frac{\partial}{\partial x_Q^j},
\end{align*}
where $A_{j, 2n}-\delta_{j, 2n}= \nu_j -\delta_{j, 2n}=
 \mathbf{O}_\vartheta(1)\sqrt{d(P)}$. In fact, $\nu_j-\delta_{j, 2n}$ is smooth in $z$ and $\nu_j -\delta_{j, 2n}=0$ at $P$. Then Lemma \ref{lem-smooth-dP} implies that 
 $\nu_j -\delta_{j, 2n} =  \mathbf{O}_\vartheta(1)\sqrt{d(P)}.$
 It verifies \eqref{eq-Yi-xi} for $i=2n.$

Secondly, 
\begin{align*}
Y_n^Q = J Y_{2n}^Q = J\left( \frac{\partial}{\partial x^Q_{2n}} + \sum_{i=1}^{2n} \mathbf{O}_\vartheta(1)\sqrt{d(P)}\frac{\partial}{\partial x_Q^i} \right),
\end{align*}
which verifies \eqref{eq-Yi-xi} for $i=n$.

Lastly, for $i \in I_2$, we use \eqref{eq-YiQ}.
Since \eqref{eq-Yi-xi} is true for case $Y^Q_{2n}, Y^Q_{n}$, the inner products $\left(\frac{\partial}{\partial x_Q^i}, Y_n\right)=\mathbf{O}_\vartheta(1)\sqrt{d(P)}$, $\left(\frac{\partial}{\partial x_Q^i}, Y_{2n}\right)= \mathbf{O}_\vartheta(1)\sqrt{d(P)}$, which concludes the lemma.

\end{proof}

 Lemma \ref{lem-Yi-xj} implies  immediately that in $B(P,\eta)$,
 \begin{align}\label{eq-xi-Yj}
 \frac{\partial}{\partial x_Q^i} = Y_i^Q + \sum_{j=1}^{2n} \mathbf{O_\vartheta}(1)\sqrt{d(P)}Y^Q_j,
 \end{align}

\begin{lemma} \label{lem-NormY}
If   $v$ is $C^{2,\alpha}_g$, then  for $\beta, \gamma =1,\cdots, n-1,$
 \begin{align}\label{eq-list-Y2v}
 dY_n v, dY_{2n} v,  \sqrt{d}Y_\beta v,
d^2 Y_n^2 v, d^2 Y_{2n}^2 v, d Y_{\beta}Y_{\gamma} v, d^{\frac{3}{2}} Y_{n}Y_{ \beta} v,  d^{\frac{3}{2}} Y_{2n}Y_{ \beta} v\in C^\alpha_g
 \end{align}
in $B(P, \eta) = \{|\vartheta|<\eta\}$. In addition, $C^\alpha_g$ norms of terms in \eqref{eq-list-Y2v} on $\{|\vartheta|<\eta\}$  are bounded by $C\|v\|_{C^{2,\alpha}_g}$, where $C$ is independent of the choice of $P$.
\end{lemma}
\begin{proof}

Applying
\eqref{eq-z-theta},  if $v\in C^{2,\alpha}_g$ in  $B(P, \eta)$, then
\begin{align}\label{eq-d-u-z}
d(P) \frac{\partial v}{\partial z^n_Q}, \sqrt{d(P)} \frac{\partial v}{\partial z^\alpha_Q}, d(P)^2 \frac{\partial^2 v}{\partial z^n_Q \partial \bar z^n_Q},d(P)^\frac{3}{2} \frac{\partial^2 v}{\partial z^n_Q \partial \bar z^\alpha_Q}, d(P) \frac{\partial^2 v}{\partial z^\alpha_Q \partial \bar z^\beta_Q} \in C^\alpha_g
\end{align}
 and their $C^\alpha_g$ norms  on $B(P, \eta)$ are independent of  the choice of $P$.

Applying \eqref{eq-xi-Yj}, Lemma \ref{lem-smooth-dP} and \eqref{eq-d-u-z},
functions listed in \eqref{eq-list-Y2v} are $C^\alpha_g$ in $\{|\vartheta|<\eta\}$. 
We  omit the index $Q$ in \eqref{eq-list-Y2v}
since \eqref{eq-list-Y2v} is still correct if we change $\{Y_i^Q\}$ to other local frames $\{Y_i^{\tilde Q}\}$ over $B(P, \eta)$.
\end{proof}

Let $\Gamma$ be a smooth open portion of $\partial \Omega$.
Define
\begin{align*}
\Gamma \times (0, R):= \{(y', d): (y', 0)\in \Gamma, 0<d<R\}.
\end{align*}
\begin{lemma}\label{lem-w-Ca}
Assume that $w\in C^{\alpha}_g$ in $\Gamma \times (0, R)$, and for any $B(P,\eta)\subseteq \Gamma \times (0, R),$
\begin{align}\label{eq-w-Ca}
\|w\|_{C^{\alpha}_g(B(P,\eta))}\leq Cd^{\epsilon}(P),
\end{align}
where $C$ is independent of $P$.
Then for any open subset $\Gamma'$ of $\Gamma$ such that $\bar \Gamma' \subseteq \Gamma$,
 there is  $\tilde \alpha>0$, such that
\begin{align*}
\|w\|_{C^{\tilde \alpha}(\Gamma' \times (0, R))} \leq C,
\end{align*}
for some different $C.$
\end{lemma}
\begin{proof}
\eqref{eq-w-Ca} implies $\|w\|_{L^\infty(\Gamma \times (0, R))} \leq C.$ The metric $g$ is equivalent to the Euclidean metric of $\mathbb{C}^n$ away from  $\partial \Omega$.
So we only have to prove that there is an $\tilde \alpha>0$ such that
$\|w\|_{C^{\tilde \alpha}(\Gamma' \times (0, r))} \leq C,$
 for some small $r>0.$
We set $r$ small enough such that if 
$P \in \Gamma' \times (0, r)$ then $B(P,\eta) \subseteq \Gamma \times (0, R).$
 In the rest of the proof, $C$ may change line by line.

 We set  $\alpha =\frac{\epsilon}{2}$ if $\alpha> \frac{\epsilon}{2}$. 
 
 First, for any $z \in$ $B(P,\eta) \subseteq \Gamma \times (0, r),$
 \begin{align*}
 \frac{|w(z)-w(P)|}{|\vartheta(z)|^{\alpha}} \leq Cd^\epsilon(P).
 \end{align*}
 Under the coordinates $z^Q$, where $Q\in \Gamma$ such that $d(P) = dist(P, Q),$
 \begin{align*}
 z^Q(P) = (0, \cdots, 0, d(P)).
 \end{align*}
 In the rest of the proof, we omit the index $Q$ for simplicity.
\eqref{eq-Theta} and \eqref{eq-c} imply that, for $\alpha = 1, \cdots, n-1,$
\begin{align*}
|\vartheta^\alpha| \sim  d(P)^\frac{1}{2}|z^\alpha| , \quad |\vartheta^n| \sim d(P) |z^n-d(P)|.
\end{align*}
Here $``a\sim b"$ means that, there is a $C$ independent of $P$ such that $C^{-1} b< a< Cb.$

Hence
  \begin{align*}
 \frac{|w(z)-w(P)|}{|d(P)^{-1} (z- z(P))|^\alpha} \leq Cd^\epsilon(P),
 \end{align*}
 which implies that
 \begin{align}\label{eq-w-Ca-1}
  \frac{|w(z)-w(P)|}{| z- z(P)|^\alpha} \leq Cd^{\epsilon-\alpha}(P)\leq C.
 \end{align}
 
 Secondly, we consider the case $z\notin B(P, \eta)$.
 
  If $|z- z(P)|>\delta$ for some fixed small $\delta>0$, then 
  \begin{align*}
  \frac{|w(z)-w(P)|}{| z- z(P)|^\alpha} \leq C.
 \end{align*}

 If $|z- z(P)|\leq \delta$  and $d(z) \in [2d(P), r)$, then
\eqref{eq-d-mvt}  implies that
\begin{align*}
\frac{1}{2}d(z) \leq |d(z)-d(P) | &= \left|B_{2n, i}(\xi) (x^i-x^i(P))\right|
\end{align*} 
By \eqref{eq-B}, $B_{2n, i}= \nu_{i} +\mathbf{O}(d)$.
So for $\delta$ small,  $|z-z(P)| =\sum_{i\in I} |x^i -x^i(P)|>\frac{1}{4} d(z)$. We derive
\begin{align*}
\frac{|w(z)-w(P)|}{| z- z(P)|^\alpha} \leq 4d(z)^{-\alpha} (Cd(z)^\epsilon + Cd(P)^\epsilon) \leq C.
\end{align*}

If $d(z) < 2d(P)$,
 \eqref{eq-Theta} and $z\notin B(P, \eta)$ imply that either $|z^n -d(P)|\geq cd(P)$ or 
 \begin{align*}
d(P)^{-\frac{1}{2}} |z^\alpha| + d(P)^{-1}|z^n -d(P)| \geq C^{-1}\eta,
 \end{align*}
 for some $C>0$, which is independent the choice of $P.$ Hence we have
 \begin{align*}
 |z - z(P)| \geq  \min\{C^{-1}\eta, c\} d(P), 
 \end{align*}
which implies that
 \begin{align*}
   \frac{|w(z)-w(P)|}{| z- z(P)|^\alpha} &\leq Cd(P)^{-\alpha} (d(z)^\epsilon + d(P)^\epsilon)\\
   & < C,
 \end{align*}
where we applied $d(z) <2 d(P)$.
\end{proof}

\section{Linearized equation and formal computation}\label{sec-FormalComputation}

Define the linearized operator of \eqref{eq-MainEq}  as
\begin{align*}
L u =\left.\frac{d}{d t}\right|_{t=0} \frac{\det(g_{i\bar j} +t u_{i\bar j})}{e^{(n+1)tu} \det(g_{i \bar j})} = g^{\bar k j}u_{j \bar k} - (n+1) u.
\end{align*}
We study $Lu$ using the frames defined in Section \ref{sec-Frame}.

\begin{lemma}\label{lem-Lu} 
Let  $\{Y_i\}_{i=1}^{2n}$ be any smooth frame on $U\subseteq \mathbb{C}^n,$ satisfying
$Y_{2n}=\frac{\partial}{\partial d}, Y_n =J\frac{\partial}{\partial d}$ and
\begin{align}\label{eq-Yi-Y2n}
\left(Y_i, Y_{n}\right)=\delta_{i, n},\quad \left(Y_i, Y_{2n}\right)=\delta_{i, 2n}.
\end{align}
Then
for any $u \in C^2$, we have that in $U \cap \Omega,$
\begin{align}\label{eq-MN}\begin{split}
g^{\bar k j}u_{j \bar k} &= d^2  (1+\mathbf{O}(d)) (u_{dd}+ Y_n^2 u)-((n-1)d+ \mathbf{O}(d^2))u_d\\
&\qquad+ d^2 \sum_{\beta \in I_2} a_{ \beta d} Y_\beta u_d 
+d^2 \sum_{\beta\in I_2} a_{\beta n} Y_\beta Y_n u
+d \sum_{ \beta,\gamma  \in I_2} a_{\beta \gamma} Y_\beta Y_\gamma u\\
&\qquad +\mathbf{O}(d)\sum_{\beta \in I_1} b_\beta Y_\beta u,
\end{split}
\end{align}
where $\mathbf{O}(d)$ denotes a smooth function that is $O(d)$.
The coefficients $a_{ij}$, $b_i$ are smooth functions of $z$ and $(a_{\beta \gamma})_{\beta,  \gamma\in I_2}$ is uniformly positive definite.
\end{lemma}


\begin{proof}
Recall that $\rho= e^{-\lambda d}-1$,  $g_{i\bar j} = -(\log (-\rho))_{i\bar j}$ and its inverse is
\begin{align}\label{eq-inverse}
g^{ \bar{j} i}&=-\rho\left({\rho}^{ \bar{j} i}+\frac{{\rho}^{ \bar{l} i} {\rho}^{\bar{j} m}{\rho}_{\bar l} \rho_{m}}{{\rho}
-{\rho}^{ \bar{q} p}\rho_{p}{\rho}_{\bar q}}\right),
\end{align}
where the term $-{\rho}^{ \bar{q} p}\rho_{p}{\rho}_{\bar q}$ is  strictly positive due to the strict pseudoconvexity of the domain.  We observe that $g^{ \bar j i} \frac{\partial^2}{\partial z_i \partial \bar{z}_j}$ is a degenerate elliptic operator but smooth   up to boundary. Under a local frame $\{Y_i\}$,
\begin{align}\label{eq-fij-fk}
g^{ \bar j i} \frac{\partial^2}{\partial z_i \partial \bar{z}_j} = \sum_{i, j=1}^{2n} \tilde a_{ij}(z) Y_i Y_j  + \sum_{i=1}^{2n} \tilde b_i(z) Y_i,
\end{align}
where the coefficients $ \tilde a_{ij},\tilde  b_k$'s are smooth functions, and are independent of the choice of complex coordinates.

We first compute $\tilde a_{ij}, \tilde b_k$'s  in \eqref{eq-fij-fk} pointwise. 
For any $P \in U$, we can find a point $Q \in \partial \Omega$ such that $d(P) = dist(P, Q)$. 
In the rest of this proof, we use  complex coordinates $(z^i_Q, \cdots, z^n_Q), $   real coordinates $(x^1_Q, \cdots, x^{2n}_Q),$ and the corresponding principal coordinates $(y^1_Q, \cdots, y^{2n}_Q),$ but omit the index $Q$. Since $\frac{\partial d}{\partial z_i}=\delta_{i n}/2$ at $P$, \eqref{eq-inverse} implies that at $P$,
\begin{align}\label{eq-MbyFrame}\begin{split}
 \\ 
g^{\bar j i}&=-\rho\frac{\rho^{\bar j i}\rho-\frac{1}{4}\lambda^2 {\rho}^{\bar j i} {\rho}^{ \bar{n} n} +
\frac{1}{4}  \lambda^2 {\rho}^{ \bar{n} i} {\rho}^{\bar{j} n} }{\rho-\frac{1}{4} \lambda^2 {\rho}^{ \bar{n}n} }.
\end{split}\end{align}
We observe that  $g^{ \bar{j} i}$ is $O(\rho^2)$ if $i$ or $j=n$, and is $O(\rho)$ otherwise.

On the other hand  at $P$, by \eqref{eq-DxDy}, 
\begin{equation}\label{eq-DxDy-P}
[A]_{2n\times 2n}=\left(
\begin{array}{ccccc}
1-\frac{\partial^2 \varphi}{\partial y^1 \partial y^1}d & -\frac{\partial^2 \varphi}{\partial y^1\partial y^2} d 
& \ldots & -\frac{\partial^2 \varphi}{\partial y^1\partial y^{2n-1}}d & 0 \\
-\frac{\partial^2 \varphi}{\partial y^2 \partial y^1}d & 1-\frac{\partial^2 \varphi}{\partial y^2 \partial y^2}d& \ldots 
&- \frac{\partial^2 \varphi}{\partial y^2 \partial y^{2n-1}}d & 0 \\
\vdots&\vdots& &\vdots&\vdots\\
-\frac{\partial^2 \varphi}{\partial y^{2n-1} \partial y^1}d & -\frac{\partial^2 \varphi}{\partial y^{2n-1} \partial y^2}d 
& \ldots & 1-\frac{\partial^2 \varphi}{\partial y^{2n-1} \partial y^{2n-1}}d & 0 \\
0& 0& \ldots & 0 & 1
\end{array}
\right),
\end{equation}
and by the proof of Lemma 14.17 in \cite{GT},
\begin{align*}
\frac{\partial^2 d}{\partial x^i \partial x^j}=B_{k j} \frac{\partial \nu_i}{\partial y^k}  
= - \frac{\partial^2 \varphi}{\partial y^i \partial y^j} +O(d),
\end{align*}
where $O(d)$ denotes a term bounded by $Cd$ where $C$ is independent of $P$. If $i$ or $j=2n$,  $\frac{\partial^2 \varphi}{\partial y^i \partial y^j}=0$ since $\varphi$ is independent of $y^{2n}$.
 We have, for $1\leq i, j <2n$,
\begin{align*}
\frac{\partial^2 \rho}{\partial x^i \partial x^j}=e^{-\lambda d} \left(-\lambda \frac{\partial^2 d}{\partial x^i \partial x^j}
+\lambda^2 \frac{\partial d}{\partial x^i} \frac{\partial d}{ \partial x^j}\right)
=\lambda\frac{\partial^2 \varphi}{\partial y^i \partial y^j} +O(d).
\end{align*}
 Thus, we obtain $\rho_{i \bar{j}}$ in terms of derivatives of $\varphi$ and $O(d)$ terms. Moreover, since $\frac{\partial}{\partial x^j} =B_{lj} \frac{\partial}{\partial y^l}$,
\begin{align}
\begin{split}\label{eq-dxdx}
\frac{\partial}{\partial x^i }\frac{\partial}{\partial x^j} &
=B_{p i}\left(\frac{\partial}{\partial y^p }B_{lj} \right)\frac{\partial}{\partial y^l }
+B_{m i}B_{l j}\frac{\partial}{\partial y^m }\frac{\partial}{\partial y^l}\\
&=\left( -\frac{\partial^2 \varphi}{\partial y^i \partial y^j}+O(d)\right) \frac{\partial}{\partial d}+ \frac{\partial^2}{\partial y^i \partial y^j}
+ d\sum_{k, l=1}^{2n}  C_{ij}^{kl}\frac{\partial^2}{\partial y^k \partial y^l} 
+\sum_{\alpha\in I_1} C^\alpha_{i j} \frac{\partial}{\partial y^\alpha}\\
&=\left(-\lambda^{-1}\frac{\partial^2 \rho}{\partial x^i \partial x^j}+O(d)\right) \frac{\partial}{\partial d}+ \frac{\partial^2}{\partial y^i \partial y^j}
+ d\sum_{k, l=1}^{2n}  C_{ij}^{kl}\frac{\partial^2}{\partial y^k \partial y^l}
+\sum_{\alpha \in I_1} C^\alpha_{i j} \frac{\partial}{\partial y^\alpha}, 
\end{split}
\end{align}
where  we used that at $P$ , for $i\in I_1$,
\begin{align*}
B_{p i}\left(\frac{\partial}{\partial y^p }B_{2n, j} \right)\frac{\partial}{\partial d }
&= -B_{pi} B_{2n, m}\frac{\partial A_{mq}}{\partial y^p }  B_{qj}\frac{\partial}{\partial d }\\
&=- \frac{\partial A_{2n, j}}{\partial y^i }\frac{\partial}{\partial d} +O(d)\frac{\partial}{\partial d}\\
&= - \frac{\partial^2 \varphi}{\partial y^i \partial y^j} \frac{\partial}{\partial d}+O(d)\frac{\partial}{\partial d}.
\end{align*}
Combining \eqref{eq-MbyFrame}, \eqref{eq-dxdx} and the fact that $\rho = -\lambda d+O(d^2)$, we obtain that at $P$, 
\begin{align}\begin{split} \label{eq-Mw}
g^{ \bar{j} i} u_{i \bar{j}}&= d^2 (1+O(d)) \left(u_{dd}+ \frac{\partial^2 u}{\partial y^n \partial y^n}\right)
-((n-1)d+ O(d^2 ) )u_d\\
&\qquad+ d^2 \sum_{\beta \in I_2} C_{ \beta d} \frac{\partial u_d}{ \partial y^\beta}
+ d^2 \sum_{\beta \in I_2} C_{ \beta n} \frac{\partial^2 u}{\partial y^\beta \partial y^n }
+d \sum_{ \beta,\gamma  \in I_2} C_{\beta \gamma} \frac{\partial^2 u}{\partial y^\beta \partial y^\gamma }\\
&\qquad +d \sum_{\beta \in I_1} C_{\beta} \frac{\partial u}{\partial y^\beta}.
\end{split}
\end{align}
Assume that near $\overline{QP}$, $\frac{\partial}{\partial y^i} =\sum_j h_{ji}Y_j$, where $h_{ji}$'s are smooth functions. 
 Then  
\begin{align*}
\frac{\partial^2}{\partial y^i \partial y^k} = h_{ji}h_{lk}Y_j Y_l + \frac{\partial h_{lk}}{\partial y^i} Y_l 
\end{align*}
at $P$. 
When  $l=2n$, the term $\frac{\partial h_{lk}}{\partial y^i}  Y_l =0$ since $h_{2n,k}= (\frac{\partial}{\partial y^k}, Y_{2n})_g= \delta_{k, 2n}$. To make sure that no $dY_n^2 u$ term is generated from $d \sum_{ \beta,\gamma  \in I_2} C_{\beta \gamma} u_{\beta \gamma}$, 
we use the fact that $\frac{\partial}{\partial y^n} = \mathcal{J}\frac{\partial}{\partial y^{2n}}=  Y_n$ at $P$. Then by \eqref{eq-Yi-Y2n},  $h_{ji} = \delta_{ji}$  at $P$ if $i$ or $j \in \{n, 2n\}$. 
Combining  these, \eqref{eq-Mw} can be expressed as
\begin{align}\label{eq-Linear-Yi}\begin{split}
g^{\bar k j}u_{j \bar k} &= d^2  (1+O(d)) (u_{dd}+ Y_n^2 u)-((n-1)d+ O(d^2))u_d\\
&\qquad+ d^2 \sum_{\beta \in I_2} a_{ \beta d} Y_\beta u_d 
+d^2 \sum_{\beta\in I_2} a_{\beta n} Y_\beta Y_n u
+d \sum_{ \beta,\gamma  \in I_2} a_{\beta \gamma} Y_\beta Y_\gamma u\\
&\qquad +O(d)\sum_{\beta \in I_1} b_\beta Y_\beta w,
\end{split}
\end{align}
at $P$. Since $P$ is any point in $U$, \eqref{eq-Linear-Yi} holds in $U$ and the coefficients $a_{ij}, b_i$'s are uniformly bounded functions.
To show the regularity of the coefficients, we compare \eqref{eq-Linear-Yi} with \eqref{eq-fij-fk} and apply Taylor expansions of $\tilde a_{ij}(z), \tilde b_k(z)$. Take $a_{nn}$ for example. The comparison shows that $\tilde a_{nn}= d^2(1+O(d))=d^2+d^3 C_1$ for some bounded function $C_1$, while the Taylor expansion  shows that $\tilde a_{nn} = D_0(y') + D_1(y') d + D_2(y') d^2 + R_3(y', d)d^3$ with smooth  $D_0, D_1, D_2$ and $R_3$. 
Then
\begin{align*}
d^2+d^3 C_1 = D_0(y') + D_1(y') d + D_2(y') d^2 + R_3(y', d)d^3.
\end{align*}
Set $d=0$, we have $D_0=0$. Divided by $d$ and set $d=0$, we have $D_1=0$. Divided by $d$ again and set $d=0$, we have $D_2=1$. Finally we get that $C_1 = R_3$, which is smooth. Thus $\tilde a_{nn}=d^2(1+\mathbf{O}(d))$.

The uniform positivity of $(a_{\beta \gamma})_{\beta,  \gamma\in I_2}$ follows from \eqref{eq-MbyFrame} and the fact that $(\rho_{i\bar{j}})>0$ in $\bar{\Omega}$.
\end{proof}

We point out that if we  apply \eqref{eq-inverse}
to compute $a_{ij}, b_k$ in \eqref{eq-fij-fk} directly, then by \eqref{eq-DxDy},  \eqref{eq-B}, and the fact that
\begin{align*}
g^{\bar j i}\rho_i = \mathbf{O}(d^2), \,\,g^{\bar j i}\rho_{\bar j} = \mathbf{O}(d^2),
\end{align*}
we have
 \begin{align*}
g^{\bar j i}u_{i \bar j} = d^2  (1+\mathbf{O}(d)) u_{dd}-((n-1)d+ \mathbf{O}(d^2))u_d+\cdots
\end{align*}
where ``$\cdots$" denotes derivatives with respect to tangential or mixed directions.

Denote
\begin{align}
Q(u) = \log \frac{\det(g_{i \bar j} + u_{i \bar j})}{\det (g_{i\bar j})} - (n+1) u -F.
\end{align}
\eqref{eq-MainEq} is equivalent to $Q(u)=0.$

\begin{lemma}\label{lem-MatrixEstm}
If $u$ is any $C^2$ function satisfying \eqref{eq-equ-metric}, then
\begin{align}\label{eq-ddbar-u2}
|Q(u) -(Lu -F)|\leq C|\partial \bar \partial u|^2_{g}.
\end{align}
In addition, under a local frame $(U, \{Y_i\}_{i=1,\cdots, 2n})$,
 $Q(u) -Lu + F$ is a smooth function of $z$ and the arguments 
\begin{align}\label{eq-C2Term}
d^2 u_{d d}, d^2 Y_n^2 u, d^{\frac{3}{2}} Y_\beta u_{d}, d^{\frac{3}{2}} Y_\beta Y_n u, d^2 Y_\beta u_{d}, d^2 Y_\beta Y_n u,  d Y_{\beta}Y_{ \gamma} u,  d Y_i u,
\end{align}
in $U \cap \Omega$,
where $\beta, \gamma$ $\in I_2,$  $i \in I$.
\end{lemma}

\begin{proof}
First we show \eqref{eq-ddbar-u2}.
Define $g_t$ as a metric such that $(g_t)_{i\bar j} = g_{i \bar j}+ tu_{i \bar j}$  for $t\in [0,1]$. According to \eqref{eq-equ-metric}, $g_t$ is equivalent to $g$. Applying $\log \det(g_{i \bar j} + u_{i \bar j}) -\log \det (g_{i\bar j}) = \int_0^1 \frac{d}{dt} \log\det (g_{i\bar j}+ tu_{i\bar j}) dt$, we derive  that
\begin{align*}
|Q(u) - (Lu - F)| &=\left|   \log \frac{\det(g_{i \bar j} + u_{i \bar j})}{\det (g_{i\bar j})} -g^{\bar j i}u_{i \bar j}\right|\\
&= \left|\int_0^1 g_t^{\bar j i} dt \cdot u_{i \bar j} - g^{\bar j i}u_{i \bar j}
 \right|\\
 &=  \left|\int_0^1\int_0^1 \left(\frac{d}{ds} g_{st}^{\bar j i} \right)ds dt \cdot u_{i \bar j}
 \right|\\
  &=  \left|\int_0^1\int_0^1  t g_{st}^{\bar j p}u_{p \bar q} g_{st}^{\bar q i}ds dt \cdot u_{i \bar j}
 \right|\\
 &\leq C |\partial \bar \partial u|_g^2.
\end{align*}

Moreover, \eqref{eq-equ-metric} implies that $\frac{\det(g_{i\bar j} + u_{i \bar j})}{\det(g_{i \bar j})}$ is bounded  in between $\frac{1}{c^n}$ and $c^n$.
The term
\begin{align}\label{eq-sigma-l}
\frac{\det(g_{i\bar j} + u_{i \bar j})}{\det(g_{i \bar j})} = \det(I_n+g^{\bar k i} u_{i \bar j}) = 1+\sum_{l=1}^n\sigma_l \left(\kappa \left(g^{\bar k i} u_{i \bar j}\right)\right),
\end{align}
where $\kappa \left(g^{\bar k i} u_{i \bar j}\right)$ denotes the eigenvalues of the matrix $(g^{\bar k i} u_{i \bar j})$ and 
\begin{align*}
\sigma_l(\kappa) = \sum_{i_1<\cdots<i_l}\kappa_{i_1}\cdots \kappa_{i_l}.
\end{align*}
 Since the function $\log (1+ s)$ is smooth in $s$ when $1+s>\frac{1}{c^n}$, $Q(u)-Lu+F$ is smooth in $g^{\bar k i} u_{i \bar j}$, and hence smooth in $z,$ $Y_i Y_j u$ and $Y_k u$. We only need to concern the $d$ factors in \eqref{eq-C2Term}. To this end, we show that all $\sigma_l\left(\kappa (g^{\bar k i} u_{i \bar j})\right)$'s, for $l=1,\cdots, n$, are smooth in terms in \eqref{eq-C2Term}.

Denote 
$
\zeta = ( u_{d d}, Y_n^2 u, Y_\beta u_{d}, Y_\beta Y_n u,   Y_{\beta}Y_{ \gamma} u,  Y_i u).
$
Then
\begin{align}\label{eq-sigma-l-zeta}
\sigma_l\left(\kappa (g^{\bar k i} u_{i \bar j})\right) = \sum_{|\alpha|=l}C_\alpha\zeta^\alpha,
\end{align}
where $\alpha$ is a multi-index and $C_\alpha$'s are smooth functions.
 We compute $C_\alpha$ at a point $P \in \Omega$.
 Let $Q \in \partial \Omega$ such that $d(P) = dist(P,Q)$. Since \eqref{eq-sigma-l} is independent of choice of complex coordinates, we
 compute $\sigma_l$ under the complex coordinates $z^i_Q$ as defined in Section \ref{sec-Frame}.
 According to 
\eqref{eq-MbyFrame},  $g^{ \bar{j} i}$ is $O(d^2)$ at $P$ if $i$ or $j=n$, and is $O(d)$ otherwise.  Let $\frac{\partial}{\partial x_i} = \sum_{k=1}^{2n} f_{ki}Y^Q_k$ where $f_{ki}$'s are smooth functions and $f_{ki}= \delta_{ki}$ at $P$. Then at $P$,
\begin{align*}
\frac{\partial}{\partial x^i }\frac{\partial}{\partial x^j}  = Y^Q_i Y^Q_j + Y^Q_i(f_{lj})Y^Q_l,
\end{align*}
where $ Y^Q_i(f_{lj})$ is  uniformly bounded.
For $i =1, \cdots, n$ without summation, if $k\neq n$,
\begin{align}\label{eq-g-u-1}
g^{\bar k i} u_{i \bar j} &=O(d) (Y^Q_{i+n}-\sqrt{-1}Y^Q_i)(Y^Q_{j+n}+\sqrt{-1}Y^Q_j)u+ \sum_{l=1}^{2n}O(d) Y^Q_l u.
\end{align}
When $i=n$, the $O(d)$ terms in \eqref{eq-g-u-1} should be $O(d^2)$. 
If $k=n$,
  \begin{align}\label{eq-g-u-2}
g^{\bar n i} u_{i \bar j} &=O(d^2)(Y^Q_{i+n}-\sqrt{-1}Y^Q_i)(Y^Q_{j+n}+\sqrt{-1}Y^Q_j) u+ \sum_{l=1}^{2n}{O}(d^2) Y^Q_l u.
\end{align}
We conclude that $g^{\bar k i} u_{i \bar j}$ is linear in terms in \begin{align}\label{eq-C2Term-Q}
d^2 u_{d d}, d^2 Y_n^2 u, d^{\frac{3}{2}} Y^Q_\beta u_{d}, d^{\frac{3}{2}}Y^Q_\beta Y_n  u, d^2 Y^Q_\beta u_{d}, d^2  Y^Q_\beta Y_n u,  d Y^Q_{\beta}Y^Q_{ \gamma} u,  d Y^Q_i u,
\end{align}
with uniformly bounded coefficients, except for the case $k\neq n, i\neq n$ and $j=n,$ since we need a $d^\frac{3}{2}$ or $d^2$ factor rather than a $d$ factor  for $Y^Q_iY_{2n}   u, Y^Q_i Y_n u, Y^Q_{i+n} Y_{2n} u$ and $Y^Q_{i+n} Y_n u$ when $i\neq n$. 

 $\sigma_l\left(\kappa (g^{\bar k i} u_{i \bar j})\right)$ equals the sum of determinants of
submatrices, which are of form
\begin{align}\label{eq-sumatrix}
(g^{\bar k i} u_{i \bar j})_{k, j \in I^-},
\end{align}
where  $I^- \subset I$ and $|I^-|=l$. 
We use the Leibniz formula to compute the determinants.  
   In each product of $l$ entries of \eqref{eq-sumatrix}, if there is a factor $g^{\bar k i} u_{i \bar n}$ where $k\neq n$, then there is a term  $g^{\bar n i} u_{i \bar \gamma}$ for some $\gamma \neq n$ and rest terms are of form $g^{\bar \alpha i} u_{i \bar \beta}$ where $\alpha, \beta \neq n$. Applying \eqref{eq-g-u-1}, \eqref{eq-g-u-2}, analyzing the four cases that whether $i=n$ and whether $ l =n$, we conclude that, for $k\neq n$,
$
 g^{\bar k i} u_{i  \bar n}
 \cdot
 g^{\bar n l} u_{l \bar \gamma} 
$
is quadratic in terms in \eqref{eq-C2Term-Q} with uniformly bounded coefficients.
For example, when $i\neq n, l=n$,
 \begin{align*}
 g^{\bar k i} u_{i  \bar n}
 \cdot
 g^{\bar n l} u_{l \bar \gamma} =C(z) \cdot d^3 u_{i  \bar n}u_{n \bar \gamma},
 \end{align*}
which expands to terms like $d^\frac{3}{2} Y^Q_{i}Y_n u \cdot d^\frac{3}{2} Y^Q_\gamma Y_n u,$ $d^\frac{3}{2} Y^Q_{i} u_d \cdot d^\frac{3}{2} Y^Q_\gamma  u_d$  and $d^2 Y^Q_{i}Y_n u \cdot d Y^Q_l u.$

It is easy to check that at $P$, terms in 
\eqref{eq-C2Term-Q} can be converted to terms in \eqref{eq-C2Term} with uniformly bounded coefficients.  So  $\sigma_l\left(\kappa (g^{\bar k i} u_{i \bar j})\right)$ is a polynomial of degree $l$ in terms in \eqref{eq-C2Term} with uniformly bounded coefficients in $U$.
Comparing with \eqref{eq-sigma-l-zeta},
and applying the Taylor expansion of $C_\alpha$, we conclude that  
the uniformly bounded coefficients are actually smooth in $z$. This trick is also used in the proof of Lemma \ref{lem-Lu}.
\end{proof}

Given a $C^2$ function $\psi$, let $g_\psi$ be the metric defined as
\begin{align}\label{eq-g-psi}
(g_\psi)_{i\bar j} = g_{i\bar j}+ \psi_{i\bar j}.
\end{align}
$g_\psi$ is equivalent to $g$ when $d$ is small.  Denote
\begin{align*}
H_{i\bar{j}}=\rho_{i\bar{j}}-\rho \psi_{i \bar{j}}.
\end{align*}
Then,
\begin{align}\label{eq-gpsi-inverse}
g^{\bar j i}_\psi&=-\rho\left(H^{ \bar{j} i}
+\frac{H^{ \bar{l} i} H^{\bar{j}m}{\rho}_m \rho_{\bar{l}}}{{\rho}-{H}^{ \bar{q}p}\rho_{\bar{q}}{\rho}_p}\right),
\end{align} 
which equals
\begin{align}
-\rho\left(\frac{H^{\bar{j} i}\rho-\frac{1}{4}\lambda^2 H^{\bar{j} i} H^{\bar{n} n} +
\frac{1}{4}  \lambda^2 H^{\bar{n} i} H^{\bar{j} n} }{\rho-\frac{1}{4} \lambda^2 H^{ \bar{n} n}}\right),
\end{align}
when $d_i = \delta_{in}/2. $ We point out that replacing $g$ by $g_\psi$, \eqref{eq-MN} and \eqref{eq-ddbar-u2} are still correct. So we can do formal computation as below.
 
 For $k =0, \cdots, n$, we want to find a polynomial of degree $k$ in principal coordinates,
 \begin{align*}
 u_{(k)} = c_0 (y') + c_1 (y')d +\cdots + c_k (y')d^k, 
\end{align*}
 such that 
$$Q(u_{(k)} )= O(d^{k+1}).$$
We start with
$$u_{(0)}=-\frac{1}{n+1}F(y', 0).$$
In fact,  \eqref{eq-MN}, \eqref {eq-inverse} and \eqref{eq-ddbar-u2} implies that
\begin{align*}
Q(u_{(0)}) = Lu_{(0)} - F + (Q(u_{(0)})  - Lu_{(0)}+F) =O(d)
\end{align*}
only when $c_0(y')=-\frac{1}{n+1}F(y', 0)$. Here
\begin{align*}
|Q(u_{(0)})  - Lu_{(0)}+F| &\leq C|i\partial \bar \partial u_{(0)}|^2_g= C g^{\bar j p} u_{(0), p \bar q}g^{\bar q i}u_{(0), i \bar j} \leq Cd^2.
\end{align*}

Inductively, if $\psi = u_{(k-1)}$ is known, then 
$
Q(\psi) = O(d^k),
$
where the $O(d^k)$ term is a smooth function that can be expressed as $f_k(y' )d^k + O(d^{k+1})$, which follows that
\begin{align*}
\log \frac{\det(g_{i \bar j} + \psi_{i \bar j})}{\det (g_{i\bar j})} - (n+1) \psi -F = f_k(y' )d^k + O(d^{k+1}).
\end{align*}

 Thus
\begin{align*}
&Q(\psi+ c_k d^k) \\
&= \log \frac{\det((g_\psi)_{i \bar j} + (c_k d^k)_{i \bar j})}{\det (g_{i\bar j})} - (n+1) (\psi+ c_kd^k) -F \\
&= \log \frac{\det((g_\psi)_{i \bar j} + (c_k d^k)_{i \bar j})}{\det ((g_\psi)_{i\bar j})} + \log \frac{\det(g_{i \bar j} + \psi_{i \bar j})}{\det (g_{i\bar j})} - (n+1) (\psi+ c_kd^k) -F \\
&= \log \frac{\det((g_\psi)_{i \bar j} + (c_k d^k)_{i \bar j})}{\det ((g_\psi)_{i\bar j})} - (n+1)  c_kd^k + f_k(y' )d^k + O(d^{k+1}).
\end{align*}
Applying  \eqref{eq-MN} and \eqref{eq-ddbar-u2} to the first term, with $g$ replaced by $g_\psi,$ we derive that
\begin{align}\label{eq-Q-formal}
Q(\psi+ c_k d^k)  = (k-n-1)(k+1)c_k(y')  d^{k} +f_k(y') d^k+ O(d^{k+1}),
\end{align}
where the $O(d^{k+1})$ term also depends on $c_k$. Here we need to use
\begin{align*}
|i\partial \bar \partial (c_k d^k)|^2_{g_\psi}\leq C g^{\bar j p} (c_k d^k)_{p \bar q}g^{\bar q i}(c_k d^k)_{i \bar j}\leq Cd^{2k}.
\end{align*}
To derive this, at any point $P$ near $\partial \Omega$, we compute under  coordinates $z_Q$ where $Q \in \partial \Omega$ and $d(P) = dist(P,Q)$.
By \eqref{eq-Q-formal},  for $k=1,\cdots, n$, we can solve out a unique $c_k(y')$, such that 
$Q(\psi+ c_k d^k) = O(d^{k+1})$.

However, we cannot continue with $k ={n+1}$. 
Fefferman \cite{Fefferman1976} pointed out that a logarithmic term $ d^{n+1}\log d$ is needed to find  
\begin{align*}
u_{(n+1, 1)}=  c_0 +c_1 d+\cdots + c_n  d^{n}+c_{n+1,1} d^{n+1}\log d
\end{align*}
such that
\begin{align*}
Q(u_{(n+1,1)})= O(d^{n+2}\log d).
\end{align*}
In fact,
\begin{align*}
Q(u_{(n)}+ c_{n+1, 1} d^{n+1}\log d)  =-nc_{n+1,1}(y')  d^{k} +f_{n+1}(y') d^k+ O(d^{k+1}).
\end{align*}
We can solve out  $c_{n+1, 1} = f_{n+1}(y')/n.$ In addition, for arbitrary $C^2$ function $c_{n+1}(y')$, 
\begin{align*}
Q(u_{(n)}+ c_{n+1, 1} d^{n+1}\log d + c_{n+1} d^{n+1})  = O(d^{k+1}).
\end{align*}
We regard $c_{n+1} $ as a nonlocal coefficient. 
Given any smooth $c_{n+1} $,  for  $k>n+1$, we can solve out a unique
\begin{align}\label{eq-u-k}
u_{(k)}=\sum_{i=0}^nc_i(y')d^i+\sum_{i=n+1}^k\sum_{j=0}^{N_i} c_{i,j}(y')d^i (\log d)^j,
\end{align}
such that
\begin{align*}
Q(u_{(k)})=O(d^{k+1}).
\end{align*}
Kichenassamy and Littman \cite{KL:1}, \cite{KL:2} shows that if $\partial \Omega$ is analytic, then given any analytic $c_{n+1}$, \eqref{eq-u-k} converges to a solution to \eqref{eq-MainEq}-\eqref{eq-F} near $\partial \Omega$.

\section{$C^{2, \alpha}_g$ Estimate of $u - u_{(1)}$}\label{sec-BasicEstimates}
Let $u$ be a solution to \eqref{eq-MainEq}-\eqref{eq-F},
and let $\psi= u_{(1)}$ be the linear function derived from the formal computation in Section \ref{sec-FormalComputation}.
Denote 
$$v = u-\psi,$$ and rewrite \eqref{eq-MainEq} as 
\begin{align}\begin{split}\label{eq-MA1}
\log \frac{\det(g_{i \bar{j}}+\psi_{i \bar{j}}+v_{i \bar{j}})}{\det(g_{i \bar{j}}+\psi_{i \bar{j}})}&=(n+1)(\psi+v)+F -\log \frac{\det(g_{i \bar{j}}+\psi_{i \bar{j}})}{\det(g_{i \bar{j}})}\\&=(n+1)v+ F_1 d^2,
\end{split}
\end{align} 
where $F_1$ is $C^{1}$ up to $\partial \Omega$.

\begin{lemma}\label{lemma-basic-estimate-v} For any $\epsilon \in (0,1)$, 
$$|v|\le C|\rho|^{1+\epsilon}.$$
\end{lemma} 

\begin{proof} 
Take a small positive $r$, such that 
$(g_{i\bar j} + \psi_{i \bar j})> \frac{1}{2}g_{i\bar j}$ in
\begin{align*}
\Omega_r = \{z\in \Omega : 0<d(z)<r\}.
\end{align*}
We first show that  $|v|\le C|\rho|$ in $\Omega_r$. 
By Theorem \ref{thm-IntEstm}, $v$ is uniformly bounded. 
Take a constant $b$ to be determined later such that  
$$b>\frac{\|v\|_{C^0}}{1-e^{-\lambda r}}.$$
Then, $v+ b\rho < 0$ when $d = r$. Assume that the supremum of  $v+ b\rho$ is nonnegative in $\Omega_r.$

Consider the function
\begin{align}\label{eq-sigma-Max}
\sigma=v+ b\rho.
\end{align}
 Applying  Proposition 1.6 in Cheng-Yau \cite{ChengYau1980CPAM}, 
there is a sequence of points $\{P_i\} \subset \Omega_r$ such that, as $i\rightarrow \infty$,
\begin{align}\label{eq-max-CY}
\sigma(P_i) \rightarrow \sup_{\Omega_r} \sigma,\,\, {\limsup} (\sigma_{p \bar q}(P_i))\leq 0.
\end{align}
As $(\rho_{i\bar{j}})>0,$ $(\sigma_{i \bar j}) = ( v_{i \bar j} +b\rho_{i \bar j}) > (v_{i \bar j})$. Then
$${\limsup}  \frac{\det(g_{i \bar{j}}+\psi_{i \bar{j}}+v_{i \bar{j}})}{\det(g_{i \bar{j}}+\psi_{i \bar{j}})} (P_i)
\le {\limsup}  \frac{\det(g_{i \bar{j}}+\psi_{i \bar{j}}+\sigma_{i \bar{j}})}{\det(g_{i \bar{j}}+\psi_{i \bar{j}})} (P_i)\leq 1,$$
by \eqref{eq-max-CY}.
Here $(g_{i \bar{j}}+\psi_{i \bar{j}}+\sigma_{i \bar{j}}) = (g_{i \bar j} + u_{i \bar j} + b \rho_{i \bar j})$ is always positive.
Applying \eqref{eq-MA1}, 
$${\limsup}\ e^{(n+1)v+F_1 d^2} (P_i)\leq 1.$$
By \eqref{eq-sigma-Max}, 
$${\limsup} ((n+1)\sigma-(n+1)b\rho+F_1 d^2)(P_i)\leq 0.$$ 
Take $b$ large such that $-(n+1)b\rho+F_1 \rho^2>0$. Therefore, 
\begin{align*}
\sup_{\Omega_r} \sigma = \lim \sigma(P_i) \leq 0. 
\end{align*}
Thus $v\leq -b\rho$. We point out that if $\sigma$ has a maximum at $P \in \Omega_r$, then we take $P_i = P$ for all $i,$ so that the above argument turns to a classical maximum principle argument. 

For the other direction, we set $\sigma^+= v- b\rho$ and apply a similar method. There could be an issue that $g_{i \bar j} + \psi_{i\bar j} +\sigma^+_{i \bar j} <0$ if $b$ is too large. However,  according to the proof of Proposition 1.6 in \cite{ChengYau1980CPAM}, we have a sequence of points $\{P_i\}$ such that
\begin{align*}
(\sigma^+_{p \bar q}(P_i)) \geq - c(\sigma^+(P_i^*) - \inf_{\Omega_r} \sigma^+)(g_{p \bar q}(P_i)),
\end{align*}
for some constant $c>0$
and some $P_i^*$ in $\Omega_r$ satisfying that $\lim  \sigma^+(P^*_i) = \inf_{\Omega_r} \sigma^+$.  Thus when $i$ is large, 
$(g_{i \bar j} + \psi_{i\bar j} +\sigma^+_{i \bar j})(P_i) > 0$.

After getting $|v|\leq C|\rho|$, the rest is a standard maximum principle argument. In fact, we use test functions
\begin{align*}
M_{\pm} = \pm a d^{1+\epsilon},
\end{align*}
and apply the maximum principle to equation \eqref{eq-MA1} on $\Omega_r$, using Lemmas \ref{lem-Lu} and \ref{lem-MatrixEstm}, 
with $g$ replaced by $g_\psi$. Take $M_+$ for example.  \eqref{eq-equ-metric} holds for $u= M_+$ in $\Omega_r$ if
\begin{align}\label{eq-equiv-to-g}
a \in (0, \delta r^{-\epsilon-1}),
\end{align}
for some small $\delta>0$ which does not depend on $r$. 
 $v=M_+ = 0$ on $\partial \Omega$. And $v<M_+$ on $\{d=r\}$ if
\begin{align}\label{eq-a-b}
a> b(1-e^{-\lambda r})/r^{1+\epsilon}.
\end{align} 
When $r$ is sufficiently small, we can set $a= Ar^{-\epsilon}$ for some constant $A>0$ so that  \eqref{eq-equiv-to-g}, \eqref{eq-a-b}  hold. The choice of $A$ does not depend on $r$.
  When $0<d<r$,
\begin{align*}
&\frac{\det(g_{i \bar{j}}+\psi_{i \bar{j}}+(M_+)_{i \bar{j}})}{\det(g_{i \bar{j}}+\psi_{i \bar{j}})} -(n+1)M_+ - F_1d^2\\
&\leq g_{\psi}^{\bar k j} (M_+)_{j \bar k}-(n+1)M_+  + C|\partial \bar \partial M_+|_{g}^2 - F_1d^2\\
&\leq a\left((2+\epsilon)(\epsilon-n)d +C_1 d^2\right) + C_2 a^2 d^{2+2\epsilon}  - F_1d^2\\
&<0,
\end{align*}
if 
\begin{align*}
Ar^{-\epsilon}\left((2+\epsilon)(\epsilon-n)+C_1 d+C_2 (Ar^{-\epsilon}) d^{1+2\epsilon}\right) -F_1 d<0,
\end{align*}
which holds when $r$ is small and the choice of such $r$ only depends on $C_1, C_2, \delta$ and $A$. So we conclude that $v<M_+$ in $\Omega_r.$
\end{proof}

Assume $O \in \partial \Omega$. Take a small $\delta>0$ and assume that $\Gamma = B(O, \delta) \cap \partial \Omega$ is smooth. Denote $(g_t)_{i\bar j} = g_{i\bar j}+\psi_{i\bar j}+ tv_{i \bar j}$. 
In a domain  $B(P,\eta) \subseteq \Gamma \times \{0<d<\delta\},$
 \eqref{eq-MA1} can be written as
\begin{align}\label{eq-gt-v}
\int_0^1 g_t^{\bar j i} dt \cdot v_{i \bar j} -(n+1)v =F_1d^2.
\end{align}
Here $g_t$ is uniformly equivalent to $g$ and its inverse equals  \eqref{eq-gpsi-inverse}, if we denote $H_{i\bar j} = \rho_{i\bar j} - \rho (\psi_{i\bar j} + tv_{i \bar j})$. 
Applying \eqref{eq-C5a}, the standard theory of elliptic equations indicates that $v\in C^\infty_g.$ Applying $|v| = O(d^{1+\epsilon})$, the Schauder estimates imply that
\begin{align}\label{eq-v-Schauder}
\|v\|_{C^{2,\alpha}_g(B(P,\eta))} \leq C d(P)^{1+\epsilon}.
\end{align}
By Lemma \ref{lem-NormY}, it further implies that $C^{\alpha}_g(B(P,\eta))$ norms of terms listed in \eqref{eq-list-Y2v} are bounded by $Cd(P)^{1+\epsilon}$.

\section{Tangential Smoothness}\label{sec-TE}

In this section, we derive the estimates of $v$ along tangential directions. Assume that  an open portion $\Gamma$ of $\partial \Omega$ is smooth. Then there is a $R>0$ such that $d$ is smooth in
$\Gamma \times [0, R) := \{(y', d): y' \in \Gamma, 0\leq d< R\}$.
Under principal coordinates $(y', d), $ denote
\begin{align}\label{eq-B'}
B'(P,\eta) = \{(y', 0): (y', d) \in B(P,\eta) \text{ for some } d\}.
\end{align}

Given any $C^2$ function $w$, denote
\begin{align*}
\xi_w = (d^2 w_{d d}, d^2 Y_n^2 w, d^{\frac{3}{2}} Y_\beta w_{d}, d^{\frac{3}{2}}Y_\beta Y_n w, d^2 Y_\beta w_{d}, d^2 Y_\beta Y_n w,  d Y_{\beta}Y_{ \gamma} w,  d Y_l w),
\end{align*}
 where $\beta ,\gamma \in I_2$, and $l \in I.$ See \eqref{eq-C2Term}.

Notice
\begin{align*}
 \frac{\det((g_\psi)_{i\bar j} + v_{i \bar j})}{\det((g_\psi)_{i \bar j})} =   \det(I_n+g_\psi^{\bar k i} v_{i \bar j}) = 1+s,
\end{align*}
where $s$ is given by \eqref{eq-sigma-l}.

Let $g_\psi$ be defined as  \eqref{eq-g-psi} with $\psi = u_{(1)}.$
A similar argument as in the proof of Lemma \ref{lem-Lu} implies that
\begin{align*}
\sigma_1\left(\kappa (g_\psi^{\bar k i} v_{i \bar j})\right) &= d^2  (1+\mathbf{O}(d)) (v_{dd}+ Y_n^2 v)-((n-1)d+  \mathbf{O}(d^2))v_d\\
&\qquad+ d^2 \sum_{\beta \in I_2} C_{ \beta d} Y_\alpha v_d +d^2 \sum_{\beta\in I_2} C_{\beta n} Y_\beta Y_n v
+d \sum_{ \beta,\gamma  \in I_2} C_{\beta \gamma} Y_\beta Y_\gamma v\\
&\qquad +d\sum_{\beta \in I_1} C_{\beta} Y_\beta v.
\end{align*}
According to the proof of Lemma \ref{lem-MatrixEstm}, $\sigma_l \left(\kappa \left(g_\psi^{\bar k i} v_{i \bar j}\right)\right)$ is a polynomial of degree $l$ in $\xi_v$
 with  coefficients smooth in $z$. We denote
\begin{align}\label{eq-F2}
F_2 = \log (1+s) -\sigma_1\left(\kappa \left(g_\psi^{\bar k i} v_{i \bar j}\right)\right).
\end{align}
Then \eqref{eq-MA1} can be written as
\begin{align}\begin{split}\label{eq-MainEqUnderFrame-v}
&d^2  (1+\mathbf{O}(d)) (v_{dd}+ Y_n^2 v)-((n-1)d+  \mathbf{O}(d^2))v_d\\
&\qquad+ d^2 \sum_{\beta \in I_2} a_{ \beta d} Y_\beta v_d +d^2 \sum_{\beta\in I_2} a_{n \beta} Y_\beta Y_n v
+d \sum_{ \beta,\gamma  \in I_2} a_{\beta \gamma} Y_\beta Y_\gamma v\\
&\qquad +d\sum_{\beta \in I_1} b_{\beta} Y_\beta v - (n+1)v= F_1 d^2+F_2,
\end{split}
\end{align}
where $a_{ij}, b_{i}$'s are smooth functions of $z$, and $F_2$ is given by \eqref{eq-F2}.

\begin{lemma}\label{lem-Yi-v}
Assume that  an open portion $\Gamma$ of $\partial \Omega$ is smooth. 
Then for any open subset  $\Gamma'$ of $\Gamma$ such that $\bar \Gamma' \subseteq \Gamma$, and  any $B(P, \eta)) \subseteq \Gamma' \times (0, R),$
\begin{align}\label{eq-Yi-v}
\|Y_\alpha v\|_{C^{2, \alpha}_g (B(P, \eta))}\leq C d(P)^{1+\epsilon},
\end{align}
if  $Y_\alpha \neq Y_{2n}$. Here $C$ is independent the choice of $P.$
\end{lemma}
\begin{proof}
According to the assumption,  $\rho = e^{-\lambda d}-1$ is smooth in $\Gamma\times [0, R)$. So is $F$ in \eqref{eq-F}. We apply  higher order derivative estimates to \eqref{eq-gt-v} to derive that, for any $k\in \mathbb{N}$,
\begin{align}\label{eq-v-higher}
\|v\|_{C^{k,\alpha}_g(B(P,\eta))} \leq C d(P)^{1+\epsilon}.
\end{align}
So for any fixed small positive $r$, $v$ is smooth in $\Gamma \times [r, R)$. We only only to prove the lemma when $d\in (0,r)$.

Let $w = Y_\alpha v.$ First we assume that $Y_\alpha$ commutes with $Y_i$ for all $i=1,\cdots, 2n.$
We compute $Y_\alpha F_2$ as
\begin{align}\label{eq-Ya-F2}
Y_\alpha \left(   \log (1+s) -g_\psi^{\bar j i}v_{i \bar j}\right)=\frac{Y_\alpha \left( s -g_\psi^{\bar j i}v_{i \bar j}\right) - s Y_\alpha \left( g_\psi^{\bar j i}v_{i \bar j}\right) }{1+s},
\end{align}
where $s -g_\psi^{\bar j i}u_{i \bar j} = \sum_{l=2}^n\sigma_l \left(\kappa (g_\psi^{\bar k i} v_{i \bar j})\right).$ 
According to the proof of Lemmas \ref{lem-Lu}, \ref{lem-MatrixEstm}, 
\begin{align*}
Y_\alpha \left( s -g_\psi^{\bar j i}v_{i \bar j}\right) - s Y_\alpha \left( g_\psi^{\bar j i}v_{i \bar j}\right) = F_3(x, \xi_v) + F_4(x, \xi_v) \cdot \xi_w,
\end{align*}
where $F_3, F_4$ are both polynomials in $\xi_v$ with smooth coefficients in $x$. In addition, $F_3$ starts with degree 2 in  $\xi_v$ and  $F_4$ starts with degree 1 in $\xi_v$. We apply $Y_\alpha$ to \eqref{eq-MainEqUnderFrame-v},  multiply the resulting equation by $1+s$ and keep only terms containing $w$ to the left-hand side to derive
\begin{align}\begin{split}\label{eq-MainEqUnderFrame-w}
&d^2  (a_{2n, 2n}w_{dd}+a_{n, 2n}Y_n w_d+ a_{nn}Y_n^2 w)+d \sum_{ \beta,\gamma  \in I_2} a_{\beta \gamma}  Y_\beta Y_\gamma w\\
&\qquad +  \sum_{\beta \in I_2}(d^{\frac{3}{2}} a_{\beta, 2n} + d^2 c_{\beta, 2n}) Y_\beta  w_d +  \sum_{\beta \in I_2}(d^{\frac{3}{2}} a_{\beta n} + d^2 c_{\beta n})  Y_\beta Y_n w \\
&\qquad +d\sum_{i \in I} b_{\beta} Y_i w - (n+1)w= \tilde F d^2 +\tilde G,
\end{split}
\end{align}
where $a_{ij}, b_i, c_{ij}$'s are  polynomials in $d, \xi_v$ and $\tilde F, \tilde G$ are  polynomials in $ \xi_v$, all with coefficients smooth in $z$. In addition, when $(d, \xi_v) = (0, 0)$,
\begin{align}\begin{split}\label{eq-aij-bi-v}
a_{2n, 2n}(z, 0, 0) &= 1, \qquad a_{nn}(z, 0, 0) =1,\\
a_{\beta, 2n}(z, 0, 0) & = 0, \qquad a_{\beta n}(z, 0, 0)  = 0\\
b_n(z, 0, 0) & = -(n-1),
\end{split}
\end{align}
and $(a_{\beta \gamma}(z,0,0) )_{\beta, \gamma \in I_2}$ is uniformly positive definite. Futhermore, $\tilde G= \sum_{|\alpha|=2}^n C_\alpha (z) \xi_v^\alpha$.

We rewrite \eqref{eq-MainEqUnderFrame-w} as
\begin{align}\label{eq-Lv-FG}
L_v w = \tilde F d^2 +\tilde G.
\end{align}
\eqref{eq-v-Schauder} and Lemma \ref{lem-NormY} implies that $C^\alpha_g$ norms of $s, F_4$ and $\tilde Fd^2+\tilde G$ in $B(P, \eta)$ are bounded by $Cd^{1+\epsilon}$.

Let $\delta>0$ be a  small number. We consider any $P \in \Gamma \times (0, r)$ such that \begin{align}\label{eq-cond-P}
 |y'(P) - y' (\partial \Gamma)|>\delta.
\end{align}  
Then $P$ is away from $\partial \Gamma \times (0,r).$
We apply the maximum principle to \eqref{eq-Lv-FG} with test functions
\begin{align}\label{eq-test-v}
M_\pm(y', d) = \pm  d^\epsilon(ad + b|y'- y'(P)|^2),
\end{align}
in the cylinder
\begin{align}\label{eq-G-P-delta}
G_{P, \delta, \delta^2} = \{(y', d) : |y' - y'(P)| <\delta, 0<d<\delta^2\}
\end{align}
to derive that
$
|w| \leq M_+.
$
In fact,
\eqref{eq-v-Schauder} implies  that $|Y_\alpha v| \leq Cd^\epsilon$ for $\alpha \in I_1$. Hence
 at $\partial G_{P, \delta, \delta^2},$ if $a, b > C\delta^{-2}$,
$
|Y_\alpha v| \leq M_+.
$
And in  $ G_{P, \delta, \delta^2},$
\begin{align}\begin{split}\label{eq-MaxP}
&L_v M_+ - \tilde F d^2 -\tilde G\\
&\leq a(\epsilon+2)(\epsilon-n)d^{1+\epsilon} +b (\epsilon+1)(\epsilon-n-1)d^{\epsilon}|y'- y'(P)|^2\\
&\qquad+ C_1 bd^{1+\epsilon} + C_2 a d^{2+\epsilon}+ C_3 d^{1+\epsilon} \\
&<0,
\end{split}
\end{align}
if $a>>b>>C_3$ and $r$ is small.
Then we derive that $|w| \leq M_+$ in $G_{P, \delta, \delta^2}.$
When $y' = y'(P)$, 
\begin{align}\label{eq-w-max}
|w| \leq a d^{1+\epsilon}.
\end{align}
Take another open subset $\Gamma'' $ of $\Gamma$ such that
\begin{align*}
 \bar\Gamma' \subseteq \Gamma'', \quad \bar\Gamma'' \subseteq \Gamma.
\end{align*}
Since the choice of $P$ is arbitrary as long as \eqref{eq-cond-P} holds, \eqref{eq-w-max} holds in $\Gamma'' \times [0, R]$ if $\delta$ is small enough.

Using \eqref{eq-B'}, we consider any $P\in \Gamma \times [0, r]$ such that
$
 B'(P, \eta) \subseteq \Gamma''.
$
 According to Lemmas \ref{lem-smooth-dP}, \ref{lem-Yi-xj} and  \eqref{eq-z-theta},  $L_v$ equals $\Delta_{g_\psi}$ plus perturbation  terms that can be written as
\begin{align*}
\tilde a_{i j}(\vartheta) \frac{\partial^2}{\partial \vartheta_i \partial \vartheta_j} +\tilde b_i(\vartheta) \frac{\partial}{\partial \vartheta_i },
\end{align*}
where $C^\alpha_g(B(P,\eta))$ norms of $\tilde a_{ij}, \tilde b_i$'s are bounded by $Cd$.
Then \eqref{eq-w-max} and the Schauder estimates applied to \eqref{eq-Lv-FG} under the coordinates $\vartheta$ in $B(P, \eta)$ imply that, \begin{align}\label{eq-Schauder-w}
&\|w\|_{C^{2, \alpha}_g (B(P,\eta))}\leq C d(P)^{1+\epsilon}.
\end{align}

Secondly, we allow that $Y_\alpha$ does not commute with some other $Y_i$'s. 
We need to control the terms generated by $Y_\alpha \xi_v - \xi_w$ that are not in $\xi_v$.

For the case $Y_\alpha = Y_n$, by Lemma \ref{lem-Lie-Y}, when $\beta, \gamma\in I_2,$
\begin{align*}
Y_n (d Y_\beta Y_\gamma v) =  d Y_\beta  Y_n Y_\gamma  v  + d \sum_{\theta \in I_1} f_\theta Y_\theta Y_\gamma v.
\end{align*}
When $\theta=n$, the term $d Y_n Y_\gamma v$ on the right-hand side is not contained in $\xi_v.$  However,
\begin{align*}
Y_n (d Y_\beta Y_\gamma v) &=dY_\beta  Y_\gamma w+d \sum_{\theta \in I_1} Y_\beta (g_\theta Y_\theta v)
+ d\sum_{\theta \in I_1}  f_\theta Y_\theta Y_\gamma v\\
&=dY_\beta  Y_\gamma w+d \sum_{\theta \in I_2} Y_\beta (g_\theta Y_\theta v) + d Y_\beta (g_n w)
\\
&\qquad + d\sum_{\theta \in I_2}  f_\theta Y_\theta Y_\gamma v
+ d  f_n Y_n Y_\gamma v\\
&=dY_\beta  Y_\gamma w+d \sum_{\theta \in I_2} Y_\beta (g_\theta Y_\theta v)
+dg_n Y_\beta w + d (Y_\beta g_n) Y_n v \\
&\qquad + d\sum_{\theta \in I_2}  f_\theta Y_\theta Y_\gamma v + df_n Y_\gamma w + d\sum_{\theta \in I_2}  (f_n h_\theta )Y_\theta v,
\end{align*}
where $d Y_\beta  Y_\gamma w, d g_n Y_\beta  w, df_n Y_\gamma w$ 
are linear in $\xi_w$ and the rest are linear in $\xi_v$. Similarly, by Lemma \ref{lem-Lie-Y}, when $\beta 
\in I_2,$
\begin{align*}
Y_n(d^\frac{3}{2} Y_\beta Y_n v) & = d^\frac{3}{2} Y_n Y_\beta w\\
&=  d^\frac{3}{2} Y_\beta Y_n  w + d^\frac{3}{2}  \sum_{\theta \in I_2} f_\theta Y_\theta w + d^\frac{3}{2} f_n Y_n w,\\
Y_n(d^\frac{3}{2} Y_\beta  v_d)&=d^\frac{3}{2} Y_\beta  Y_n v_d +d^\frac{3}{2} \sum_{\theta \in I_2} f_\theta Y_\theta v_d +d^\frac{3}{2} f_n Y_n v_d\\
&=d^\frac{3}{2} Y_\beta  w_d +d^\frac{3}{2} Y_\beta \sum_{\theta \in I_1} f_\theta Y_\theta v
 +d^\frac{3}{2} \sum_{\theta \in I_2} f_\theta Y_\theta v_d \\
 &\qquad+d^\frac{3}{2} f_n  w_d + d^\frac{3}{2} f_n  \sum_{\theta\in I_1} g_\theta Y_\theta v.
\end{align*}
All terms on the right-hand side are  linear in $\xi_w$ or $\xi_v$ but may have extra $d^\frac{1}{2}$ factors. We point out that the extra $d^{\frac{1}{2}}$ factors can be absorbed by other terms in the equation.

We apply $Y_n$ to \eqref{eq-MainEqUnderFrame-v}, switch orders of $Y_i$'s if necessary, multiply both hand sides by $1+s$ and keep only terms containing $w$ to the left-hand side to derive an equation of form \eqref{eq-MainEqUnderFrame-w}. Schauder estimates imply \eqref{eq-Schauder-w}.

For the case $Y_\alpha \neq Y_n$, again we have to take care of exchanging orders of vector fields. For example,
\begin{align*}
Y_\alpha (d Y_\beta Y_\gamma v) &=d Y_\beta Y_\alpha Y_\gamma v + \sum_{\theta \in I_2} df_\theta Y_\theta Y_\gamma v + df_n Y_n Y_\gamma v\\
&=d Y_\beta Y_\gamma w + d Y_\beta \sum_{\theta\in I_2} f_\theta Y_\theta v + d Y_\beta ( f_n Y_n v )\\
&\qquad + \sum_{\theta \in I_2} df_\theta Y_\theta Y_\gamma v + df_n  Y_\gamma Y_n  v + df_n \sum_{\theta\in I_1} g_\theta Y_\theta v,
\end{align*}
where all terms but 
\begin{align}\label{eq-Yn-v-1}
d Y_\beta ( f_n Y_n v ), \quad df_n  Y_\gamma Y_n  v
\end{align}
 are linear in $\xi_v, \xi_w$. Notice that we already derived $C^{2, \alpha}_g$ estimates of $Y_n v$, which implies that $C^{ \alpha}_g(B(P,\eta))$ norms of \eqref{eq-Yn-v-1} are bounded by $Cd(P)^{1+\epsilon}.$  Then we conclude the lemma by a similar argument as it for the case $Y_\alpha =Y_n$.
\end{proof}

For some  vector fields $Y_{\alpha_l}$'s, where $l=1, \cdots, p$ and $\alpha_l \in I_1$, denote  
\begin{align}\label{eq-wp}
D^p_{Y'} v=Y_{\alpha_1}\cdots Y_{\alpha_p} v,
\end{align}
which is one of the $p$-th tangential derivative of $v.$ 

\begin{lemma}\label{lemma-tangential-estimates} Suppose the assumption of Lemma \ref{lem-Yi-v} holds. Then for  any $p$-th tangential derivative of $v$ in the form of \eqref{eq-wp},
\begin{align}\label{eq-TanEstm-wp}
&\|D^p_{Y'} v\|_{C^{2, \alpha}_g (B(P, \eta))}\leq C_p d(x)^{1+\epsilon},
\end{align}
where $C_p$ is independent the choice of $P.$
\end{lemma}

\begin{proof} We prove by  induction on $p$. Case $p=1$ follows from Lemma \ref{lem-Yi-v}. Assume that case $l$ for  $l=1,\cdots, p-1$ is correct. We prove case $p$ using a similar method.

First we assume that $Y_{\alpha_1}, \cdots, Y_{\alpha_p}$ in \eqref{eq-wp} commute with $Y_i$ for $i \in I$. Denote $w = Y_{\alpha_p} v$. Then it satisfies \eqref{eq-MainEqUnderFrame-w}. Applying $Y_{\alpha_1}\cdots Y_{\alpha_{p-1}}$ to \eqref{eq-MainEqUnderFrame-w}, we derive an equation of $D^p_{Y'} v,$
\begin{align}\begin{split}\label{eq-MainEqUnderFrame-wp}
&d^2  (a_{2n, 2n}D^p_{Y'} v_{dd}+a_{n, 2n}Y_n D^p_{Y'} v_d+ a_{nn}Y_n^2 D^p_{Y'} v)+d \sum_{ \beta,\gamma  \in I_2} a_{\beta \gamma}  Y_\beta Y_\gamma D^p_{Y'} v\\
&\qquad +  \sum_{\beta \in I_2}(d^{\frac{3}{2}} a_{\beta, 2n} + d^2 c_{\beta, 2n}) Y_\beta  D^p_{Y'} v_d +  \sum_{\beta \in I_2}(d^{\frac{3}{2}} a_{\beta n} + d^2 c_{\beta n})  Y_\beta Y_n D^p_{Y'} v \\
&\qquad +d\sum_{i \in I} b_{\beta} Y_i D^p_{Y'} v - (n+1)D^p_{Y'} v= F_p d^2 + G_p,
\end{split}
\end{align}
where $a_{ij}, b_i, c_{ij}$'s are  polynomials in $d, \xi_v $ with \eqref{eq-aij-bi-v} holds and $F_p, G_p$ are  polynomials in $ \xi_v, \cdots, \xi_{D^{p-1}_{Y'} v}$, all with coefficients smooth in $z$.

We rewrite \eqref{eq-MainEqUnderFrame-wp} as
\begin{align}\label{eq-Lv-FG-p}
L_v (D^p_{Y'} v) = F_p d^2 +G_p.
\end{align}
By the induction, $C^\alpha_g$  norms of $F_p, G_p$ are bounded by $Cd^{1+\epsilon}.$ We apply the maximum principle with test functions \eqref{eq-test-v} and Schauder estimates to derive \eqref{eq-TanEstm-wp}.

Secondly, we allow that $Y_{\alpha_l}$'s in \eqref{eq-wp} do not commute with other $Y_i$'s. We apply the same method and derive \eqref{eq-Lv-FG-p} with additional terms generated by $Y_{\alpha_1\cdots \alpha_p} \xi_v - \xi_{D^p_{Y'} v}.$ 
We need to control the additional terms  that are not  in $\xi_v, \cdots, \xi_{D^{p-1}_{Y'} v}$.

For the case $D^p_{Y'} v = Y_n^p v$, by Lemma \ref{lem-Lie-Y}, when $\beta, \gamma\in I_2,$
$
Y_n^p (d Y_\beta Y_\gamma v) =  d Y_\beta   Y_\gamma  D^p_{Y'} v$ plus $
2p$ terms. For $i, j\in I$, assume that $[Y_i, Y_j] = S^k_{i j} Y_k$. The summation of $k$ is from $1$ to $2n-1$, since
$S_{ij}^{2n}=0$  according to Lemma \ref{lem-Lie-Y}.
Taking switching orders from
$d {Y^p_n}   Y_\beta Y_\gamma  v$ 
to $dY_\beta {Y^p_n}    Y_\gamma  v$, $\beta, \gamma\in I_2$, for example, 
it generates terms of the form, for $p_1=1,\cdots, p$,
\begin{align}\label{eq-Switch-v}
d {Y^{p_1-1}_n} (S_{n\beta}^m  \cdot Y_m {Y_n^{p-p_1}} Y_\gamma v),
\end{align}
which  is actually generated from $d {Y^{p_1}_n}   Y_\beta {Y^{p-p_1}_n} Y_\gamma v$ 
to $d {Y^{p_1-1}_n} Y_\beta {Y^{p-p_1+1}_n} Y_\gamma v$.
Similarly, from
$d Y_\beta  {Y^p_n}   Y_\gamma  v$ 
to $dY_\beta  Y_\gamma {Y^p_n}     v$, 
we generate terms of the form, 
\begin{align}\label{eq-Switch1-v}
d Y_\beta {Y^{p_1-1}_n} (S_{n \gamma}^m  \cdot Y_m {Y^{p-p_1}_n} v).
\end{align}
We apply the prodcut rule to \eqref{eq-Switch-v}, \eqref{eq-Switch1-v}. Notice that the $C^\alpha_g$ norms of $p$-th tangential derivatives of $v$ in the form of $ d Y_{\beta_1}\cdots Y_{\beta_p} v =dY_{\beta_1}D^{p-1}_{Y'} v$ are all bounded by $Cd^{1+\epsilon}$  due to the induction. The $(p+1)$-th derivatives of $v$ derived from \eqref{eq-Switch-v}, \eqref{eq-Switch1-v} are
\begin{align*}
d {Y^{p_1-1}_n}  Y_m {Y_n^{p-p_1}} Y_\gamma v, \quad d Y_\beta {Y^{p_1-1}_n}    Y_m {Y^{p-p_1}_n} v,
\end{align*}
which can be transformed to
\begin{align}\label{eq-p+1-v}
d Y_\gamma Y_m   {Y^{p-1}_n}   v, \quad d Y_\beta  Y_m {Y^{p-1}_n}    v,
\end{align}
plus a few $p$-th derivatives of $v$ by switching orders of $Y_i$'s. In the case $m\neq n$, \eqref{eq-p+1-v} is in the form of  $\xi_{D^{p-1}_{Y'} v}.$ In the case $m=n,$ \eqref{eq-p+1-v} are  first derivatives of $D^p_{Y'} v$, which are in $\xi_{D^p_{Y'} v}$, so we  move those terms to the left-hand side of the equation as main terms.
Similarly, we can deal with $Y_n^p(d^\frac{3}{2}Y_\beta Y_n v), Y_n^p (d^\frac{3}{2} Y_\beta v_d).$
 In sum, $D^p_{Y'} v$ satisfies an equation of form \eqref{eq-Lv-FG-p} and then \eqref{eq-TanEstm-wp} holds.

Now we induct on the number of $Y_n$'s in \eqref{eq-wp}. Assume that if $D^p_{Y'} v$ has strictly more than $k$ $Y_n$'s in its form \eqref{eq-wp}, then \eqref{eq-TanEstm-wp} holds. For case $k$, 
first we write $D^p_{Y'} v$ as
$
Y_n^k  Y_{\alpha_1}\cdots Y_{\alpha_{p-k}}v 
$
plus $(p-1)$-th tangential derivatives of $v$, where the latter has desired $C^{2, \alpha}_g$ estimates due to the induction. Again we consider switching the orders of $Y_i$'s in
\begin{align*}
Y_n^k  Y_{\alpha_1}\cdots Y_{\alpha_{p-k}} (dY_\beta Y_\gamma v)
\end{align*}
to derive $dY_\beta Y_\gamma D^p_{Y'} v.$ 
Each generated $(p+1)$-th derivative in $v$ either has more than $k$ $Y_n$'s or has at least two $Y_i$'s in its form, where $i \in I_2$. In the first case, since it also has $Y_\beta$ or $Y_\gamma$ remaining in its form,  we  switch it to $dY_\beta \tilde D^p_{Y'} v$ or $dY_\gamma \tilde D^p_{Y'} v$ for some $\tilde D^p_{Y'} v$   that has strictly more than $k$  $Y_n$'s in its form \eqref{eq-wp}. In the second case, we switch it to $dY_{\beta_1}Y_{\beta_2} D^{p-1}_{Y'} v$ where $\beta_1, \beta_2\in I_2.$ In both cases, the $(p+1)$-th derivative in $v$ has desired $C^\alpha_g$ estimate. In sum, $D^p_{Y'} v$ satisfies an equation of form
\begin{align*}
L_v (D^p_{Y'} v) = F_p d^2 +\tilde G_p,
\end{align*}
where the $C^\alpha_g(B(P,\eta))$ norm of $
\tilde G_p$ is bounded by $Cd^{1+\epsilon}$. The maximum principle and Schauder estimates imply \eqref{eq-TanEstm-wp}.
\end{proof}

\section{Proof of Theorem \ref{thm-Main1}}\label{sec-thm-1}
 
We introduce a  new variable $t$ replacing $d$ which satisfies that
$d=\frac{1}{2} t^2.$ Lemma \ref{lem-smooth-dP} implies that $\frac{t}{\sqrt{d(P)}}$ is smooth in $B(P,\eta)$ with uniform $C^{k,\alpha}_g$ estimates. Applying
\begin{align*}
\frac{\partial}{\partial d}= t^{-1}\frac{\partial}{\partial t},\quad
\frac{\partial^2}{\partial d^2}= t^{-2}\frac{\partial^2}{\partial t^2} - t^{-3}\frac{\partial}{\partial t},
\end{align*}
and \eqref{eq-TanEstm-wp}, we derive that if $p\in \mathbb{N}$, and $B(P,\eta)\subseteq \Gamma' \times [0,R)$, then
\begin{align}\label{eq-TanEstm3}
\| D^p_{Y'}v_{tt}\|_{C^\alpha_g (B(P, \eta))}, 
\|t^{-1}D^p_{Y'}v_{t}\|_{C^\alpha_g (B(P, \eta))}, 
\|t^{-2}{D^p_{Y'}v}\|_{C^\alpha_g(B(P, \eta))}\leq C_p d(P)^{\epsilon}.
\end{align}
According to Lemma \ref{lem-w-Ca}, for a smaller  $\Gamma'$, and for probably a smaller $\alpha$,
\begin{align}\label{eq-Euc-wp}
\|D^p_{Y'}v_{tt}\|_{C^\alpha (\Gamma'\times [0, R))},
\|t^{-1} {D^p_{Y'}v_{t}}\|_{C^\alpha (\Gamma'\times [0, R))}, 
\|t^{-2}{D^p_{Y'}v}\|_{C^\alpha(\Gamma'\times [0, R))}\leq C_p,
\end{align}
where the $C^\alpha$ norm is taken with respect to the Euclidean metric on $\mathbb{C}^n.$

\begin{definition}
We say that a function $w$ has an expansion of order $t^k$ in $\Gamma'\times [0, R)$, if there is some $\epsilon>0$ such that  in $\Gamma'\times [0, R)$,
\begin{align*}
w = \sum_{i=2}^k \sum_{j=0}^{N_i} c_{i, j}(y') t^i (\log t)^j + R_k,
\end{align*}
where $N_i$'s are some integers depending on $i, n$, $c_{i, j}$'s are smooth functions in $y'$, and $R_k=O(t^{k+\epsilon})$. In addition, for $p \in \mathbb{N}$, $R_k$ satisfies
\begin{align*}
\|D^p_{Y'} R_k\|_{C^{\alpha} (\Gamma'\times [0, R))}\leq C_{p,k} t^{k},
\end{align*}
where $C_{p,k}$ is some constant.
\end{definition}

Then \eqref{eq-TanEstm3}, \eqref{eq-Euc-wp} imply that $v$ has an expansion of order $t^2$ with $c_{2, j}=0$ for any $j$.

Keeping only $d^2 v_{dd} -(n-1) dv_d-(n+1) v$ on the left-hand side, we rewrite  \eqref{eq-MainEqUnderFrame-v} as, 
\begin{align*}
t^2v_{tt} - (2n-1) tv_t -(4n+4) v = F_1 t^4+ \bar F.
\end{align*}
Taking derivative in $t$ and multiplied by $1+s$, we derive
\begin{align}\label{eq-v-ttt-s}
(1+s)\left(t^2 v_{ttt} - (2n-3) tv_{tt} -(6n+3) v_t \right)=(1+s)F_1 t^4+(1+s)\bar F_t,
\end{align}
where
we compute $\bar F_t$ in a similar way as in  \eqref{eq-Ya-F2} with $Y_\alpha$ replaced by $\frac{\partial}{\partial d}$, and derive that $(1+s)\bar F_t$ is a polynomial in
$\xi_v, (\xi_v)_t$
with coefficients smooth in $y', t.$ In addition, every monomial containing $(\xi_v)_t$, if
it is not
\begin{align*}
(d^2 Y_n^2 v)_t,\quad
(d C_{\beta \gamma}Y_\beta Y_\gamma v)_t, \quad
(d C_{\alpha}Y_\alpha v)_t,
\end{align*}
for some $\beta, \gamma \in I_2, \alpha\in I_1,$   has a $\xi_v$ or $t$ factor. 
In particular,  monomials containing
\begin{align*}
(d^2 v_{dd})_t = \frac{1}{4} t^2 v_{ttt} + \frac{1}{4}(tv_{tt} -v_t)
\end{align*}
 have $\xi_v$ or $t$ factors. Combine all  monomials containing $t^2 v_{ttt}$ into a single term in the form of $H(y', t, \xi_v) \cdot t^2 v_{ttt}$, move it the left-hand of \eqref{eq-v-ttt-s},  divide the equation by $1+s - H$ and simplify it as
\begin{align}\label{eq-v-ODE}
t^2 v_{ttt}-(2n-3) t v_{tt} -(6n+3) v_t =t G,
\end{align}
where  $G$ is smooth in
$$y', t, t v_{tt}, v_t,    D_{Y'} v, D^2_{Y'}v,  tD_{Y'}v_t, tD^2_{Y'}v_t, t^2 D_{Y'}v_{tt}.$$
Then, we can follow the quasilinear case as in \cite{HanJiang} to show that $v$ has an expansion of order $k$ for any $k \in \mathbb{N}$. See also \cite{JiangXiao}. In fact, we iterate the following two lemmas, starting with $k=2.$

\begin{lemma}\label{lem-iterate-1}
If $v$ has an expansion of order $t^k$ in $\Gamma'\times [0, R)$, then $tG$ has an expansion of order $t^{k}$ in $\Gamma'\times [0, R)$.
\end{lemma}
\begin{lemma}\label{lem-iterate-2}
If  $tG$ has an expansion of order $t^k$ in $\Gamma'\times [0, R)$, then $v$ has an expansion of order $t^{k+1}$ in $\Gamma'\times [0, R)$.
\end{lemma}
Proof of Lemma \ref{lem-iterate-1} is by formal computation. Proof of Lemma \ref{lem-iterate-2} is by applying the ODE solution 
\begin{align*}
v_t &= \left(v_t(R)t^{-2n-1} + \frac{R^{-2n-4}}{2n+4} \int_0^{R}t^3 G dt\right)t^{2n+1}\\
&\qquad- \frac{ t^{-3}}{2n+4} \int_0^{t} t^3 G dt -\frac{t^{2n+1}}{2n+4}\int_t^R t^{-2n-1} G dt
\end{align*}
of \eqref{eq-v-ODE}. 

Next, we proof Theorem \ref{thm-Main1}.
\begin{proof}
 By iterating Lemmas \ref{lem-iterate-1}, \ref{lem-iterate-2}, we obtain 
smooth functions $c_2, \cdots, c_{n}$ and $c_{i,j}$ for $i\ge 2n+2$ and $0\le j\le N_i$ in $\Gamma'$ such that 
\begin{align}\label{eq-expression-any-ell}
v(y', t)=\sum_{i=2}^{n}c_i(y') t^i
+\sum_{i=2n+2}^{\ell} \sum_{j=0}^{N_i} c_{i, j}(y')t^{i}(\log t)^j
+R_{\ell}(y',t),\end{align}
where $R_\ell=O(t^{\ell+\epsilon})$ is $C^{\ell,\alpha}$ in $\Gamma' \times [0, r)$, for any $r\in (0,R)$. 
By formal computation  of \eqref{eq-v-ODE}, 
\begin{align}\label{eq-Ni}
N_i = \left\lfloor \frac{i}{2n+2}\right\rfloor.
\end{align}
For unification, we write $c_{i,0}=c_{i}$ for $2\le i\le n$. Consider a formal expansion 
\begin{align*}
v_*=\sum_{i=1}^\infty \sum_{j=0}^{N_i}c_{i,j}t^i(\log t)^j, 
\end{align*}
where we set $c_{i,j}=0$ when such term is absent in the expansion of $v$. 
By exchanging the order of summations, we get 
\begin{align*}
v_*=\sum_{j=0}^\infty \sum_{i=j(2n+2)}^{\infty}c_{i,j}t^i(\log t)^j.
\end{align*}
For each fixed pair $i$ and  $j$, introduce $k$ such that 
\begin{align}\label{eq-relation-ijk} 
i=j(2n+2)+k.\end{align} 
Throughout the proof, any $i,j,k$ are related by \eqref{eq-relation-ijk}. 
Then, 
\begin{align*}
v_*=\sum_{j=0}^\infty \sum_{k=0}^{\infty}c_{j(2n+2)+k,j}t^{j(2n+2)+k}(\log t)^j
=\sum_{j=0}^\infty \sum_{k=0}^{\infty}c_{j(2n+2)+k,j}t^k(t^{2n+2}\log t)^j.
\end{align*}
We now write
\begin{align}\label{eq-Form-Series}
v_*=\sum_{l=1}^\infty\sum_{k+j=l}c_{j(2n+2)+k,j}t^k(t^{2n+2}\log t)^j.
\end{align}
We emphasize that \eqref{eq-Form-Series} is a formal expansion. The series in the right-hand side may not converge. 
Following \cite{HanWang}, we now introduce a sequence of cutoff functions to construct a convergent series. 

Set 
\begin{align}\label{eq-definition-TS} 
T =t,\quad S =-t^{2n+2}\log t.
\end{align} 
For simplicity, we assume $R\in (0,1)$. For each $l\ge 1$, set 
\begin{align}\label{eq-definition-w-l} 
w_l(y',T,S)=\sum_{k+j=l}(-1)^jc_{j(2n+2)+k,j}(y')T^kS^j.
\end{align}
This is a homogeneous polynomial in $T$ and $S$ of degree $l$, with coefficients given by smooth functions of $y'$. 
Take a cutoff function $\eta\in C_0^\infty(-R,R)$ with $\eta=1$ on $(-R/2,R/2)$. 
For any constant $\lambda_l\ge 1$, we observe that
$\eta(\lambda_l {T}) =0$ if $|{T} |\geq {1}/{\lambda_l}$ and $\eta(\lambda_l {T}) =1$ if  $|{T} |\leq {1}/{(2\lambda_l)}$. 
A similar result holds for $\eta(\lambda_l {S})$. 
It is straightforward to verify that, for any $\alpha \in \mathbb{Z}_+^{2n+2}$ with $|\alpha|<l$, 
\begin{align*}
\big| \partial^\alpha_{y',T,S} \big[w_l(y', T, S) \eta(\lambda_l{T})\eta(\lambda_l{S})\big] \big| 
\leq \frac{C_{l, \alpha}}{\lambda_l^{l - |\alpha|}},
\end{align*}
for some positive $C_{l, \alpha}$ depending only on $w_l$ and $\alpha$. 
By choosing $\lambda_l$ sufficiently large, we get 
\begin{align}\label{eq-decay-w}
\|w_l(y', T, S) \eta(\lambda_l{T})\eta(\lambda_l{S})\|_{C^{l-1}(B_R\times I\times I)} 
\leq 2^{-l},
\end{align}
where $I=(-R, R)$. Next, set 
$$w(y',T,S)=\sum_{l=1}^\infty w_l(y', T, S)\eta(\lambda_l{T})\eta(\lambda_l{S}).$$
We note that each term in the summation is smooth in $y', T$, and $S$. 
By \eqref{eq-decay-w}, we obtain, for any $m\in\mathbb N$, 
\begin{align*}
\sum_{l=m+1}^\infty\|w_l(y', T, S) \eta(\lambda_l{T})\eta(\lambda_l{S})\|_{C^{m}(B_R\times I\times I)} 
\leq \sum_{l=m+1}^\infty 2^{-l}=2^{-m}.
\end{align*}
Hence, $w$ is $C^m$ in $y', T$, and $S$. This is true for any $m\in\mathbb N$. 
Therefore, $w$ is smooth in $y', T$, and $S$.

Next, with \eqref{eq-definition-TS}, we have 
\begin{align}\label{eq-expression-w}
w(y',t,-t^{2n+2}\log t)=\sum_{l=1}^\infty w_l(y', t, -t^{2n+2}\log t)
\eta(\lambda_l{t})\eta(-\lambda_l{t^{2n+2}\log t}).\end{align}
For each $l\ge 1$, we set 
\begin{equation}\label{eq-definition-w-tilde-l1}\widetilde w_l(y',t)=w_l(y', t, -t^{2n+2}\log t).\end{equation}
By \eqref{eq-definition-w-l}, we have 
\begin{equation}\label{eq-definition-w-tilde-l2}
\widetilde w_l(y',t)=\sum_{k+j=l}c_{j(2n+2)+k,j}t^{j(2n+2)+k}(\log t)^j.
\end{equation}
With $k+j=l$, we have $j(2n+2)+k=l+jn$, and hence
$$\widetilde w_l(y',t)=\sum_{k+j=l}c_{l+jn,j}t^{l+jn}(\log t)^j.$$
 Then, the lowest power of $t$ is $l$, given by $j=0$. 
Therefore, by taking $\lambda_l$ larger if necessary, we obtain 
\begin{align}\label{eq-decay-w-tilde}
\|\widetilde w_l(y',t)\eta(\lambda_l{t})\eta(-\lambda_l t^{2n+2}\log t)\|_{C^{l-1}(B_R\times I)} 
\leq 2^{-l}.
\end{align}
Fix an arbitrary $m\ge 2n+2$ and set 
\begin{align}\label{eq-definition-Sm} 
S_m(y',t)=\sum_{l=m+1}^\infty \widetilde w_l(y',t)\eta(\lambda_l{t})\eta(-\lambda_l t^{2n+2}\log t).
\end{align} 
By \eqref{eq-decay-w-tilde}, we obtain
\begin{align*}
\sum_{l=m+1}^\infty\|\widetilde w_l(y',t)\eta(\lambda_l{t})\eta(-\lambda_l t^{2n+2}\log t)\|_{C^{m}(B_R\times I)}  
\leq \sum_{l=m+1}^\infty 2^{-l}=2^{-m}.
\end{align*}
Hence, $S_m\in C^m(\bar B_r\times I)$ for any $r\in (0,R)$. 
By \eqref{eq-expression-w}, \eqref{eq-definition-w-tilde-l1}, and \eqref{eq-definition-Sm}, we obtain 
\begin{align}\label{eq-decomposition-w-9}
w(y',t,-t^{2n+2}\log t)=\sum_{l=1}^m \widetilde w_l(y', t)
\eta(\lambda_l{t})\eta(-\lambda_l{t^{2n+2}\log t})+S_m(y', t).\end{align}
We point out that summation in the right-hand side contains terms with regularity worse than $C^m$. 
We now prove that this summation also contains all terms in $v$ with regularity worse than $C^m$. 

For the same arbitrarily fixed $m\ge 2n+2$, set 
\begin{equation}\label{eq-definition-M}
M=\max\{j(2n+2)+k;\, 1\leq k+j\le m, k\ge 0, j\ge 0\}.\end{equation}
It is easy to see that $M\ge m$. By \eqref{eq-expression-any-ell}, we have 
\begin{align}\label{eq-expression-v-M}
v=\sum_{i=1}^{M} \sum_{j=0}^{N_i} c_{i, j}t^{i}(\log t)^j
+R_{M},\end{align}
and $R_M\in C^M(\Gamma' \times [0, r))\subset C^m(\Gamma' \times [0, r))$. 
By the definition of $M$ in \eqref{eq-definition-M}, any $i$ as in \eqref{eq-relation-ijk} with $k+j\le m$ satisfies $i\le M$. 
Hence, we can decompose the summation in the right-hand side of \eqref{eq-expression-v-M} into two parts
according to $k+j\le m$ or $k+j\ge m+1$. Thus, 
\begin{align*}v&=\sum_{k+j\le m}c_{j(2n+2)+k,j}t^{j(2n+2)+k}(\log t)^j\\
&\qquad+\sum_{\substack{k+j\ge m+1\\ i\le M, j\le N_i}}c_{j(2n+2)+k,j}t^{j(2n+2)+k}(\log t)^j+R_M.\end{align*}
In the second summation, the power of $t$ is given by  $j(2n+2)+k\ge j+k\ge m+1$. 
We now use the elementary fact that $t^{m+1}(\log t)^j$ is a $C^m$ function on $[0,1]$. Hence, 
each term in the second summation is in $C^m(\Gamma' \times [0, r))$, for any $r\in (0,R)$. 
We can combine the second summation with $R_M$ and write 
\begin{align*}v=\sum_{k+j\le m}c_{j(2n+2)+k,j}t^{j(2n+2)+k}(\log t)^j+\widetilde{R}_m,\end{align*}
for some $\widetilde{R}_m\in C^m(\Gamma' \times [0, r))$, for any $r\in (0,R)$. 
By \eqref{eq-definition-w-l}, we obtain 
\begin{align*}v=\sum_{l=1}^{m}\widetilde w_l+\widetilde{R}_m.\end{align*}
By introducing the cutoff functions $\eta(\lambda_l{t})\eta(-\lambda_l t^{2n+2}\log t)$ for each term in the summation, 
we write
\begin{align*}v=\sum_{l=1}^{m}\widetilde w_l\eta(\lambda_l{t})\eta(-\lambda_l t^{2n+2}\log t)
+\sum_{l=1}^{m}\widetilde w_l\big[1-\eta(\lambda_l{t})\eta(-\lambda_l t^{2n+2}\log t)\big]
+\widetilde{R}_m.\end{align*}
In the second summation, each term is 0 for small $t$ and hence is smooth in $\Gamma' \times [0, r)$ for any $r\in (0,R)$. 
Therefore, 
\begin{align}\label{eq-definition-w-l2}v(y',t)=\sum_{l=1}^{m}
\widetilde w_l(y',t)\eta(\lambda_l{t})\eta(-\lambda_l t^{2n+2}\log t)
+\widehat{R}_m(y',t),\end{align}
for some $\widehat{R}_m\in C^m(\Gamma' \times [0, r))$ for any $r\in (0,R)$.

Finally, we set 
$$w_0(y', t)= v(y', t)-w(y', t, t^{2n+2}\log t).$$
By \eqref{eq-definition-w-l2} and \eqref{eq-decomposition-w-9}, we obtain 
$w_0=\widehat{R}_m-S_m\in C^m(\Gamma' \times [0, r))$ for any $r\in (0,R)$. This holds for any $m\ge\overline{m}$. 
Thus, $w_0\in C^\infty(\Gamma' \times [0, r))$ for any $r\in (0,R)$. 
As a consequence,  
$v = w_0(y', T) + w(y', T, S)$ is smooth in $y', T, S$ and the theorem is proved. 
\end{proof}

\section{A Counterexample to the Local Convergence}\label{Sec-CntrEx}

In \cite{HanJiang2}, we proved a local convergence theorem for solutions to a type of uniformly  degenerate elliptic equations. For example, let $v$ be a solution of
\begin{align}\label{eq-d2-T}
d^2 v_{dd}+  d^2T v-(n-1)d v_d-(n+1)v=F
\end{align}
in $\Gamma \times (0, R) = \{(y', d) : (y', 0) \in \Gamma, d\in (0, R)\}$,
and $T$ is a uniformly elliptic operator in tangential coordinates. 
In addition, if we assume that $\Gamma$ is analytic open portion of $\partial 
\Omega$ and $T, F$ are analytic, then $v$ is an analytic function of $y', d, d^{n+1} \log d$ in $\Gamma \times [0, R)$. We proved that the Taylor series of $v$  with logarithmic terms  converges  in $\Gamma' \times [0, r)$ for some $\Gamma' \times [0, r) \subseteq \Gamma \times [0, R)$.

However, our equation
\eqref{eq-MainEqUnderFrame-v}  is not a simple perturbation of \eqref{eq-d2-T}. In fact, $Y_n^2 v$ has a degenerate factor $d^2$ but $Y_\beta Y_\gamma$, for $\beta, \gamma \in I_2$,  has a different degenerate factor $d$. 
In this section, we disprove the  convergence of the Taylor series with logarithmic terms of the solutions to 
\begin{align} \label{eq-CntrEx1}
d^2 v_{dd}+ d^2v_{tt}+ dv_{ss}-(n-1)d v_d-(n+1)v=0,
\end{align}
in $G_r= \{d \in \mathbb{R} : 0<d<r\} \times \{(t, s) \in \mathbb{R}^2 :  t^2+s^2 <r^2\}$.
Here $s$ may extend to higher dimensional coordinates.

Consider the operator 
$$A=d^2 \partial_{dd}+d^2 \partial_{tt}-(n-1)d\partial_d-(n+1)=d^2 \Delta_{d,t}-(n-1)d\partial_d-(n+1). $$
By Proposition 1 of \cite{BG1972}, there exists a nonanalytic function $w$ in $d$ and $t$, 
defined in a neighborhood of the origin in $\mathbb{R}^+ \times \mathbb{R}$, such that,
for each $k \in \mathbb{N}$ and $\alpha \in \mathbb{N}^2$,
\begin{align}\label{eq-NAw}
\|\partial^\alpha \Delta_{d, t}^k w\|_{L^2}&\leq C^{|\alpha|+k+1}(2k)!(2\alpha)!.
\end{align}
The same estimates hold for $\|\partial^\alpha \Delta_{d, t}^k w\|_{C^2}$ if we enlarge $C$.

Let $v$ be a bounded  $C^2$ solution of \eqref{eq-CntrEx1}. 
By a formal computation, all local terms in the expansion of $v$ vanish. 
Assume the first global term of $v$ is $c_{n+1}$. 
Then, $v$ has a  formal expansion given by
\begin{align}\label{eq-CN1}
v= c_{n+1}d^{n+1}+c_{n+2}d^{n+2}+\cdots+c_{k}d^{k}+\cdots.
\end{align}
A convergent  series solution of $Av +dv_{ss}=0$ (or \eqref{eq-CntrEx1})
can be constructed as
\begin{align}\label{eq-CntrEx}
\overline{v}=\sum_{k=0}^\infty \frac{(-A/d)^k (d^{n+1}w )}{(2k)!}s^{2k},
\end{align}
where $w$ is a nonanalytic function satisfying \eqref{eq-NAw}.
Note that  the coefficient of $s^{2k}$ has an explicit factor $d^{n+1}$ for each $k$. In fact, formally,
\begin{align}\label{eq-A-over-d}
(-A/d) (d^{n+1} h)=-d^{n+1}(d \Delta_{d, t} h+2h_d ),
\end{align}
for any function $h$. Using \eqref{eq-NAw}, we can prove that there is a universal constant $B$ such that
\begin{align*}
 |(-A/d)^k (d^{n+1}w )| \leq B^{k}(2k)!,
\end{align*}
which implies that \eqref{eq-CntrEx} and its derivative of any order in $s$ converge when $Bs^2 <1$. Further calculation shows that $\overline{v}$ belongs to Gevrey space $G_2$. In fact, \eqref{eq-NAw} and \eqref{eq-A-over-d} imply that
\begin{align*}
|\partial^p_{d, t} (-A/d)^k (d^{n+1}w )| \leq B^{k}(2k)! (2p)!.
\end{align*}
Finally, since $\overline{v}= d^{n+1}w$ is not analytic in $d, t$ when $s=0$,
 $\overline{v}$ is not analytic in $d, t$.

The example constructed above is for a linear equation with the same structure as the 
linear part of \eqref{eq-MainEqUnderFrame-v}. It strongly suggests that it is impossible to prove the 
convergence in the local setting for \eqref{eq-MainEqUnderFrame-v}.

\section{Global Convergence Theorem}\label{Sec-Converg}

In this section, we prove Theorem \ref{thm-GlobalConvergence}. 
In Section \ref{sec-Frame}, we defined a set of frames $(U_Q, \{Y_i^Q\}_{i \in I})$. Notice that all $Y_i^Q$'s are globally defined, while only in $U_Q$, they form a frame. Let  $R$ be a positive and small number so that we can take a covering of
\begin{align*}
\Omega_R:= \{z\in \mathbb{C}^n: d(z)\in (0, R)\}
\end{align*}
with  finite such $U_Q$'s. Denote  $S$ the set of  $Y_i^Q$'s corresponding to all $U_Q$'s in the covering. Assume that
\begin{align*}
S = \{X_1, X_2, \cdots, X_N\},
\end{align*}
for some $N\geq 2n$.
Here we denote $X_{2n} = \frac{\partial}{\partial d}, X_n =T$, which will replace $Y_n$ and is to be defined later in \eqref{eq-theta-T}. Then $X_i \perp Y_n, Y_{2n}$ if $i \neq n, 2n.$ When $\partial \Omega$ is analytic, all $X_i$'s are analytic in $\overline   \Omega_R$. Define the index sets $J = \{1, 2, \cdots, N\}, J_1 = \{i\in J: i\neq 2n\}$ and $J_2 = \{i\in J: i\neq n, 2n\}$.

Let $H(\partial \Omega)$ be the subbundle of $T \partial \Omega$ that its fiber at any point $Q \in \partial \Omega$ is
\begin{align}
span\{X_i(Q) : X_i \in S, i\in J_2\}.
\end{align}
Let 
$$\theta= \frac{1}{\sqrt{-1}}(\partial \rho -\bar{\partial} \rho).$$
Then, $(\partial \Omega, H(\partial \Omega), J)$ is a CR-manifold, and $\theta$ annihilates exactly $H(\partial \Omega)$. In addition, there is a unique tangent vector $T$  on $\partial \Omega$ such that
\begin{align}\label{eq-theta-T}
\theta(T) = -\lambda,\,\, d\theta(T, \cdot) = 0.
\end{align}
See \cite{WangXD}. The following result is well known.
\begin{lemma}\label{lem-TX}
 $[T, X]\in \Gamma(H(M))$ if  $X\in \Gamma(H(M))$, where $\Gamma(H(M))$ denotes the set of sections on $H(M)$.
\end{lemma}
It follows from the Cartan's formula as,
\begin{align*}
0=d\theta (T,X)
=T\theta(X)-X\theta(T)-\theta([T,X])
=-\theta([T,X]).
\end{align*}
Extend $T$ analytically onto $\overline \Omega_R$.

\begin{lemma}\label{lemma-new-frame} Under the new frame, 
\begin{align*}
\tilde{Y}_{2n} = Y_{2n}, \,\, 
\tilde{Y}_n &= T, \,\, \tilde{Y}^Q_\alpha = Y^Q_\alpha\text{ for }\alpha \in I_2,
\end{align*}
Lemma \ref{lem-Lu} and Lemma \ref{lem-NormY} still hold  in $U_Q$.
\end{lemma}

\begin{proof}
For any  point $P$ near $\partial \Omega$, 
we find a unique $Q \in \partial \Omega$ such that $d(P)= dist(P, Q)$. Under the real coordinates corresponding to  $\{z^Q_i\}$, we have
\begin{align*}
\theta = 2\mathrm{Im}(\partial \rho) =\sum_{i=1}^n\left(\frac{\partial \rho}{\partial x_{i+n}} dx_i -\frac{\partial \rho}{\partial x_{i}} dx_{i+n}\right),
\end{align*}
which is $-\lambda e^{-\lambda d(P)} dx_n$ at $P$, for   $\lambda$  given in \eqref{def-lambda}. Thus on  $M=\partial \Omega$, 
$\theta(Y_n) = -\lambda$ and $ \theta(Y_\alpha)=0$ for $\alpha \in I_2$. This verifies that $\theta$ annihilates exactly $H (\partial \Omega)$. In addition,
  $Y_n  |_{\partial \Omega}- \tilde Y_n |_{\partial \Omega} \in H(M)$ by \eqref{eq-theta-T}. 
Therefore  in $\partial \Omega \times [0, R)$, 
\begin{align*}
Y^Q_\alpha &=\tilde{Y}^Q_\alpha \quad\text{if } \alpha \in I_2,\\
Y_n  &= (1+df_n) \tilde{Y}_n  +\sum_{\alpha \in I_2} f_\alpha \tilde{Y}^Q_\alpha,
\end{align*}
for some smooth functions $f_i, i\in I_1$. Then, the form of \eqref{eq-MN} still holds, since
\begin{align*}
d^2 Y_n^2  =d^2   \tilde{Y}_n^2 +\cdots,
\end{align*} 
where omitted terms in $``\cdots"$do not affect the form of \eqref{eq-MN}.
Lemma \ref{lem-NormY}, with $\{Y_i\}$ replaced by $\{\tilde Y_i\}$, follows similarly.
\end{proof}

In the rest of this section, we use the frame  $\{\tilde{Y}_i^Q\}$ in $U_Q$ but still denote it by $\{Y_i^Q\}$ or simply $\{Y_i\}$. 
Assume that in $U_Q$, for $i, j\in J$, 
\begin{align}\label{eq-Zij}
[X_i, X_j]=S_{ij, Q}^m Y_m, 
\end{align}
where $S_{ij, Q}^m$'s are analytic functions under our assumption and the summation over $m$ is for $m\in I.$
By Lemma \ref{lem-TX}, we have for $\alpha \in J_2$,
\begin{align}\label{eq-LieBr}
S_{n \alpha, Q}^n=0,
\end{align}
which implies that when exchanging the orders of $X_n$ and $X_\alpha, \alpha \in J_2,$ in a differentiation, it does not generate a new $Y_n$ during the process. This is the key reason that why we replace $Y_n$ by $T$.

There are positive constants $D$ and $ B>0$ such that, for any $k \geq 0$, any $i,j,m$ such that $ 1\leq i, j, m \leq 2n-1$, and any $U_Q$ in the finite covering,
\begin{align}\label{eq-AnaS}
\|D_{X'}^k S_{ij, Q}^m\|_{C^\alpha (U_Q)}\leq DB^{(k-3)^+} (k-3)!.
\end{align}
If $l<0$, we set   $l!=1$.  
In deriving \eqref{eq-CombK} later on, we will also use a weaker estimate
$$\|D_{X'}^k S_{ij, Q}^m\|_{C^\alpha(U_Q)}\leq DB^{(k-2)^+} (k-2)!. $$

Let $\Phi$ 
be an analytic function in $\overline \Omega_R\times\mathbb R^N$. Then there exists positive constants $C, A$, such that for any 
$(x,y)\in \Omega_R \times\mathbb R^N$ and any nonnegative integers $j$ and $k$ 
with $j+k\le p$, 
\begin{equation}\label{eq-Composition0}
\Big|\frac{\partial^{j+k}\Phi(x,y)}{\partial x^j\partial y^k}\Big|
\le CA^{j+k}(j-2)!(k-2)! 
\end{equation}

Similar to \eqref{eq-wp}, 
 denote  
\begin{align*}
 D^p_{X'}  v=  X_{\alpha_1} \cdots X_{\alpha_p} v,
\end{align*}
where each $X'$ denotes an $X_{\alpha_l}$ for some  $\alpha_l \in J_1.$
\begin{theorem}
\label{Thm-TanAnaly}
Assume that $\rho$ is analytic on $\overline \Omega_R$ for some $R \in \mathbb{R}^+$. Let $v$ be a solution to \eqref{eq-MainEqUnderFrame-v} such that \eqref{eq-TanEstm-wp} holds.
Then there exists $D, B>0$ such that for any $p \in \mathbb{N}$, and for any $P$ satisfying $B(P, \eta) \subseteq \Omega_R$,
\begin{align}\label{eq-TanAna}
\| D^p_{X'}  v\|_{C^{2, \alpha}_g(B(P, \eta))}\leq DB^{(p-2)^+} (p-2)! d(P)^{1+\epsilon}.
\end{align}
\end{theorem}

\begin{proof} Due to
the interior analyticity of solutions to uniformly elliptic equations, $v$ is analytic in $\Omega$.
\eqref{eq-TanAna} always holds when $r<d(P)<R$ for any $r \in (0, R)$. So we only have to show \eqref{eq-TanAna} when $d(P)$ is smaller than a fixed small  $r$. By taking $r$ small enough,  for any $P$ satisfying $0<d(P)<r$, we have  
\begin{align}\label{eq-P-condition}
B(P, \eta) \subseteq U_Q \cap \Omega_R,
\end{align}
 for some $U_Q$ in the finite covering of $ \Omega_R$.

Let
\begin{align}\label{eq-Gamma}
\Gamma_Q = U_Q \cap \partial \Omega.
\end{align}
 Then $v$ satisfies \eqref{eq-MainEqUnderFrame-v} in $U_Q$, where $Y_i = Y_i^Q$ for $i \in I$ . By rearranging the sub-indices if necessary, we assume that $X_i = Y_i^Q$ for $i \in I$. For simplicity,  we omit the index $Q$ in $Y_i^Q$ and $ S_{ij, Q}^m$ in the rest of this proof.

We prove  by induction on $p$. By taking $D$  sufficiently large,  \eqref{eq-TanAna} holds for $p\leq 2$. 
 Assume that  \eqref{eq-TanAna}  is true for case $l$ when $l<p$.
  We prove case $p$ with $p\geq 3.$ We always set $C$ to be some positive constant that does not depend on $B, p$ but may change line by line.

 First we consider the easy case that each $X'$ commutes with all $Y_i$'s for $i\in I$ in
\begin{align*}
\Gamma_Q \times [0, r]:=\{(y', d): (y', 0 )\in \Gamma_Q, 0\leq d\leq r\}.
\end{align*} 
By the induction, for  any $P$ such that $0<d(P)<r$, and for any $l <p$,
 \begin{align}\label{eq-xi-Dlv-1}
 \|\xi_{D^l_{Y'} v}\|_{L^\infty(B(P,\eta))} \leq DB^{(l-2)^+}(l-2)! d^{1+\epsilon}.
 \end{align}
If we only consider the first derivatives of  $D^l_{Y'} v$ in  $\xi_{D^l_{Y'} v}$ for $l=p-1$, \eqref{eq-xi-Dlv-1} and \eqref{eq-d-dP} imply that
 \begin{align}\label{eq-Dpv-MaxPr}
  |D^p_{Y' } v| \leq 2^{1+\epsilon}DB^{p-3}(p-3)!d^{\epsilon}
 \end{align}
 in $\Omega_r$.
We  write \eqref{eq-MainEqUnderFrame-w} as
\begin{align}\label{eq-w-a}
\overrightarrow{a} \cdot \xi_{w}- (n+1)w = \tilde F d^2 +\tilde G,
\end{align} 
where $\overrightarrow{a}, \tilde F$ are polynomials in $d, \xi_v$ and $ \tilde G$ is  a polynomial in $\xi_v$, all with analytic coefficients in $z$. 
We assume that $B$ is large such that
\eqref{eq-Composition0} holds when $\Phi=\overrightarrow{a}$ or $\tilde F$, and $y = (d, \xi_v)$. So does \eqref{eq-Composition0} when $\Phi=\tilde G$, and $y =  \xi_v$.

Applying $D^{p-1}_{X'}$ to \eqref{eq-w-a}, we derive
\begin{align}\label{eq-wp-a}
\overrightarrow{a} \cdot \xi_{D_{Y'}^p v}- (n+1)D_{X'}^p v =H_p,
\end{align} 
 where 
\begin{align*}
 H_p :=D_{X'}^{p-1} ( \tilde F d^2 +\tilde G) - \sum_{l=1}^{p-1}{p-1 \choose l} (D^l_{X'} \overrightarrow{a} )\cdot \xi_{D_{X'}^{p-l} v}.
\end{align*}
  By Lemma \ref{lemma-Composition} and the induction,
 \begin{align*}
 |D_{X'}^{p-1} ( \tilde F d^2 +\tilde G)| \leq CB^{p-3}(p-3)!d^{1+\epsilon},
 \end{align*}
and by Lemma \ref{lem-Composition2} and the induction,
 \begin{align*}
&\left| \sum_{l=1}^{p-1}{p-1 \choose l} (D^l_{X'} \overrightarrow{a} )\cdot \xi_{D_{X'}^{p-l} v}\right|\\
&\leq  C\sum_{l=1}^{p-1} {p-1 \choose l}B^{(l-2)^+}(l-2)! \cdot B^{(p-l-2)^+}(p-l-2)! d^{1+\epsilon}\\
&\leq C B^{p-3} (p-2)!d^{1+\epsilon},
 \end{align*}
 where $C$ changes line by line. Thus
 \begin{align}\label{eq-Hp-Linf}
 |H_p| \leq C B^{p-3} (p-2)!d^{1+\epsilon}.
 \end{align}
 
We apply the maximum principle with  test function 
\begin{align}\label{eq-M-p}
M(x)=bB^{p-3} (p-2)!    d^{1+\epsilon}\quad\text{in }\Omega_r
\end{align}
to \eqref{eq-wp-a}
and obtain
$$|D^p_{X'} v|\leq CB^{p-3} (p-2)!   d^{1+\epsilon}.$$
In fact, by \eqref{eq-Dpv-MaxPr}, $|D^p_{X'} v|\leq M(x)$ on $\partial \Omega_r$ if 
\begin{align*}
b> 2^{1+\epsilon} D/r.
\end{align*}
 If $D^p_{X'} v - M(x)$ has a local maximum at $P \in \Omega_r$, then \eqref{eq-P-condition} holds for $P$, and  \eqref{eq-wp-a} holds in $B(P, \eta)$ under the frame $\{Y_i^Q\}$. However, 
$
w = D^p_{X'} v - M(x)
$
 is a subsolution of  the elliptic equation  $\overrightarrow{a} \cdot \xi_{w}=0$ if $b$ is much larger than $C$ in \eqref{eq-Hp-Linf}. It contradicts that $D^p_{X'} v - M(x)$  has a local maximum at $P$. We observe that the choice of $C$ is independent of $B, p$.

Note that Lemma \ref{lemma-Composition} still holds 
if we replace the absolute value $|\cdot|$ by the $C^\alpha_g$-norm. Similarly as in deriving \eqref{eq-Hp-Linf}, we have
 \begin{align*}
 \|H_p\|_{C^\alpha_g(B(P,\eta))} \leq C B^{p-3} (p-2)!d^{1+\epsilon}.
 \end{align*}
Applying the Schauder estimates to \eqref{eq-wp-a},
\begin{align*}
\| D^p_{X'}  v\|_{C^{2, \alpha}_g(B(P, \frac{1}{2}\eta)))} \leq CB^{p-3}(p-2)! d^{1+\epsilon},
\end{align*}
which implies
 \eqref{eq-TanAna} by setting $B$  larger than $C$ and coving $B(P, \eta)$ with 
 smaller balls of radius $\frac{1}{2}\eta$.

In the rest of the proof, we do not assume that $X'$  commutes with all $Y_i$'s in $\Gamma \times [0, r].$ 
We  switch the $d$ factors in $D^{l}_{X'} \xi_v $ to the left of $D^{l}_{X'} $ and still denote it as $D^{l}_{X'} \xi_v$.
For example, the entries of $D^{l}_{X'} \xi_v$  include
\begin{align}\label{eq-example-1}
d  D^{l}_{X'} Y_\beta Y_\gamma v ,\quad
d^{\frac{3}{2}}  D^{l}_{X'} Y_\beta v_d, \quad
d^2  D^{l}_{X'}  v_{dd}. 
\end{align}
Set $\widehat{D^{l}_{X'} \xi_v}$
to be a vector that it includes all entries in $D^{l}_{X'} \xi_v$, and any term generated from entries in $D^{l}_{X'} \xi_v$ by switching the orders of vector fields. For example, $\widehat{D^{l}_{X'} \xi_v}$ includes
\begin{align*}
d  Y_\beta D^{l}_{X'}  Y_\gamma v,\quad d^{\frac{3}{2}} \frac{\partial}{\partial d}  D^{l}_{X'} Y_\beta v,\quad d^2  D^{2}_{X'} \frac{\partial^2}{\partial d^2} D^{l-2}_{X'} v,
\end{align*}
which are derived from \eqref{eq-example-1} by switching the orders of differentiation.
Besides \eqref{eq-TanAna}, we   prove a stronger estimate inductively that,  when $l\leq p$, 
\begin{align}\label{eq-TanAna01}
\|  \widehat{D^{l}_{X'} \xi_v}\|_{C^\alpha_g(B(P, \eta))} \leq DB^{(l-2)^+} (l-2)! d(P)^{1+\epsilon}.
\end{align}
The benefit is that to prove case $l=p$ during the induction, for certain derivatives of $v$ up to $(p+1)$-th order, we do not have to worry about the order of differentiation.

Applying $D^{p-1}_{X'}$ to \eqref{eq-w-a}, we derive 
\begin{align}\label{eq-Hp-Gp}
\overrightarrow{a} \cdot \xi_{D_{X'}^p v}- (n+1)D_{X'}^p v =\tilde H_p  +\tilde G_p,
\end{align} 
where
\begin{align*}
\tilde H_p &= D_{X'}^{p-1} ( \tilde F d^2 +\tilde G) - \sum_{l=1}^{p-1}{p-1 \choose l} (D^l_{X'} \overrightarrow{a} )\cdot D_{X'}^{p-l} (\xi_{ v}),\\
\tilde G_p &= \overrightarrow{a} \cdot \xi_{D_{X'}^p v} - \overrightarrow{a} \cdot D_{X'}^p \xi_{ v}.
\end{align*}
Since the highest order  derivatives of $v$ in $\tilde H_p$ is $D^{p+1}_{X'} v$, applying \eqref{eq-TanAna01} and Lemmas \ref{lemma-Composition}, \ref{lem-Composition2}, we have
\begin{align*}
\|\tilde H_p \|_{C^\alpha_g (B(P, \eta))} \leq CB^{p-3}(p-2)! d^{1+\epsilon},
\end{align*}
which is quite similar to the case that each $X'$ commutes with $Y_i$'s for all $i\in I$.
We want to show that
\begin{align*}
\|\tilde G_p \|_{C^\alpha_g (B(P, \eta))} \leq CB^{p-3}(p-2)! d^{1+\epsilon}.
\end{align*}

{\it Case 1: $D^p_{X'}= D^p_{X_n}$}. We need the following lemma.

\begin{lemma}\label{lem-SwitchOrder}
Assume that $p\geq 3$ and \eqref{eq-AnaS}, \eqref{eq-TanAna01} hold
for $1\leq l \leq p-1$.
Then, 
\begin{align}\label{eq-Switch-p-factorial}
\|  D_{X^n}^p \xi_{ v} -\xi_{D_{X^n}^p v}\|_{C^\alpha_g(B(P,\eta))}
\leq CB^{p-3}(p-2)!d(P)^{1+\epsilon},
\end{align}
where $C$ depends on $D$ but  it is independent of $B, p$.
\end{lemma}

\begin{proof}
We take the term
\begin{align*}
d D^p_{X'}   Y_\beta Y_\gamma  v-  d  Y_\beta Y_\gamma D^p_{X'}  v
\end{align*}
for example. The rest entries in $D_{X^n}^p \xi_{ v} -\xi_{D_{X^n}^p v}$ can be estimated in a similar way.

Switching $Y_\beta$ with $X_n$'s from
$d D^p_{X_n}   Y_\beta Y_\gamma  v$ 
to $dY_\beta D^p_{X_n}    Y_\gamma  v$, 
it generates terms of the form, for $p_1=1,\cdots, p$, 
\begin{align}\label{eq-Switch}
d D^{p_1-1}_{X_n} (S_{n\beta}^m  \cdot Y_m D^{p-p_1}_{X_n} Y_\gamma v),
\end{align}
which  is actually generated from $d D^{p_1}_{X_n}   Y_\beta D^{p-p_1}_{X_n} Y_\gamma v$ 
to $d D^{p_1-1}_{X_n} Y_\beta D^{p-p_1+1}_{X_n} Y_\gamma v$.
Similarly, switching $Y_\gamma$ with $X_n$'s from
$d Y_\beta  D^p_{X_n}   Y_\gamma  v$ 
to $dY_\beta  Y_\gamma D^p_{X_n}     v$, 
it generates terms of the form, 
\begin{align}\label{eq-Switch1}
d Y_\beta D^{p_1-1}_{X_n} (S_{n \gamma}^m  \cdot Y_m D^{p-p_1}_{X_n} v).
\end{align}
By \eqref{eq-LieBr}, $Y_m$ cannot be $X_n$ in \eqref{eq-Switch} and \eqref{eq-Switch1}.

Since both $m$ and $\gamma$ are not $n$ in \eqref{eq-Switch},  by \eqref{eq-TanAna01} and the fact that
\begin{align*}
d D^{p_1-1-l}_{X_n}  Y_m D^{p-p_1}_{X_n} Y_\gamma v
\end{align*}
  is an entry of $\widehat{D^{p-1-l}_{X_n} \xi_v}$, we get
\begin{align}
& d\| D^{p_1-1}_{X_n} (S_{n \beta}^m  \cdot Y_m D^{p-p_1}_{X_n} Y_\gamma v)\|_{C^\alpha_g(B(P, \eta))}\nonumber\\
&\leq d\sum_{l=0}^{p_1-1}  {p_1-1 \choose l} 
\| (D_{X_n}^l S^m_{n\beta})(D^{p_1-1-l}_{X_n}  Y_m D^{p-p_1}_{X_n} Y_\gamma v)\|_{C^\alpha_g(B(P, \eta))} \nonumber\\
&\leq \sum_{l=0}^{p_1-1}  {p_1-1 \choose l} (DB^{(l-2)^+} (l-2) !)(DB^{(p-3-l)^+} (p-3-l)! )d(P)^{1+\epsilon}\nonumber\\
&\leq \sum_{l=0}^{p_1-1}  {p_1-1 \choose l}  (l-2) ! (p-3-l)! D^2 B^{p-3}d(P)^{1+\epsilon} \label{eq-CombK}\\
&= \sum_{l=0}^{p_1-1}  {p_1-1 \choose l}  (l-2) ! (p_1-3-l)! \frac{(p-3-l)!}{(p_1-3-l)!} D^2 B^{p-3}d(P)^{1+\epsilon}.\nonumber
\end{align}
For $0\leq l \leq p_1-1$,
\begin{align*}
\frac{(p-3-l)!}{(p_1-3-l)!} \leq \frac{(p-3)!}{(p_1-3)!}.
\end{align*}
Lemma \ref{lem-Composition2} implies
\begin{align*}
\sum_{l=0}^{p_1-1}  {p_1-1 \choose l}  (l-2) ! (p_1-3-l)!\leq C(p_1-3)!.
\end{align*}
 Combining all these, we obtain
\begin{align*}
\sum_{p_1=1}^{p}\|d D^{p_1-1}_{X_n} (S_{n \beta}^m  \cdot Y_m D^{p-p_1}_{X_n} Y_2 v)\|_{C^\alpha_g(B(P, \eta))}
\leq C B^{p-3}(p-2)! d(P)^{1+\epsilon}.
\end{align*}
This finishes the estimate of terms in the form of \eqref{eq-Switch}. 

To estimate \eqref{eq-Switch1}, we have to consider two types of terms: 
applying $Y_\beta$ to $S_{n \gamma}^m$, or to $Y_m D^{p-p_1}_{X_n} v$. Then, 
\begin{align*}
&  d\|Y_\beta D^{p_1-1}_{X_n} (S_{n \gamma}^m  \cdot Y_m D^{p-p_1}_{X_n} v)\|_{C^\alpha_g(B(P, \eta))}\\
&\leq d\sum_{l=0}^{p_1-1}  {p_1-1 \choose l} 
\|(Y_\beta D_{X_n}^l S^m_{n\gamma})(D^{p_1-1-l}_{X_n}  Y_m D^{p-p_1}_{X_n}  v)\|_{C^\alpha_g(B(P, \eta))}\\
&\qquad\qquad +d\sum_{l=0}^{p_1-1}  {p_1-1 \choose l} 
\| (D_{X_n}^l S^m_{n\gamma})(Y_\beta D^{p_1-1-l}_{X_n}  Y_m D^{p-p_1}_{X_n} v)\|_{C^\alpha_g(B(P, \eta))},
\end{align*}
where the second term is bounded by \eqref{eq-CombK} as above.
The estimate of the first term is slightly diffferent. By  \eqref{eq-TanAna01},
\begin{align*}
& d\sum_{l=0}^{p_1-1}  {p_1-1 \choose l}  
\| (Y_\alpha D_{X_n}^l S^m_{n\beta})(D^{p_1-1-l}_{X_n}  Y_m D^{p-p_1}_{X_n}  v) \|_{L^\infty(B(P, \eta))}\\
&\leq \sum_{l=0}^{p_1-1}  {p_1-1 \choose l} (DB^{(l-2)^+} (l-2) !)(DB^{(p-3-l)^+} (p-3-l)! )d(P)^{1+\epsilon}\\
&\leq \sum_{l=0}^{p_1-1}  {p_1-1 \choose l}  (l-2) ! (p-3-l)! D^2 B^{p-3}d(P)^{1+\epsilon},
\end{align*}
which is the same as \eqref{eq-CombK}. So we complete the proof of the the lemma.
\end{proof}

Again we apply the maximum principle with   test function 
$$M(x)=bd^{1+\epsilon}\quad\text{in }\Omega_r$$ 
to \eqref{eq-Hp-Gp}
and obtain
$$|D^p_{X_n} v|\leq CB^{p-3} (p-2)!   d^{1+\epsilon}.$$
For any $P$ such that $0<d(P)<r$, we apply the Schauder estimates to \eqref{eq-Hp-Gp} to derive
\begin{align}\label{eq-D-Xn-p}
\| D^p_{X^n}  v\|_{C^{2, \alpha}_g(B(P, \frac{1}{2}\eta)))} \leq CB^{p-3}(p-2)! d(P)^{1+\epsilon}.
\end{align}

\smallskip

{\it Case 2: In $D^p_{X'}$ there are exactly $(p-1)$ $X_n$'s.} 
We first consider $D^p_{X'}= X_\alpha D^{p-1}_{X_n} $, for some $\alpha\in I_2$.

Switching from
$d Y_\alpha D^{p-1}_{X_n}   (Y_\beta Y_\gamma  v)$ 
to $d   Y_\beta Y_\gamma  (Y_\alpha D^{p-1}_{X_n}     v)$, 
it generates terms of the form, for $p_1+p-p_1=p$ and $p_1\geq 2$,
\begin{align}
&dY_\alpha D^{p_1-2}_{X_n} (S_{n\beta}^m  \cdot Y_m D^{p-p_1}_{X_n} Y_\gamma v),\label{eq-Ya-p-1}\\
&d Y_\beta Y_\alpha D^{p_1-2}_{X_n} (S_{n \gamma}^m  \cdot Y_m D^{p-p_1}_{X_n} v)\label{eq-Ya-p-2},
\end{align}
and two terms of  form 
\begin{align}\label{eq-DYnp}
dS_{ \alpha \beta}^i Y_i D^{p-1}_{X_n} Y_\gamma v, \quad dY_\beta (S_{ \alpha \gamma}^i Y_i  D^{p-1}_{Y_n} v),
\end{align}
generated by switching $Y_\alpha$ with $Y_\beta$ or $Y_\gamma$.
\eqref{eq-Ya-p-1}, \eqref{eq-Ya-p-2} can be estimated as in the proof of Lemma \ref{lem-SwitchOrder}. For term in in \eqref{eq-DYnp}, when $i$ in the summation is not $n$, their $C^\alpha_g(B(P,\eta))$ norms are bounded by $CB^{p-3}(p-3)!$ due to \eqref{eq-TanAna01} with $l=p-1$. In the case $i = n$ in \eqref{eq-DYnp}, the terms
\begin{align}
dS_{ \alpha \beta}^n D^{p}_{X_n} Y_\gamma v, \quad dY_\beta (S_{ \alpha \gamma}^n D^{p}_{Y_n} v),
\end{align}
are bounded by \eqref{eq-D-Xn-p} and \eqref{eq-TanAna01}. For example,
\begin{align*}
dS_{ \alpha \beta}^n D^{p}_{X_n} Y_\gamma v = S_{ \alpha \beta}^n \left(d Y_\gamma  D^{p}_{X_n} v  + \sum_{p_1=1}^{p} D_{X_n}^{p_1-1} (S_{n\gamma}^m Y_m D_{X_n}^{p-p_1}v)\right),
\end{align*}
whose $C^\alpha_g(B(P,\eta))$ norm is bounded by
\begin{align*}
C\left(CB^{p-3}(p-2)! d(P)^{1+\epsilon} + \sum_{p_1=1}^{p} \sum_{l=0}^{p_1-1}{p_1-1 \choose l}   CB^{p-3} (l-2) !(p-3-l)! )d(P)^{1+\epsilon}\right),
\end{align*}
which is bounded by $CB^{p-3}(p-2)!$ for a larger $C$ by the  estimate of  \eqref{eq-CombK} in Lemma \ref{lem-SwitchOrder}. In sum,
Lemma \ref{lem-SwitchOrder} still holds if we replace $D^p_{X_n}$ by $X_\alpha D^p_{X_n}$. Then, we can apply the maximum principle and Schauder estimates to derive
\begin{align*}
\|X_\alpha D^{p-1}_{X_n} v\|_{C^{2, \alpha}_g(B(P, \frac{1}{2}\eta)))} \leq CB^{p-3}(p-2)! d(P)^{1+\epsilon}.
\end{align*}

If  $D^p_{X'}= X_n X_\gamma D^{p-2}_{X_n}, \cdots,  D^{p-1}_{X_n}  X_\gamma $, 
then we switch $X_\gamma$ in $D_{X'}^p v$ to the left of all $X_n$'s to derive $X_\gamma D_{X_n}^{p-1}v$. It generates  at most $p-1$ additional terms, where each term is a $(p-1)$-th derivative of $v$ and  is estimated by \eqref{eq-TanAna01}.
Hence we  have
\begin{align}\label{eq-Dp-Xn-2}
\|D^p_{X'} v\|_{C^{2, \alpha}_g(B(P, \frac{1}{2}\eta)))} \leq CB^{p-3}(p-2)! d(P)^{1+\epsilon},
\end{align}
if exactly one of the $X'$'s is $X_n.$

\smallskip

{\it Case 3:  there are at most $(p-2)$ $X_n$'s in $D^p_{X'} $.}
Note that all $(p+1)$-th derivatives generated by switching  orders of the vector fields 
from  $dD^p_{X^\prime} Y_\beta Y_\gamma v$ to $d Y_\beta Y_\gamma D^p_{X^\prime} v $
have at least 2 non-$X_n$ vector fields, which can be estimated similarly as estimating \eqref{eq-Switch}, \eqref{eq-Switch1} in the proof of Lemma \ref{lem-SwitchOrder}.
Then, we can apply the maximum principle and Schauder estimates to derive
\begin{align}\label{eq-Dp-Xn-3}
\|D^p_{X'} v\|_{C^{2, \alpha}_g(B(P, \frac{1}{2}\eta)))} \leq CB^{p-3}(p-2)! d(P)^{1+\epsilon},
\end{align}
if there are at most $(p-2)$ $X_n$'s in $D^p_{X'} $.

\medskip
Finally,  \eqref{eq-D-Xn-p}, \eqref{eq-Dp-Xn-2}, \eqref{eq-Dp-Xn-3} imply 
 \eqref{eq-TanAna} and \eqref{eq-TanAna01}  for case $l=p$. In fact, for an entry $w$ in  $\widehat{D^p_{X'} \xi_v}$, 
we can always switch orders of differentiation to get terms in the form of  $\xi_{D^p_{X'} v}$, where $X'$ may change, with  at most $4n p$ additional  $(p+1)$-th derivatives of $v$. We never switch two vector fields if they are  neither  $X_n$ nor $X_{2n}$.  Take 
 \begin{align*}
 \widehat{D^p_{X'} \xi_v} = d D_{X_n}^{p-1} X_\alpha (Y_\beta Y_\gamma v)
\end{align*}
for example, 
where $\alpha\in J_2, \beta, \gamma \in I_2.$
We switch it to $d  X_\alpha Y_\beta (D_{X_n}^{p-1} Y_\gamma v)$, then apply $X_\alpha = \sum_{\theta \in I_2} c_{\alpha\theta} Y_\theta$ in $B(P, \eta)$ to get $2n-2$ terms in the form of $\xi_{D^p_{X'} v}.$
In general, \eqref{eq-LieBr} implies that all generated $(p+1)$-th derivatives of $v$ are in the form of $\xi_{D^{p-1}_{X'}v}$. Hence \eqref{eq-D-Xn-p}, \eqref{eq-Dp-Xn-2}, \eqref{eq-Dp-Xn-3} imply that
 \begin{align*}
\|  \widehat{D^{p}_{X'} \xi_v}\|_{C^\alpha_g(B(P, \eta))} \leq CB^{p-3} (p-2)! d(P)^{1+\epsilon}.
 \end{align*}
which implies \eqref{eq-TanAna01} since $C$ is independent of $B, p$. 
\end{proof}

Denote $d = \frac{1}{2}t^2$.
 Theorem \ref{Thm-TanAnaly} implies that,  for probably  bigger  $D$ and $B$, 
under local principal coordinates $\{y', d\}$, 
\begin{align} 
\label{eq-Euc-wp-ana}
\|D^p_{y'}v_{tt}\|_{C^\alpha (\Gamma_Q\times [0, R))},
\|t^{-1} {D^p_{y'}v_{t}}\|_{C^\alpha (\Gamma_Q\times [0, R))}, 
\|t^{-2}{D^p_{y'}v}\|_{C^\alpha(\Gamma_Q\times [0, R))}\leq DB^{p-2}(p-2)!.
\end{align}
Consider the series of $v$ with logarithmic terms given by
 \begin{align*}
v(y', t) = c_3t^3 +\cdots + c_{2n+1}t^{2n+1} + c_{2n+2, 1} t^{2n+2}\log t +c_{2n+2, 0}t^{2n+2}+\cdots.
\end{align*} 
 Then the analyticity of the local coefficients $$c_3(y'), c_4(y'), \cdots, c_{2n+1}(y'), c_{2n+2, 1}(y')$$ follows directly from the assumption that $\partial \Omega$ is analytic. These coefficients are determined by the formal computation in Section \ref{sec-FormalComputation}.
The analyticity of
  the first nonlocal coefficient $c_{2n+2,0}(y')$ follows from  \eqref{eq-Euc-wp-ana}.
See Section 3 of \cite{HanJiang2}.  
  
  Set
\begin{align*}
Z = (v, v_t, tv_{tt}, D_{y'}v, t D_{y'}v_t)^T.
\end{align*}  
Then \eqref{eq-v-ODE} can be transformed to a Fuschsian equation of form
\begin{align*}
t\partial_t Z + A Z = t H(y', t, Z, D_{y'}Z),
\end{align*}
where $A$ is a constant matrix.
Techniques in \cite{KL:1}, \cite{KL:2} show that when
\begin{align*}
c_3, \cdots, c_{2n+1}, c_{2n+2, 1}, c_{2n+2, 0},
\end{align*}
are analytic, then $v$ is analytic in $y', t, t\log t$ in $\Gamma_Q\times [0, R)$. It follows from \eqref{eq-Ni} that $v$ is analyt.ic in $y', t, t^{2n+2}\log t$ in $\Gamma_Q\times [0, R)$, which concludes Theorem  \ref{thm-GlobalConvergence}.

\section{The Local Case and the Gevrey Space}\label{Sec-Local}

In this section, we discuss the behavior of $v$ if only an  portion of 
$\partial \Omega$ is analytic. We will prove that  $v$ belongs to 
the Gevrey space of order 2 in tangential directions, 
which consists of functions $v$ such that, under the principal coordinates $y', d,$
$$\|D^p_{y^\prime} v\|_{L^\infty}\leq DB^p (p!)^2,$$
for some constants $D, B$ independent of $p.$

Define 
$$G_R=
\{y=(y^\prime, d) : |y^\prime|<R, 0< d <R\}.$$ 
Use the local frame defined in Lemma \ref{lemma-new-frame}.
Assume that $R$ is small enough so that $G_R \subset U_Q$ for some $Q \in \partial \Omega$. Denote $Y_i =Y_i^Q$ for $i \in I$.
Assume that in $G_R$, for $i, j\in I$, 
\begin{align}\label{eq-Yij}
[Y_i, Y_j]=S_{ij}^m Y_m, 
\end{align}
where $S_{ij}^m$'s are smooth functions. The summation over $m$ is for $m\in I.$
By Lemma \ref{lem-TX}, we have for $\alpha \in I_2$,
\begin{align}\label{eq-LieBr-1}
S_{n \alpha}^n=0.
\end{align}
If $\bar G_R \cap \{d=0\}$ is analytic, there are positive constants $D$ and $ B$ such that, for any $k \geq 0$, any $i,j,m$ such that $ 1\leq i, j, m \leq 2n-1$, 
\begin{align}\label{eq-AnaS-Y}
\|D_{Y'}^k S_{ij}^m\|_{C^\alpha (G_R)}\leq DB^{(k-3)^+} (k-3)!.
\end{align}

\begin{theorem}
\label{Thm-LocTanAnaly}
Let $v$ be a solution  to  \eqref{eq-MA1} in $G_{2R}$.
Assume that  $\bar G_{2R} \cap \{ d=0\}$ is analytic and  \eqref{eq-TanEstm-wp} holds for all $P\in G_R.$
Then there exists $D, B>0$ such that for any $p \in \mathbb{N}$,
\begin{align}\label{eq-TanEsmNoN}
 \delta_p^{(p-2)^+} |\xi_{D_{Y^\prime}^p v}| &\leq   DB^{(p-2)^+} (p-2)! d^{1+\epsilon}\tilde{d}^{-(p-2)^+}, \text{ in } G_R,
\end{align}
where $\tilde{d}(z) = R - |y'(z)|$, $\delta_p(z)={\tilde d}(z)/{p}$, 
 and $D_{Y^\prime}^p v$ denotes any $p$-th order derivative of $v$ with respect to $Y_1, \cdots, Y_{2n-1}$.
\end{theorem}

\begin{proof}

The proof is similar to that of Theorem \ref{Thm-TanAnaly}. The main difference  is that we have to deal with the domain shrinking  for interior estimates.
By the interior analyticity for uniformly elliptic equations, for any fixed $r\in (0, R)$, there are some constants $D, B$ such that
\eqref{eq-TanEsmNoN} holds in $\{z\in G_R : d(z)>r\}$. 
For $l \in \mathbb{N}$, set 
\begin{align}\label{eq-Tl}
T_l=\{z\in G_R: 0< d(z)<\delta_l^2 =\tilde{d}(z)^2/l^2 \}.
\end{align}

Define  $\widehat{D^{l}_{Y'} \xi_v}$  as in the proof of Theorem \ref{Thm-TanAnaly}. 
We prove inductively that
\begin{align}\label{eq-TanAna02}
|  \widehat{D^{l}_{Y'} \xi_v}| \leq D(B \delta_l^{-1} \tilde d^{-1})^{(l-2)^+} (l-2)! d^{1+\epsilon} \text{ in } G_R. 
\end{align}
For  $l\leq 3$,  \eqref{eq-TanAna02} follows from \eqref{eq-TanEstm-wp} by setting $D$ sufficiently large. 
Assume that  \eqref{eq-TanAna02} holds for $l=0,1, \cdots, p-1$. We prove case $p$ in three steps.

{\it Step 1.} We show that
 $$|D^p_{Y^\prime} v|\leq \delta_p^{-p+2} DB^{p-2} (p-2)! d^{1+\epsilon} \tilde{d}^{-(p-2)} \text{ in }G_R.$$ 
In Step 1, we exclude the special case that exactly one of the $Y'$'s in $D_{Y'}^p v$ is $Y_n$ and will discuss it in Step 2.

First, we consider at  $P\in T_p$, which means that
\begin{align*}
0<d(P) < \tilde d(P)^2 / p^2.
\end{align*}
Set
\begin{align*}
G_{P, \delta, \delta^2} =\{y: |y^\prime- y^\prime(P)|<\delta, 0\leq d< \delta^2\}.
\end{align*}
Take
\begin{align}\label{eq-delta}
\delta={\tilde{d}(P)}/{p}.
\end{align}
 Then $G_{P, \delta, \delta^2}$ contains $P$ and is contained in $T_{p-1}$. In fact, if $z \in G_{P, \delta, \delta^2}$, then
\begin{align*}
|y'(z)- y'(P)|< \delta, 0< d(z) < \delta^2,
\end{align*} 
 which implies that
\begin{align*}
|\tilde d (z)  - \tilde d(P)| \leq  |y'(z)- y'(P)| < \delta = \tilde d(P)/p.
\end{align*} 
Hence  for $p> 3$,
\begin{align*}
d(z) < \delta^2 = \tilde d(P)^2/p^2 \leq (\tilde d(P) - |\tilde d (z) - \tilde d(P)|)^2/(p-1)^2 <\tilde d(z)^2/(p-1)^2,
\end{align*}
 which implies that $z \in T_{p-1}$.
 
Set  test function
\begin{align}\label{eq-M-Ana}
M(y', d)=b_1|y^\prime - y'(P)|^2 d^{ \epsilon}+b_2d^{1+\epsilon}.
\end{align}
Then, 
\begin{align*}
M(y^\prime, \delta^2)&=b_1 |y^\prime- y'(P)|^2 \delta^{2\epsilon} +b_2\delta^{2+2\epsilon} ,\\
M(y^\prime, d)&=b_1\delta^2 d^{\epsilon} +b_2d^{1+\epsilon} \text{ if } |y^\prime- y'(P)|=\delta.
\end{align*}
For  any point $y= (y',d) \in G_{P,\delta, \delta^2}$, $\tilde{d}(y) \geq \tilde{d}(P)-\delta.$ 
By \eqref{eq-TanEsmNoN} for $l=p-1$,
\begin{align}\label{eq-Dpv-1}
|D^p_{Y^\prime} v| \leq  \delta^{-p+3}D B^{p-3}(p-3)!d^{\epsilon} (\tilde{d}(P)-\delta)^{-(p-3)}.
\end{align} 

If
\begin{align}
b_1&\geq  \delta^{-p+1} D B^{p-3}(p-3)!   (\tilde{d}(P)-{\delta})^{-(p-3)},\label{eq-b1}\\
b_2&\geq  \delta^{-p+1} D B^{p-3}(p-3)!  (\tilde{d}(P)-{\delta})^{-(p-3)},\label{eq-b2}
\end{align}
then,  $|D_{Y'}^p v| \leq M$ on $\partial G_{P, \delta, \delta^2}$. Notice
\begin{align}\label{eq-tilde-d-p}
 (\tilde{d}(P)-\delta)^{-(p-3)}
=\left(\tilde{d}(P)-\tilde{d}(P)/{p}\right)^{-(p-3)} \leq C \tilde{d}(P)^{-(p-3)},
\end{align}
where $C$ is some universal constant.  Applying  \eqref{eq-delta},
\eqref{eq-b1}-\eqref{eq-b2} hold if
\begin{align*}
 b_1&\geq C \delta^{-p+2}  B^{p-3}(p-2)!  \tilde{d}(P)^{-(p-2)},\\
 b_2&\geq C  \delta^{-p+2} B^{p-3}(p-2)!  \tilde{d}(P)^{-(p-2)},
\end{align*}
for some  constant $C$ independent of $B, p$.

Now $v$ satisfies equation \eqref{eq-Hp-Gp} in $G_{P, \delta, \delta^2}$ with $D_{X'}^p$ replaced by $D_{Y'}^p $.
Applying \eqref{eq-TanAna02} and Lemmas \ref{lemma-Composition}, \ref{lem-Composition2}, we have that in $G_{P, \delta, \delta^2}$,
\begin{align*}
|\tilde H_p | \leq C\delta^{-p+3}  B^{p-3}(p-2)! d(P)^{1+\epsilon}\tilde d(P)^{-p+3}.
\end{align*}

To bound $\tilde G_p$ in \eqref{eq-TanAna02},
 we need a lemma similar to Lemma \ref{lem-SwitchOrder}.

 \begin{lemma}\label{lem-SwitchOrder1}
Assume that $p\geq 4$ and \eqref{eq-AnaS-Y}, \eqref{eq-TanAna02} hold
for $1\leq l \leq p-1$ in $G_R$.
Then, 
\begin{align}\label{eq-Switch-p-factorial-Y}
|  D_{Y^n}^p \xi_{ v} -\xi_{D_{Y^n}^p v}|
\leq C (B\delta^{-1} \tilde d^{-1})^{p-3}(p-2)!d^{1+\epsilon},
\end{align}
in $G_R$,
where $C$ is independent of $B, p$.
\end{lemma}

The proof is similar to the proof of Lemma \ref{eq-Switch-p-factorial}. The difference is that we replace $B$ by $B\delta^{-1} \tilde d^{-1}$ and replace the $C^\alpha_g(B(P,\eta))$ norm by the absolute value $|\cdot |$.

By the maximum principle and similar computation to \eqref{eq-MaxP}, when $D^p_{Y^\prime} v =D^p_{Y_n} v$, $D^p_{Y^\prime} v \leq M$ in $G_{P, \delta, \delta^2}$, which further implies that
\begin{align}
|D^p_{Y^n} v|\leq \delta^{-(p-2)}  CB^{p-3} (p-2)! d^{1+\epsilon} \tilde{d}^{-p+2}
\end{align}
at any $P \in T_p.$ 

For the case
$
D_{Y'}^p v = Y_\alpha D_{Y'}^{p-1} v
$,
with $\alpha \in I_2$, \eqref{eq-Dpv-1} can be improved to be
\begin{align*}
|D^p_{Y^\prime} v| \leq  \delta^{-p+3}D B^{p-3}(p-3)!d^{\frac{1}{2}+\epsilon} (\tilde{d}(P)-\delta)^{-(p-3)}.
\end{align*}
In the case that there is at least one $Y'$ in $Y_\alpha D_{Y'}^{p-1} v$
is not $Y_n$, we  apply the maximum principle with test function
\begin{align}\label{eq-M-Ana-1}
M(y) = b_1 |y' - y'(P)|^2 d^{\frac{1}{2}+\epsilon} + b_2 d^{1+\epsilon}
\end{align}
to show that
\begin{align}\label{eq-Dpv-max}
|D_{Y'}^{p} v|\leq \delta^{-(p-3)}  CB^{p-3} (p-2)! d^{1+\epsilon} \tilde{d}^{-p+2}.
\end{align}
The computation is similar to \eqref{eq-MaxP}. By exchanging the orders of differentiation if necessary, \eqref{eq-Dpv-max} is  true as long as at least two $Y'$'s in $D_{Y'}^{p} v$ are not $Y_n$.

Secondly, we consider $z\in T_{p}^C \cap G_R$. By \eqref{eq-TanAna02} for $l=p-2$, 
if  $D^p_{Y^\prime} v=D_{Y_n}^2 D^{p-2}_{Y^\prime} v$, then
\begin{align*}
|D^p_{Y^\prime} v|&\leq d^{-2} \delta^{-(p-4)} DB^{p-4}(p-4)! d^{1+\epsilon} \tilde{d}^{-p+4}\\
&\leq \delta^{-(p-2)} CB^{p-4} (p-2)! d^{1+\epsilon} \tilde{d}^{-(p-2)},
\end{align*}
since $d\geq \delta^2$. Similarly, if $D^p_{Y^\prime} v=Y_n Y_\alpha D^{p-2}_{Y^\prime} v, Y_\alpha Y_n D^{p-2}_{Y^\prime} v$ or $Y_\alpha Y_\beta D^{p-2}_{Y^\prime} v$, for some $\alpha \in I_2$, we have
\begin{align}\label{eq-TpC-Linf}
|D^p_{Y^\prime} v| \leq  \delta^{-(p-3)} CB^{(p-4)} (p-2)! d^{1+\epsilon} \tilde{d}^{-(p-2)}.
\end{align} 

In sum, we have that in $G_R$,
\begin{align}\label{eq-Dpv-Linf}
|D^p_{Y^n} v| \leq  \delta^{-(p-2)} CB^{(p-3)} (p-2)! d^{1+\epsilon} \tilde{d}^{-(p-2)},
\end{align}
and if at least two $Y'$'s in $D^p_{Y^\prime} v$ do not equal $Y_n$,
\begin{align}\label{eq-Dpv-Linf-spe}
|D^p_{Y^\prime} v| \leq  \delta^{-(p-3)} CB^{(p-3)} (p-2)! d^{1+\epsilon} \tilde{d}^{-(p-2)}.
\end{align}

\smallskip

{\it Step 2.} Let $P \in T_p$.
We observe that  for any $z \in B(P, \eta)$, $\tilde d (z) \geq (1-1/p)\tilde{d}(P)$, which implies that
\begin{align*}
\tilde d (z)^{-(p-2)} \leq C\tilde{d}(P)^{-(p-2)}.
\end{align*}
By the interior $C^{1, \alpha}$ estimates for uniformly elliptic equations, for  any $P \in T_p$.
\begin{align}\label{eq-C1a-Dpv}
\|D^p_{Y^\prime} v\|_{C^{1,\alpha}_g(B(P, \frac{1}{2}\eta))} 
&\leq C (\|D^p_{Y^\prime} v\|_{L^\infty(B(P, \eta))}+\|\tilde H_p + \tilde G_p\|_{L^\infty(B(P, \eta))}).
\end{align}

Then by \eqref{eq-Dpv-Linf},
\begin{align}\label{eq-C1a1}
\|D^p_{Y^n} v\|_{C_g^{1,\alpha}(B(P, \frac{1}{2}\eta))}
\leq \delta^{-p+2} CB^{p-3}(p-2)!d(P)^{1+\epsilon}\tilde{d}(P)^{-p+2}.
\end{align}
We applied \eqref{eq-tilde-d-p} here.
\eqref{eq-C1a1} also holds if we replace $p$ by  $l < p$. If at least two $Y'$'s in $D_{Y'}^p v$ are not $Y_n$, \eqref{eq-Dpv-Linf-spe}, \eqref{eq-C1a-Dpv}  imply that
\begin{align}\label{eq-C1a2}
\|D^p_{Y^\prime} v\|_{C_g^{1,\alpha}(B(P, \frac{1}{2}\eta))}
\leq \delta^{-p+3} CB^{p-3}(p-2)!d(P)^{1+\epsilon}\tilde{d}(P)^{-p+2}.
\end{align}

Now consider the special case that exactly one $Y'$ in $D^p_{Y'}$ is $Y_n$. Our goal is to show that \eqref{eq-Dpv-Linf-spe} and \eqref{eq-C1a2}  hold in this case.
Take $D^p_{Y'} v = Y_\alpha D^{p-1}_{Y^n} v$, $\alpha\in I_2,$ for example. Otherwise we can exchange the orders of the differentiation.  The problem in estimating $|Y_\alpha D^{p-1}_{Y^n} v|$ in $T_p$ is that while exchanging the orders of differentiation, terms like
\begin{align}\label{eq-Bad}
dY_\gamma D^p_{Y_n} v, \,\,d^\frac{3}{2}  D^{p+1}_{Y_n} v,\,\, d^\frac{3}{2} ( D^p_{Y_n} v)_d, 
\end{align} 
appear in $\tilde G_p.$  \eqref{eq-C1a1} is still correct with $B(P,\frac{1}{2}\eta)$ replaced by $B(P, \eta)$, for probably a large $C$. Thus we have
\begin{align}\begin{split}\label{eq-Spe-Ca}
&\|dY_\gamma D^p_{Y_n} v\|_{C_g^{\alpha}(B(P, \eta))} +\|d^\frac{3}{2}  D^{p+1}_{Y_n} v\|_{C_g^{\alpha}(B(P, \eta))}+\|d^\frac{3}{2} ( D^p_{Y_n} v)_d\|_{C_g^{\alpha}(B(P, \eta))}\\
&\leq \sqrt{d(P)}\|D^p_{Y^n} v\|_{C_g^{1,\alpha}(B(P, \eta))}\\
&\leq \delta^{-p+3}CB^{p-3}(p-2)!d(P)^{1+\epsilon}\tilde{d}(P)^{-p+2}.
\end{split}
\end{align}
The set of all balls $B(P, \eta)$ for $P \in T_p$ covers $T_{p-1}$.
In fact, for any $z \in T_{p-1}$, we take the point $P'$ that satisfies $y' (P') =y' (z)$ and $ d(P')= (1-10^{-10})(\tilde d(z)/p)^2$. Then $P' \in T_p$ and $z \in B(P', \eta).$ 
Based on this, when $P\in T_p,$   for any $z\in G_{P, \delta, \delta^2} \subseteq T_{p-1}$,
\begin{align*}
&d|Y_\gamma D^p_{Y_n} v(z) | +d^\frac{3}{2}|  D^{p+1}_{Y_n} v(z) |+d^\frac{3}{2}| ( D^p_{Y_n} v)_d(z) |\\
&
\leq \delta^{-p+3}CB^{p-3}(p-2)!d(z)^{1+\epsilon}\tilde{d}(P)^{-p+2}.
\end{align*}
Then terms in \eqref{eq-Bad} are bounded.
The maximum principle and  \eqref{eq-TpC-Linf} imply that \eqref{eq-Dpv-Linf-spe} holds for the special case.
Applying the interior $C^{1, \alpha}_g$ estimate, we derive that
\eqref{eq-C1a2} also holds when exactly one of the $Y'$'s are not $Y_n.$ Thus we completed the discussion of the special case.

To track the constants appeared in the estimates, we denote the larger $C$ in \eqref{eq-C1a1}, \eqref{eq-C1a2} as $C_1$.
Inductively, we  prove that for $l\leq p-1$ and $P \in T_p$,
\begin{align}\label{eq-induct-a}
\|D^l_{Y^\prime} v\|_{C_g^{2,\alpha}(B(P,  \frac{1}{2}\eta))}
\leq \delta^{-l+2} C_0  B^{l-2}(l-2)!d(P)^{1+\epsilon}\tilde{d}(P)^{-l+2},
\end{align}
where $C_0$ is independent of $B, p.$
Take $l=p-1$ for an  illustration.  
For $w \in \xi_{D^{p-1}_{Y^\prime} v}$, and $w \neq d^2 (D^{p-1}_{Y^\prime} v)_{dd}$, it follows directly from \eqref{eq-C1a1}, \eqref{eq-C1a2} that
\begin{align}\label{eq-w-Ca-ana}
\|w\|_{C_g^{\alpha}(B(P,  \frac{1}{2}\eta))}
\leq \delta^{-p+3} C_2 C_1  B^{p-3}(p-3)!d(P)^{1+\epsilon}\tilde{d}(P)^{-p+3}.
\end{align}
For example, applying \eqref{eq-C1a1}, \eqref{eq-C1a2},
\begin{align*}
&\|d^2 Y_n (D_{Y'}^{p-1}v)_d\|_{C_g^{\alpha}(B(P,  \frac{1}{2}\eta))} \\
&= \|d^2 (Y_n D_{Y'}^{p-1}v)_d\|_{C_g^{\alpha}(B(P,  \frac{1}{2}\eta))} + \sum_{\beta\in I_1}  \|d^2 f_\beta Y_\beta D_{Y'}^{p-1}v\|_{C_g^{\alpha}(B(P,  \frac{1}{2}\eta))}\\
&\leq \delta^{-p+3} C_2 C_1  B^{p-3}(p-3)!d(P)^{1+\epsilon}\tilde{d}(P)^{-p+3}.
\end{align*}
For $(D_{Y'}^{p-1} v)_{dd}$, 
we consider the equation \eqref{eq-Hp-Gp} with $p$ replaced by $p-1$ and write it as
\begin{align}\begin{split}\label{eq-p-1-ana}
a_{dd}( D_{Y'}^{p-1} v)_{dd} &= \left(a_{dd}( D_{Y'}^{p-1} v)_{dd}- \overrightarrow{a} \cdot \xi_{D_{Y'}^{p-1} v }\right)\\
&\qquad+ (n+1)D_{Y'}^{p-1} v  +\tilde H_{p-1}  +\tilde G_{p-1}.
\end{split}
\end{align} 
Then we estimate $C^\alpha_g(B(P,\frac{1}{2}\eta))$ norm of the right-hand side by \eqref{eq-w-Ca-ana} and \eqref{eq-induct-a}.
In fact, the $C^\alpha_g(B(P, \frac{1}{2}\eta))$ norm of $\left(a_{dd}( D_{Y'}^{p-1} v)_{dd}-   \overrightarrow{a} \cdot \xi_{D_{Y'}^{p-1} v }\right) + (n+1)D_{Y'}^{p-1} v$ is bounded by 
\begin{align*}
\delta^{-p+3} C_3 C_2 C_1  B^{p-3}(p-3)!d(P)^{1+\epsilon}\tilde{d}(P)^{-p+3}.
\end{align*}
The $C^\alpha_g(B(P, \frac{1}{2}\eta))$ norm of $\tilde H_{p-1}  +\tilde G_{p-1}$ is bounded by
\begin{align}\label{eq-G-p-1}
\delta^{-p+3} C_4  B^{p-4}(p-3)!d(P)^{1+\epsilon}\tilde{d}(P)^{-p+3}.
\end{align}
Here $C_4$ depends on $C_0$. We take
\begin{align*}
C_0 = \sup_{z\in B(P, \frac{1}{2}\eta)} \left\lbrace d(z)^2 a^{-1}_{dd}(z)\right\rbrace \cdot (C_3C_2C_1 +1)
\end{align*}
 and set $B>C_4$ so that in \eqref{eq-G-p-1}, $C_4  B^{p-4} \leq B^{p-3}.$
This completes the proof of \eqref{eq-induct-a}. 
\eqref{eq-induct-a}  implies that
\begin{align}\label{eq-G-H-ana}
\|\tilde H_p +\tilde G_p\|_{C_g^\alpha (B(P, \frac{1}{2}\eta))} \leq \delta^{-p+2} CB^{p-3}(p-2)!d(P)^{1+\epsilon}\tilde{d}(P)^{-p+2}.
\end{align}
For the special case that exactly one $Y'$ in $D^p_{Y'}$ equals $Y_n$, we also need  \eqref{eq-Spe-Ca} to estimate $\tilde G_p$.
In obtaining \eqref{eq-G-p-1} and \eqref{eq-G-H-ana}, we encounter the issue of switching orders of the differentiation. This can be resolved by inductively proving
\begin{align*}
\|\widehat{D^{l}_{Y'} \xi_v}\|_{C_g^{2,\alpha}(B(P,  \frac{1}{4}\eta))}
\leq \delta^{-l+2}C B^{l-2}(l-2)!d(P)^{1+\epsilon}\tilde{d}(P)^{-l+2}.
\end{align*}

By the Schauder estimates, for probably a larger $C$,
\begin{align}\label{eq-C2a-Tp}
\|D^p_{Y^\prime} v\|_{C_g^{2,\alpha}(B(P, \frac{1}{4}\eta))}
\leq \delta^{-p+2}CB^{p-3}(p-2)!d(P)^{1+\epsilon}\tilde{d}(P)^{-p+2}.
\end{align}
Therefore,  \eqref{eq-TanEsmNoN} holds for all $P\in T_p$.

\smallskip

{\it Step 3.} We now prove \eqref{eq-TanEsmNoN} for  $l=p$ in $T_p^C \cap G_R$. 
Define
\begin{align}\label{eq-norm-w}
\|w\|_{\varrho}=\sup\{\left(d^{-1-\epsilon}|w| \right)(z): z \in T_{p}^C \cap G_R, \tilde{d}(z)\geq \varrho\}.
\end{align}
Set
\begin{align}\label{eq-def-wpk}
w_{p,k}= D^{p-k}_{Y^\prime}D^k_{Y_n} v.
\end{align}
Here  we assume that none of the $Y'$'s equals $Y_n$ or $Y_{2n}$.

We claim 
\begin{align}
\label{eq-Step3-claim}
&\max_{0\leq k\leq p}\| \delta_p^{p-2} \widehat{\xi_{w_{p,k}}}\|_\varrho \leq \frac{1}{10}\max_{0\leq k\leq p}\|\delta_p^{p-2}  \widehat{\xi_{w_{p,k}}}\|_{(1-\frac{1}{p+1})\varrho}
+CB^{p-3} (p-2)! \varrho^{-p+2}.
\end{align}
We need the following result from \cite{Friedman1958}.

\begin{lemma}
\label{lem-LocTanAna}
Let $g(\varrho)$ be a positive monotone decreasing function, defined in the interval $0\leq \varrho \leq 1$ and satisfying 
\begin{align*}
g(\varrho)\leq \frac{1}{10}g\left(\varrho\left(1-\frac{1}{n}\right)\right)+\frac{C}{\varrho^{n-3}},\end{align*}
for some $n\geq 4$ and $ C>0$. 
Then, 
$$g(\varrho)< CA\varrho^{-n+3},$$ 
for some constant $A$ independent of $C, n, \varrho$.
\end{lemma}

By applying Lemma \ref{lem-LocTanAna} to \eqref{eq-Step3-claim}, we obtain
\begin{align}\label{eq-DelKWpk}
\max_{0\leq k\leq p}\|\delta_p^{p-2}  \widehat{\xi_{w_{p,k}}}\|_\varrho 
\leq CB^{p-3} (p-2)! \varrho^{-p+2},
\end{align}
which is the desired estimate to complete the induction.

Fix a $\varrho>0$. For any $P\in T_p^C \cap G_R$ with $\tilde{d}(P)\geq\varrho$, set
\begin{equation}\label{eq-xi}
 \tau = \frac{\delta}{N\sqrt{d(P)}},
\end{equation} 
where $N$ is a constant to be determined, which does not depend on $B, p$.
As $\delta\leq \sqrt{d(P)}$,  
$\tau \leq \frac{1}{N}.$

Replacing $p$ by $p+1$ and $X'$ by $Y'$ in  \eqref{eq-Hp-Gp} and multiplying it by $\tau^2$, we derive
\begin{align}\label{eq-Hp-Gp+1}
\tau^2 \overrightarrow{a} \cdot \xi_{D_{Y'}^{p+1} v}- (n+1) \tau^2 D_{Y'}^{p+1} v =\tau^2 \tilde H_{p+1}  +\tau^2 \tilde G_{p+1}.
\end{align} 
This equation is uniformly elliptic under  the coordinate $\tau^{-1}\vartheta$. The metric ball  with radius $\eta$ centered at $P$ under the coordinate $\tau^{-1}\vartheta$ is $B(P, \tau\eta)$ under $\vartheta$, which is contained in $G_R.$
 Denote the $C^{1, \alpha}$ estimates  under coordinates $\tau^{-1}\vartheta$ as $C^{1,\alpha}_{g/\tau^2}$.
By the interior $C^{1, \alpha}_{g/\tau^2}$ estimates for \eqref{eq-Hp-Gp+1}, we derive
\begin{align}\begin{split}\label{eq-w-p+1}
&\|w_{p+1,k}\|_{C_{g/\tau^2}^{1,\alpha}(B(P, \frac{1}{2}\tau\eta))}\\
&\leq C ( \|w_{p+1,k}\|_{L^{\infty}(B(P, \tau\eta))}
+ \tau^2\|\tilde H_{p+1} + \tilde G_{p+1}\|_{L^{\infty}(B(P, \tau \eta))}).
\end{split}
\end{align}

In the rest of this proof, we always denote $Y'$ as a non-$Y_n$ vector.
For 
\begin{align}\label{eq-wpk-k}
w_{p, k}=Y_n w_{p-1, k-1}, 1\leq k\leq p,
\end{align}
we set $ w_{p+1,k} = Y' w_{p, k}$.
Applying  \eqref{eq-TanAna02} for  $l=p-1$ and the fact  that $\delta\leq \sqrt{d}$, we have that in $G_R,$
\begin{align}\label{eq-estimate-p+1-k}
 |w_{p+1,k}|\leq  \delta_p^{-p+3} DB^{p-3}  (p-3)! d^{-\frac{1}{2}+\epsilon} \tilde{d}^{-(p-3)}.
\end{align}
By \eqref{eq-w-p+1}, \eqref{eq-estimate-p+1-k},
\begin{align*}
&\|w_{p+1,k}\|_{C_{g/\tau^2}^{1,\alpha}(B(P, \frac{1}{2}\tau\eta))}\\
&\leq  \delta^{-p+3}C B^{p-3} (p-3)!d(P)^{- \frac{1}{2}+\epsilon}  \tilde{d}(P)^{-(p-3)}\\
&\qquad+ C \tau^2 \|\tilde H_{p+1} +\tilde G_{p+1}\|_{L_{g/\tau^2}^{\infty}(B(P,\tau\eta))}.
\end{align*}
When $1\leq k\leq p-1$, all of the $(p+2)$-th derivatives in $\tilde H_{p+1} + \tilde G_{p+1}$ are in the form of  $\widehat{\xi_{ w_{p, k-1}}}, \widehat{\xi_{w_{p, k}}}, \widehat{\xi_{ w_{p, k+1}}} $
and there are at most $C p$ such terms, for some constant $C$ that only depends on $n$.  The rest terms in $\tilde H_{p+1} + \tilde G_{p+1}$ are bounded by induction. 
Here $\widehat{\xi_{ w_{p, k+1}}}$ appears as exchanging the orders of $Y_\alpha$ and $Y_\beta$, for $\alpha, \beta \in I_2,$ may generate $Y_n$. For example, when $k=p-1$, 
\begin{align*}
Y' Y_n^{p-1} Y_\alpha ( d Y_\beta Y_\gamma v) =Y' Y_n^{p-1} Y_\beta ( d Y_\alpha  Y_\gamma v) + d  Y' Y_n^{p-1} (S_{\alpha\beta}^n Y_n)   Y_\gamma v + \cdots,
\end{align*}
from which we obtain a term 
$d  Y' Y_n^{p}  Y_\gamma v \in \widehat{\xi_{w_{p,p}}}=\widehat{\xi_{w_{p,k+1}}}$.
Thus
\begin{align}\begin{split}\label{eq-w-p+1-k}
&\|w_{p+1,k}\|_{C_{g/\tau^2}^{1,\alpha}(B(P, \frac{1}{2} \tau \eta))}\\
&\leq  \delta^{-p+3}  C B^{p-3} (p-3)!d(P)^{-\frac{1}{2}+\epsilon} \tilde{d}(P)^{-(p-3)} + C p \tau^2 \|\widehat{\xi_{w_{p,k-1}}}\|_{L^\infty(B(P, \tau\eta))}\\
&\qquad+ Cp \tau^2\|\widehat{\xi_{w_{p,k}}}\|_{L^\infty(B(P, \tau\eta))}+ C p \tau^2\|\widehat{\xi_{w_{p,k+1}}}\|_{L^\infty(B(P, \tau\eta))}.
\end{split}
\end{align}
Multiplied by $\tau^{-1} \sqrt{d(P)}$,
\eqref{eq-w-p+1-k}, \eqref{eq-xi}  imply that at $P$,
\begin{align}
\begin{split}\label{eq-Ynprime}
&d^{\frac{3}{2}} |Y_n Y^\prime w_{p,k}| + d^{\frac{3}{2}} |( Y^\prime w_{p,k})_d| +
d |D^2_{Y^\prime}w_{p,k}|\\
&\leq  \delta^{-p+2} C B^{p-3} (p-2)! d^{1+\epsilon} \tilde{d}^{-(p-2)}+ C \tilde d \|\widehat{\xi_{w_{p,k-1}}}\|_{L^\infty(B(P, \tau\eta))}\\
&\qquad 
+ C \tilde d\| \widehat{\xi_{w_{p,k}}}\|_{L^\infty(B(P, \tau\eta))}+ C \tilde d\| \widehat{\xi_{w_{p,k+1}}}\|_{L^\infty(B(P, \tau\eta))}.
\end{split}
\end{align}
Here we applied $p^{-1} \tilde d(P) = \delta\leq d(P)^\frac{1}{2}$.

When $k=p$ in \eqref{eq-wpk-k}, there are terms in $\tilde G_{p+1}$ that are  not included in $\widehat{\xi_{w_{p, l}}}$ for any $l.$
In fact,   switching the orders of differentiation from $Y'Y_n^p (\xi_v)$ to  $\xi_{Y'w_{p, p}}$ will result in extra terms in the form of
\begin{align}\label{eq-except}
d Y_n D_{Y_\alpha} D_{Y_n}^p  v,\,\, d^\frac{3}{2}  D^{p+2}_{Y_n} v,\,\, d^\frac{3}{2} ( D^{p+1}_{Y_n} v)_d, 
\end{align}
where $\alpha \in I_2$. They are entries in $d^{-\frac{1}{2}}\widehat{\xi_{w_{p, p}}}$ and there are at most $C$ such terms, where $C$ only depends on $n$. By making $C$ larger, similar to \eqref{eq-w-p+1-k}, we have
\begin{align}
\begin{split}\label{eq-w-p+1-p}
&\|w_{p+1,p}\|_{C_{g/\tau^2}^{1,\alpha}(B(P, \frac{1}{2} \tau \eta))}\\
&\leq  \delta^{-p+3}  C B^{p-3} (p-3)!d(P)^{-\frac{1}{2}+\epsilon} \tilde{d}(P)^{-(p-3)} + C p \tau^2 \|\widehat{\xi_{w_{p,p-1}}}\|_{L^\infty(B(P, \tau\eta))}\\
&\qquad+ C(p+ d(P)^{-\frac{1}{2}}) \tau^2\|\widehat{\xi_{w_{p,p}}}\|_{L^\infty(B(P, \tau\eta))},
\end{split}
\end{align}
which implies that at $P$,
\begin{align}\begin{split}\label{eq-Ynprime-2}
& d^\frac{3}{2} |Y_n Y' w_{p,p}| + d^\frac{3}{2}|( Y' w_{p,p})_d|+d |D_{Y'}^2 w_{p,p}| \\
& \leq \delta^{-p+2}  C B^{p-3} (p-2)!d^{1+\epsilon} \tilde{d}^{-(p-2)}\\
&\qquad + C \tilde d \|\widehat{\xi_{w_{p,p-1}}}\|_{L^\infty(B(P, \tau\eta))}+ C (\tilde d+\tau)\|\widehat{\xi_{w_{p,p}}}\|_{L^\infty(B(P, \tau\eta))}.
\end{split}
\end{align}

Back to \eqref{eq-wpk-k}, we set
\begin{align}
w_{p+1, k+1}=Y_n w_{p, k}.
\end{align}
Same arguments as above show that at $P$, for $1\leq k \leq p,$
\begin{align}\label{eq-Ynprime-3}
\begin{split}
&d^2 |Y_n^2 w_{p,k}| + d^2 |( Y_n w_{p,k})_d| + d^\frac{3}{2}| Y' Y_n w_{p,k}| \\
& \leq \delta^{-p+2}  C B^{p-3} (p-2)!d^{1+\epsilon} \tilde{d}^{-(p-2)}+ C (\tilde d+\tau) \|\widehat{\xi_{w_{p,k}}}\|_{L^\infty(B(P, \tau\eta))}\\
&\qquad + C (\tilde d+\tau) \|\widehat{\xi_{w_{p,k+1}}}\|_{L^\infty(B(P, \tau\eta))}+  C (\tilde d+\tau)  \|\widehat{\xi_{w_{p,k+2}}}\|_{L^\infty(B(P, \tau\eta))}.
\end{split}
\end{align}
Here we set $w_{p,k}=0$ if $k<0$ or $k>p.$
So for $1\leq k\leq p$, \eqref{eq-Ynprime}, \eqref{eq-Ynprime-2} and \eqref{eq-Ynprime-3} indicate estimates for all entries in $\xi_{w_{p,k}}$ but $d^2 ( w_{p,k})_{dd}$, which is to be estimated using   \eqref{eq-Hp-Gp}. 
Thus at $P$,  for $1\leq k\leq p$,
\begin{align}\begin{split}\label{eq-Wpk}
|\xi_{w_{p,k}}|
&\leq \delta^{-p+2} CB^{p-3} (p-2)! d^{1+\epsilon} \tilde{d}^{-(p-2)}\\
&\qquad + C (\tilde d+\tau) \|\widehat{\xi_{w_{p,k-1}}}\|_{L^\infty(B(P, \tau\eta))}
+ C (\tilde d+\tau) \| \widehat{\xi_{w_{p,k}}}\|_{L^\infty(B(P, \tau\eta))}\\
&\qquad+ C (\tilde d+\tau) \| \widehat{\xi_{w_{p,k+1}}}\|_{L^\infty(B(P, \tau\eta))}+ C (\tilde d+\tau) \|\widehat{ \xi_{w_{p,k+2}}}\|_{L^\infty(B(P, \tau\eta))}
\end{split}
\end{align}
\eqref{eq-Wpk} is also true when $k=0.$ In fact, we can denote $w_{p,k}= Y' w_{p-1, k}$, for $0\leq k \leq p-1$. Setting $w_{p+1, k} = Y' w_{p, k}$ and $w_{p+1, k+1} = Y_n w_{p+1, k+1}$, similar arguments as above imply \eqref{eq-Wpk} for case $k=0$.


Finally, we multiply \eqref{eq-Wpk}  by $\delta^{p-2}$ to  obtain that at $P,$
\begin{align*}
\max_{0\leq k\leq p} ( \delta^{p-2}|\xi_{w_{p,k}}|)
&\leq CB^{p-3} (p-2)! d^{1+\epsilon} \varrho^{-(p-2)}\\
 &\qquad
 +C(\tilde{d}+\tau)\max_{0\leq k\leq p} \delta^{p-2} \|\widehat{\xi_{w_{p,k}}}\|_{L^\infty(B(P, \tau\eta))},
\end{align*}
since $\tilde{d}(P)\geq\varrho$.

For any point $z\in B(P, \tau\eta)$,  $\tilde d(z) \geq (1 - \frac{1}{p+2}) \tilde d(P).$
At points in $B(P, \tau\eta)$  but not in $T_p^C$, 
$\xi_{w_{p,k}}$ is already estimated in Step 2.   As  $\tilde d(P)\leq  R$, $\tau \leq \frac{1}{N}$,
\begin{align}\begin{split}\label{eq-SumWPK}
\max_{0\leq k\leq p}\| \delta^{p-2}\widehat{ \xi_{w_{p,k}}}\|_{\varrho} 
&\leq \frac{1}{10}\max_{0\leq k\leq p}\|\delta^{p-2}\widehat{\xi_{w_{p,k}}}\|_{(1-\frac{1}{p+2})\varrho}+CB^{p-3} (p-2)! \varrho^{-p+2},
\end{split}
\end{align}
provided that $C (R+\frac{1}{N})\leq {1}/{10}$, which holds if we set $N = 20C$ and assume  $C R\leq {1}/{20}$. If $R$ is large, we can divide $G_R$ into small cylinders and estimate on each small cylinder. 
Here we ignored the first derivatives of $w_{p,k}$ in $\xi_{w_{p,k}}$, since  $d^\frac{1}{2}|Y^\prime w_{p, k}|$, $d |Y_n w_{p, k}|$, and $d |(
w_{p,k})_d|$ can be estimated directly by the induction. For example,
\begin{align*}
\max_{0\leq k\leq p}\| \delta_p^{p-2} d(w_{p,k})_d\|_\varrho 
&\leq \max_{0\leq k\leq p-1}  \|  d^{-1} \delta_p^{p-2} \widehat{\xi_{w_{p-1,k}}}\|_\varrho\\
&\leq \varrho^{-1} p  \max_{0\leq k\leq p-1}  \| \delta_p^{p-3} \widehat{\xi_{w_{p-1,k}}}\|_\varrho\\
&\leq DB^{p-3} (p-2)! \varrho^{-p+2}.
\end{align*}

 Applying Lemma \ref{lem-LocTanAna}, \eqref{eq-TanAna02} can be derived similarly as in the proof of Theorem \ref{Thm-TanAnaly}.
\end{proof}

 In the domain 
$G_R\cap \{\tilde{d}(y)\geq \varrho\}$, \eqref{eq-TanEsmNoN} implies that
\begin{align*}
 |\xi_{D_{Y^\prime}^l v}|
 &\leq   l^{l-2} DB^{l-2} (l-2)! d^{1+\epsilon} \tilde{d}^{-2(l-2)}\\
    &\leq   D (\varrho^{-2}B)^{l-2} (2l)! d^{1+\epsilon}. 
\end{align*}
Note $(2l)!< (2\cdot 4\cdot 6\cdots (2l))^2= 2^{2l} (l!)^2$. 
Hence, 
\begin{align}\label{eq-estimate-Gevrey}
 |\xi_{D_{Y^\prime}^l v}|
    &<  C (4\varrho^{-2}B)^{l-2} (l!)^2 d^{1+\epsilon}.
\end{align}
Therefore,  $v$ is in the Gevrey space of order 2 along the tangential directions. 

\appendix

 \section{Analyticity Type Estimates}
 \label{sec-Ana-estm}
 If $p<0$, we assume that $p!=0.$
The following result is essentially due to Friedman \cite{Friedman1958}. Also see \cite{HanJiang2}.
\begin{lemma}\label{lemma-Composition}
Let $\Omega$ be a domain in $\mathbb R^n$ and $p$ be a positive integer. Assume that $\Phi$ 
is a $C^p$-function in $ \Omega\times\mathbb R^N$ satisfying, for any 
$(x,y)\in \Omega\times\mathbb R^N$ and any nonnegative integers $j$ and $k$ 
with $j+k\le p$, 
\begin{equation}\label{eq-Composition1}
\Big|\frac{\partial^{j+k}\Phi(x,y)}{\partial x^j\partial y^k}\Big|
\le A_0A_1^jA_2^k(j-2)!(k-2)!,\end{equation}
for some positive constants $A_0$, $A_1$ and $A_2$. Then, there exist 
positive constants $B_0$, $\widetilde B_0$ and $B_1$, 
depending only on $n$, $N$, $A_0$, $A_1$ and $A_2$, such that, 
for any $C^p$-function $y=(y_1, \cdots, y_N): \Omega\to\mathbb R^N$, if 
for any $x\in \Omega$ and any nonnegative integer $k\le p$, 
\begin{equation}\label{eq-Composition2}
\sum_{i=1}^N|\partial^k_xy_i(x)|\le B_0B_1^{(k-2)^+}(k-2)!,\end{equation}
then, for any $x\in\Omega$, 
\begin{equation}\label{eq-Composition3}
|\partial_x^p[\Phi(x,y(x))]|\le \widetilde B_0B_1^{(p-2)^+}(p-2)!.\end{equation}
\end{lemma}

\begin{remark}\label{rmk-Composition} Write $x=(x',x_n)$. In Lemma \ref{lemma-Composition}, 
if we assume \eqref{eq-Composition1} and 
\eqref{eq-Composition2} hold only for $D_{x'}$ instead of $D_x$, then 
\eqref{eq-Composition3} holds for $D_{x'}$.\end{remark}

\begin{lemma}\label{lem-Composition2}
Let $p$ be any integer and $p\geq 2$. Then there is a universal constant $C$ such that
\begin{align} \label{eq-p-fact-1}
 \sum_{l=1}^{p-1} {p-1 \choose l}(l-2)! \cdot (p-2-l)! \leq C  (p-2)!,
\end{align}
and 
\begin{align}\label{eq-p-fact-2}
\sum_{l=0}^{p-1}  {p-1 \choose l}  (l-2) ! (p-3-l)!\leq C(p-3)!.
\end{align}
\end{lemma}

\begin{proof}
When $l=p-1$,
$${p-1 \choose l}(l-2)! \cdot (p-l-2)!  = (p-3)!.$$
And
\begin{align*}
  \sum_{l=1}^{p-2} {p-1 \choose l}(l-2)! \cdot (p-2-l)!&\leq 
\sum_{l=2}^{p-2} \frac{p-1}{l(l-1)(p-l-1)}(p-2)!
\\
&\leq C(p-2)!,
\end{align*}
as
\begin{align*}
\sum_{l=2}^{p-2} \frac{p-1}{l(l-1)(p-l-1)}\
&=\sum_{l=2}^{p-2} \frac{1}{l-1}\left(\frac{1}{l}+\frac{1}{p-l-1}\right)\\
&=\sum_{l=2}^{p-2} \frac{1}{(l-1)l}+\frac{1}{p-2}\sum_{l=2}^{p-2}\left( \frac{1}{l-1}+\frac{1}{p-l-1}\right),
\end{align*}
which is bounded by a universal constant $C$. Notice that $\sum_{l=2}^{p-2} \frac{1}{l-2} < 1+\ln p$ and $\frac{1+\ln p}{p-2}$ is bounded by a universal constant when $p \geq 3$.
Then we derive \eqref{eq-p-fact-1}.

For \eqref{eq-p-fact-2}, first we observe that for cases $l=0, 1, p-1$ or $p-2$, 
$$ {p-1 \choose l}  (l-2) ! (p-3-l)!\leq C(p-3)!.$$
And
\begin{align*}
&\sum_{l=2}^{p-3}  {p-1 \choose l}  (l-2) ! (p-3-l)!\\
&= (p-3)! \sum_{l=2}^{p-3}  \frac{(p-1)(p-2)}{l(l-1)(p-1-l)(p-2-l)} \\
&\leq 2(p-3)! \sum_{l=2}^{p-3}  \left(\frac{1}{l} + \frac{1}{p-2-l}\right)\left(\frac{1}{l-1} + \frac{1}{p-1-l}\right),
\end{align*}
where similarly, the summation over $l$ is bounded by a universal constant. Then we derive \eqref{eq-p-fact-2}.
\end{proof}


\begin{thebibliography}{DG}



\bibitem{BG1972}M.S. Baouendi, C. Goulaouic, \emph{Nonanalytic-Hypoellipticity for some degenerate elliptic operators}, Bulletin of the American Mathematical Society, volume 78, number 3, May 1972.


\bibitem{ChengYau1980CPAM} S.-Y. Cheng, S.-T. Yau, 
\emph{On the existence of a complete K\"ahler metric on non-compact complex 
manifolds and the regularity of Fefferman's equation}, 
Comm. Pure Appl. Math., 33(1980), 507-544. 

\bibitem{Chrusciel2005} 
P. Chru\'sciel, E. Delay, J. Lee, D. Skinner, 
\emph{Boundary regularity of conformally compact 
Einstein metrics}, J. Diff. Geom., 69(2005), 
111-136.

\bibitem{Fefferman1976} C. Fefferman, 
\emph{Monge-Amp\`ere equation, the Bergman kernel, and geometry of pseudoconvex domains}, 
Ann. Math., 103(1976), 395-416. 

\bibitem{Fefferman&Graham2002} C. Fefferman, C. R. Graham, 
\emph{$Q$-curvature and Poincar\'e metrics}, Math. Res. Lett., 9(2002), 139-151. 

\bibitem{Fefferman&Graham2012} C. Fefferman, C. R. Graham, 
\emph{The Ambient Metric}, Annals of Mathematics Studies, 178, Princeton University Press, 
Princeton, 2012.

\bibitem{Friedman1958} A. Friedman,
{\it On the regularity of the solutions of nonlinear Elliptic and Parabolic systems of partial differential equations},
Journal of mathematics and mechanics, Vol 7, No.1 (1958) 43-59.

\bibitem{GT} D. Gilbarg, N. Trudinger,
{\it Elliptic Partial Differential Equations of Elliptic Type},
Springer, Berlin, 1983.

\bibitem{Graham&Witten1999} C. R. Graham, E. Witten, 
\emph{Conformal anomaly of submanifold observables in AdS/CFT correspondence}, 
Nuclear Physics B, 546(1999), 52-64. 

\bibitem{HanKhuri2014CPDE} Q. Han, M. Khuri, {\it Existence and blow-up behavior for solutions
of the generalized Jang equation}, Comm. P.D.E., 38(2013), 2199-2237. 

\bibitem{HanJiang} Q. Han, X. Jiang, {\it Boundary regularity for minimal graphs in the hyperbolic space}, 	arxiv:1412.7608

\bibitem{HanJiang2} Q. Han, X. Jiang, {\it The convergence of boundary expansions and the analyticity of minimal surfaces in the hyperbolic space}, arxiv:1801.08348.

\bibitem{HanWang} Q. Han, Z. Wang, \emph{Solutions of  the minimal surface equation and 
of the Monge-Amp\`{e}re equation near infinity}, preprint. 

\bibitem{Hardt&Lin1987} R. Hardt, F.-H. Lin, \emph{Regularity at infinity 
for area-minimizing hypersurfaces in hyperbolic space}, Invent. Math., 88(1987), 217-224. 


\bibitem{Hellimell2008} D. Helliwell, 
\emph{Boundary regularity for conformally compact Einstein metrics in 
even dimensions},
Comm. P.D.E., 33(2008),  
842-880.




\bibitem{JiangXiao} X. Jiang, L. Xiao, {\it Optimal regularity of constant curvature graphs in Hyperbolic space}, Calc. Var. P.D.E., 58:133 (2019).

\bibitem{KL:1} S. Kichenassamy, W. Littman,
{\it Blow-up Surfaces for Nonlinear Wave Equations, Part I},
Commun. in P. D. E., 18 (3\&4) 431-452 (1993).


\bibitem{KL:2} S. Kichenassamy, W. Littman,
{\it Blow-up Surfaces for Nonlinear Wave Equations, Part II},
Commun. in P. D. E., 18 (11) 1869-1899 (1993). 


\bibitem{LeeMelrose} J. Lee, R. Melrose, 
\emph{Boundary behavior of the complex Monge-Amp\`ere equation}, 
Acta Math., 148(1982), 159-192. 



\bibitem{N:1} L. Nirenberg,
{\it An abstract form of the nonlinear Cauchy-Kowalewski Theorem},
J. Differential Geometry, 6(1972) 561-576.



\bibitem{WangXD} Xiaodong Wang, 
\emph{Some recent results in CR geometry}, Tsinghua lectures in mathematics, 469-484,
Adv. Lect. Math. , 45 (2019).




\bibitem{Yau1978} S. T. Yau, \emph{On the Ricci curvature of a compact K\"ahler manifold and the complex Monge-Amp\`ere equation I}, Comm. Pure Appl. Math., 31 (1978), 339-411.


\end{thebibliography}
\end{document}